\documentclass[english,a4paper,12pt,reqno]{amsart}
\usepackage{tgtermes}
\usepackage[T1]{fontenc}
\usepackage{textcomp}
\usepackage[utf8]{inputenc}
\usepackage{xcolor}
\usepackage{babel}
\usepackage{verbatim}
\usepackage{float}
\usepackage{calc}
\usepackage{mathrsfs}
\usepackage{mathtools}
\usepackage{amstext}
\usepackage{amsthm}
\usepackage{amssymb}
\usepackage{graphicx}
\usepackage{esint}
\usepackage[bookmarks=true,bookmarksnumbered=true,bookmarksopen=true,bookmarksopenlevel=3,
 breaklinks=true,pdfborder={0 0 0},pdfborderstyle={},backref=false,colorlinks=true]
 {hyperref}
\hypersetup{
 bookmarksdepth=3,linkcolor=red,citecolor=blue}

\makeatletter
\numberwithin{equation}{section}
\numberwithin{figure}{section}

\usepackage{tikz-cd}
\usepackage[twoside,inner=1.3in,outer=1in,top=1in,bottom=1in]{geometry}
\usepackage{csquotes}
\DeclareRobustCommand{\skipsubsectionsinTOC}{\addtocontents{toc}{\protect\setcounter{tocdepth}{1}}}
\DeclareRobustCommand{\showsubsectionsinTOC}{\addtocontents{toc}{\protect\setcounter{tocdepth}{3}}} 

\makeatother

\theoremstyle{plain}
\newtheorem{thm}{\protect\theoremname}
\newtheorem{cor}[thm]{\protect\corollaryname}
\newtheorem{prop}[thm]{\protect\propositionname}
\theoremstyle{remark}
\newtheorem{rem}[thm]{\protect\remarkname}
\theoremstyle{definition}
\newtheorem{defn}[thm]{\protect\definitionname}
\newtheorem{example}[thm]{\protect\examplename}
\theoremstyle{plain}
\newtheorem{lem}[thm]{\protect\lemmaname}
\usepackage[style=alphabetic,backref=true, isbn=false, url=false, doi=false, backend=biber]{biblatex}
\providecommand{\corollaryname}{Corollary}
\providecommand{\definitionname}{Definition}
\providecommand{\examplename}{Example}
\providecommand{\lemmaname}{Lemma}
\providecommand{\propositionname}{Proposition}
\providecommand{\remarkname}{Remark}
\providecommand{\theoremname}{Theorem}

\begin{document}
\title{K-unstable Toric Varieties and Secondary Polytopes}
\author{Chi Li}
\address{%
\begin{tabular}[t]{@{}l@{}}
Department of Mathematics, Rutgers University, Piscataway, NJ 08854, U.S.A.\\
Department of Mathematics, Johns Hopkins University, Baltimore, MD 21218, U.S.A.
\end{tabular}}
\email{chi.li@rutgers.edu, cli261@jh.edu}
\author{$\;$Li Sheng}
\address{School of Mathematics, Sichuan University, Chengdu 610064, P. R. China.}
\email{lshengscu@gmail.com}
\author{$\;$Yi Yao}
\address{School of Mathematics, Hunan University, Changsha 410082, P. R. China.}
\email{yeeyoe@163.com}
\subjclass[2020]{32Q15, 14M25, 14L24, 32Pxx}
\begin{abstract}
For K-unstable toric varieties, Székelyhidi's optimal test-function
$\Theta_{P}$ is a mysterious concave function over the moment polytope
$P$, which gives the maximal destabilizer for K-stability, and encodes
the limiting behavior of the divergent Calabi flow. As balanced norms
quantize cscK metrics, we show that $\Theta_{P}$ can be quantized
by the maximal destabilizers for Chow-stability, which are given by
the shortest GKZ vectors, i.e., the least norm point on the secondary
polytope for the set $P\cap k^{-1}\mathbb{Z}^{n}$. Properties and
algorithm for general sGKZ vectors are given. Our result may provide
a new method to detect the K-unstability of toric varieties. 
\end{abstract}

\maketitle
\global\long\def\lam{\lambda}%
 
\global\long\def\vphi{\varphi}%
 
\global\long\def\cch{\mathcal{H}}%
 
\global\long\def\bbc{\mathbb{C}}%
 
\global\long\def\bbcs{\mathbb{C}^{*}}%
 
\global\long\def\bbr{\mathbb{R}}%
 
\global\long\def\bbz{\mathbb{Z}}%
 
\global\long\def\bbn{\mathbb{N}}%

\global\long\def\ra{\rightarrow}%
 
\global\long\def\lead{\leadsto}%
 
\global\long\def\'{\prime}%
 
\global\long\def\cleq{\preceq}%
\global\long\def\cgeq{\succeq}%
 
\global\long\def\na{{\scriptscriptstyle \mathsf{NA}}}%
 
\global\long\def\an{\mathrm{an}}%
 
\global\long\def\ma{\mathrm{MA}}%
 
\global\long\def\deff{\coloneqq}%

\tableofcontents{}

\section{Introduction}

\skipsubsectionsinTOC

Consider a seemingly innocent variational problem. Let $P$ be a full-dimensional
polytope in $\bbr^{n}$, one wants to minimize the functional 
\begin{equation}
\Psi(\eta)=\frac{1}{2}\int_{P}\eta^{2}\mathrm{d}x+\int_{\partial P}\eta\ \mathrm{d}\sigma,\ \eta:P\rightarrow\bbr\text{ is convex},\label{eq: begin Phi}
\end{equation}
where $\mathrm{d}x$ is Lebesgue measure and boundary measure $\mathrm{d}\sigma$
is proportional to the Lebesgue measure on each facet. If one decreases
the boundary integral, convexity will force the $L^{2}$-norm to increase.
Indeed, $\Psi$ admits a unique minimizer, which is affine in most
situations. It remains unknown whether the minimizer is always piecewise
affine. 

This variational problem originates from Székelyhidi \cite{szekelyhidi_optimal_2008}.
Its minimizer relates to the long-time behavior of the Calabi flow
on the toric Kähler manifold associated to $P$. We discuss this context
below.

Let $(X,L)$ be the polarized toric Kähler manifold induced by a lattice
Delzant polytope $P\subset M_{\bbr}\coloneqq M\otimes\bbr$, where
$M$ is a lattice of rank $n$. As an instance of Yau--Tian--Donaldson
conjecture, the existence of extremal metrics in class $c_{1}(L)$
is expected to be equivalent to a stability condition on $P$, see
\cite{donaldson_scalar_2002,szekelyhidi_introduction_2014}. With
the aid of torus symmetry, the existence can be reduced to the solvability
of a fourth-order PDE on $P$. The stability condition is called the
relative K-polystability, which requires that the relative Donaldson--Futaki
invariant
\[
\mathbb{M}^{\mathrm{rel}}_{P}(\eta)\coloneqq\int_{\partial P}\eta\ \mathrm{d}\sigma-\int_{P}\eta\ell_{P}\ \mathrm{d}x\geq0
\]
for all rational piecewise affine convex functions $\eta$, and equality
holds only for affine $\eta$. See §\ref{subsec:Szekelyhidi's-optimal-test}
for details. 

It is widely accepted that this stability condition needs to be strengthened
to make the equivalence hold. There are two kinds of strengthened
conditions: uniform K-stabilities \cite{dervan_uniform_2016} and
testing more functions \cite{szekelyhidi_filtrations_2015}. To date,
various versions of equivalence have been established, refer to \cites{donaldson_extremal_2008}{chen_extremal_2018}{chen_constant_2021-1}{hisamoto_stability_2020}{li_geodesic_2022}{li_extremal_2023}{jubert_yautiandonaldson_2023}{nitta_uniform_2021}
and their references. Meanwhile, there exist several criteria for
practically verifying stability, see \cite{zhou_relative_2008,yotsutani_relative_2019}. 

\subsection*{\texorpdfstring{Optimal test function $\Theta_{P}$}{Optimal test function Theta\_P}}

The unstable case has received less attention, see \cites{wang_existence_2011}{wang_k-stability_2014}{sektnan_investigation_2018}{yotsutani_relative_2019}{araujo_calabi_2023}{hwang_toric_2024}
for examples. When $P$ is relatively K-unstable (i.e. $\mathbb{M}^{\mathrm{rel}}_{P}$
may be negative), \cite{szekelyhidi_optimal_2008} showed the minimization
problem 
\[
\inf\left\{ \frac{\mathbb{M}^{\mathrm{rel}}_{P}(\eta)}{\left\Vert \eta\right\Vert _{L^2}}\ \Big|\ 0\neq\eta\in\mathcal{C}_{1}(P)\cap L^{2}(P)\right\} 
\]
admits a unique minimizer $\ell_{P}-\Theta_{P}$ up to scaling, see
(\ref{eq:space C1}) for the enlarged space $\mathcal{C}_{1}(P)$
of convex functions. Proposition \ref{prop: recharac Theta} shows
that $-\Theta_{P}$ is actually the unique minimizer of $\Psi$ (\ref{eq: begin Phi})
above. Let $\Theta_{P}=\ell_{P}$ when $P$ is relatively K-semistable
(i.e. $\mathbb{M}^{\mathrm{rel}}_{P}\geq0$). We call this concave
function $\Theta_{P}$ ($=B$ in \cite{szekelyhidi_optimal_2008})
the \emph{optimal test function} for $P$. 

The boundedness of $\Theta_{P}$ is proved by Li, Lian and the second
author \cite{li_extremal_2023}. We find that their proof can yield
stronger results. 
\begin{thm}[Continuity of $\Theta_{P}$, §\ref{proof of Li-Lian-Sheng}]
\label{thm: theta is continuous-INTRO}  The optimal test function
$\Theta_{P}$ can extend to a concave function on the whole space.
In particular, $\Theta_{P}$ is continuous up to the boundary of $P$,
and the subdifferential set $\partial\left(-\Theta_{P}\right)\left(\mathrm{int}P\right)$
is bounded. 
\end{thm}

As \cite{szekelyhidi_optimal_2008} showed, $\Theta_{P}$ encodes
important information of $X$. For instance, if the Calabi flow admits
long-time solution, then the scalar curvature (as function on $P$)
would converge to $\Theta_{P}$. More interestingly, when $\Theta_{P}$
is piecewise affine, it induces a polyhedral decomposition of $P$,
then every subpolytope is semistable in a certain sense, see Theorem
\ref{thm: Gabor paper}. Hence $\Theta_{P}$ is analogous to the Harder--Narasimhan
filtration for the unstable holomorphic vector bundles, where the
monotonicity of the slopes of quotient bundles mimics the concavity
of $\Theta_{P}$. To the authors' knowledge, there is no explicit
example of $\Theta_{P}$ so far, see \cite{sektnan_investigation_2018}
for examples with zero boundary measure. By the way, the optimal test-functions
for Ding-stability of toric Fano varieties have been determined in
\cite{yao_mabuchi_2022}. They are always piecewise affine. 

\subsection*{Maximal destabilizers for K-stability}

Optimal test function $\Theta_{P}$ is a realization of maximal K-destabilizers
for general K-unstable Kähler manifold $(X,L)$. Recall that the constant
scalar curvature Kähler metrics are critical points of K-energy. When
$(X,L)$ is K-unstable, Xia \cite{xia_sharp_2021} show that there
is a unique (modulo parallel relation) finite-energy geodesic ray
along which the K-energy decreases fastest. By \cites{berman_variational_2021}{li_geodesic_2022},
this ray can be encoded into a non-Archimedean metric on $L$, which
we call the \emph{maximal K-destabilizer} of $(X,L)$. In the toric
setting, with the aid of symmetry, it reduces to $\Theta_{P}$, see
\cite[Example 44]{yao_maximal_2025}. 

\subsection*{Maximal destabilizers for Chow-stability}

The third author \cite{yao_maximal_2025} suggests to study the maximal
K-destabilizers via the Chow-stability. Since its birth in \cite{tian_kahler-einstein_1997},
K-stability has been closely connected with Chow-stability. Let $\iota_{k}:X\hookrightarrow\mathbb{P}R^{*}_{k}$
be the Kodaira embedding, where $R_{k}\coloneqq\mathrm{H}^{0}(X,kL)$.
Chow-stability means the GIT stability of Chow-form attached to $\iota_{k}(X)$.
Ross and Thomas \cite{ross_study_2006} showed that asymptotic Chow-semistability
can imply K-semistability. Hence if $(X,L)$ is K-unstable, then $(X,kL)$
is Chow-unstable for sufficiently large $k$. We can consider the
\emph{maximal Chow-destabilizer} for $\iota_{k}(X)$, which is defined
as a filtration on $R_{k}$ induced by the 1-ps of $\mathrm{SL}\left(R_{k}\right)$
minimizing the Chow-weights. 

A natural question is whether the maximal Chow-destabilizer converges
to the maximal K-destabilizer in a certain sense as $k\rightarrow\infty$.
This is similar to the stable case, in which Donaldson \cite{donaldson_scalar_2001}
used balanced metrics to approximate cscK metrics. One difference
is that there is no Euler--Lagrange equation for maximal destabilizers,
only a variational inequality. 

This paper study this problem for toric varieties. The main results
are 

(1) We determine the maximal Chow-destabilizers by a combinatorial
object called the shortest GKZ vector, see Theorem \ref{thm: Chow-desta INTRO}
below. We find two properties for sGKZ vectors: convexity and positivity,
and provide an algorithm to compute them. 

(2) We establish the $L^{2}$-convergence of maximal Chow-destabilizers
to the maximal K-destabilizer, see Theorem \ref{thm:main converge}
below. 

\subsection*{Maximal Chow-destabilizers for toric varieties}

For toric varieties, an explicit formula of Chow-weights (for 1-ps
in the diagonal torus) is obtained in \cite{kapranov_chow_1992},
see also \cite{gelfand_discriminants_1994,ono_algebro-geometric_2013}.
We will derive this formula via the NA expression of Chow-weight obtained
in \cite{yao_maximal_2025}. 

Let $(X^{\an},L^{\an})$ be the Berkovich analytification of $(X,L)$
with the trivially-normed base field $\bbc$. Let $\mathcal{N}(R_{k})$
be the space of NA norms (or filtrations) on $R_{k}$. Then the Chow-weights
for subvariety $\iota_{k}(X)$ are given by functional 
\begin{equation}
M_{k}:\mathcal{N}(R_{k})\ra\bbr,\ M_{k}(\chi)=\mathbb{E}\left(\mathrm{FS}_{k}(\chi)\right)-E_{k}(\chi),\label{eq:Chow w INTRO}
\end{equation}
where $\mathrm{FS}_{k}(\chi)$ (\ref{def FSk}) is the NA Fubini--Study
metric induced by $\chi$, $\mathbb{E}$ is NA Monge--Ampère energy,
and $E_{k}(\chi)$ is the average of jumping numbers (\ref{eq: def Ek}). 

If $\iota_{k}(X)$ is Chow-unstable (i.e. $\inf M_{k}<0$), the maximal
Chow-destabilizer $\chi^{*}_{k}$ is the unique (up to scaling) minimizer
of the functional 
\[
M_{k}(\chi)/\left\Vert \chi\right\Vert _{2},\ \chi\in\mathcal{N}(R_{k})\backslash\{\chi_{tr}\},
\]
see \cite[Theorem 1]{yao_maximal_2025}. Moreover, $\chi^{*}_{k}$
is invariant under all automorphisms of $(X,L)$, that allows us to
search the minimizer only among the invariant norms. 

Apply these facts to toric variety $(X,L)$ given by a lattice polytope
$P\subset M_{\bbr}$. It suffices to consider the space $\mathcal{N}(R_{k})^{T}$
of torus-invariant norm on $R_{k}$. This space can be identified
with 
\[
\bbr^{P_{k}}=\left\{ f:P_{k}\ra\bbr\right\} ,\ P_{k}=P\cap k^{-1}M.
\]
On the other hand, via the tropicalization map, the space $\mathrm{CPSH}^{\na}_{T}(L)$
of invariant continuous psh metrics on $L^{\an}$ can be identified
with the space $\mathcal{C}(P)$ of convex functions on $P$ (continuous
up to $\partial P$). Then the NA Fubini--Study map $\mathrm{FS}_{k}$
and super-norm map $\mathrm{SN}_{k}$ (\ref{def SNk}) are identified
with the convex envelope $\vee$ (\ref{eq:conv envelop}) and restriction
respectively. 
\begin{eqnarray*}
\mathcal{N}(R_{k})^{T} & \xrightleftharpoons[\ \ \ \mathrm{SN}_{k}\ \ \ ]{\mathrm{FS}_{k}} & \mathrm{CPSH}^{\na}_{T}(L)\\
(\ref{eq:identify N W})\ \big\downarrow\simeq &  & \simeq\big\downarrow\ (\ref{eq:identify CPSH to conv})\\
\bbr^{P_{k}} & \xrightleftharpoons[\text{restriction}]{\text{convex envelope }\vee} & \mathcal{C}(P)
\end{eqnarray*}
The Chow-weight $M_{k}$ (\ref{eq:Chow w INTRO}) is identified with
\[
M_{P_{k}}(f)\coloneqq\frac{1}{\left|P_{k}\right|}\sum_{u\in P_{k}}f(u)-\frac{1}{V_{P}}\int_{P}f^{\vee}\ \mathrm{d}x,\ \forall f\in\bbr^{P_{k}},
\]
which is consistent with the formula in \cite{kapranov_chow_1992,ono_algebro-geometric_2013}. 

We analyze $M_{P_{k}}$ via Gelfand--Kapranov--Zelevinsky's secondary
polytope $\Sigma(P_{k})\subset\bbr^{P_{k}}$, which encodes all subdivisions
of $(P,P_{k})$ and the refinement relations between them. Despite
its complicated structure, we are only concerned with the least-norm
point on $\Sigma(P_{k})$, which we call the \textit{shortest GKZ
vector} $\sigma_{k}\in\Sigma(P_{k})$. 
\begin{thm}[determine the maximal Chow-destabilizers, Theorem \ref{thm: when semistable},
\ref{thm: Chow destabi}]
\label{thm: Chow-desta INTRO}

For toric variety $(X,L)$ given by lattice polytope $P\subset M_{\bbr}$,
suppose that $\iota_{k}:X\hookrightarrow\mathbb{P}R^{*}_{k}$ is an
embedding ($k\ge n-1$ is enough). 

(1) Subvariety $\iota_{k}(X)\subset\mathbb{P}R^{*}_{k}$ is Chow-semistable
(w.r.t. the action of $\mathrm{SL}(R_{k})$) if and only if the constant
vector $\left|P_{k}\right|^{-1}\in\Sigma(P_{k})$, or equivalently
$\sigma_{k}\equiv\left|P_{k}\right|^{-1}$. 

(2) When $\iota_{k}(X)$ is Chow-unstable, the minimizers (i.e. maximal
Chow-destabilizers) of 
\begin{equation}
\inf\left\{ \frac{M_{k}(\chi)}{\left\Vert \chi\right\Vert _{2}}\ \Big|\ \chi\in\mathcal{N}(R_{k})\backslash\{\chi_{tr}\}\right\} \label{eq: min Chow  INTRO}
\end{equation}
 are of the form $a\chi^{*}_{k}$ ($a>0$), where $\chi^{*}_{k}$
is diagonalized by the basis $\left\{ s_{u,k}\right\} _{u\in P_{k}}$
(\ref{eq:eigen basis}) and satisfies 
\[
\chi^{*}_{k}(s_{u,k})=\left|P_{k}\right|^{-1}-\sigma_{k}(u),\ \forall u\in P_{k}.
\]
\end{thm}

With respect to the diagonal torus action, the equivalence in (1)
has been obtained in \cite{ono_algebro-geometric_2013}. The symmetry
of maximal destabilizers enables the equivalence to hold also for
$\mathrm{SL}(R_{k})$. 

\subsection*{Shortest GKZ vectors for general point sets }

We have not found any existing studies on the shortest GKZ vectors,
which may be of independent interest for combinatorics. Thus we investigate
it for general point sets that generalize $P_{k}$. 

A marked polytope $(Q,A)$ is a polytope $Q\subset M_{\bbr}$ with
marking point set $A$ such that $Q=\mathrm{conv}(A)$. A triangulation
of $(Q,A)$ gives rise to a GKZ vector $g_{T}:A\ra\bbr$, see (\ref{eq:def GKZ vec}).
The secondary polytope $\Sigma(A)\subset\bbr^{A}$ is the convex hull
of all GKZ vectors which never contains the origin. Its faces (vertices)
are in one-to-one correspondence to the regular subdivisions (triangulations)
of $(Q,A)$. 

Shortest GKZ vector $\sigma_{A}$ is the unique point on $\Sigma(A)$
with the least norm, which may not be a vertex of $\Sigma(A)$. Note
that $\sigma_{A}$ depends only on affine relations among the points
in $A$, not on the background metric. One may think of $\sigma_{A}$
as representing the most uniform subdivision of $(Q,A)$ in some sense. 
\begin{thm}[properties of shortest GKZ vectors, Theorem \ref{thm: convex of mGKZ},
\ref{thm: positive of mGKZ}]
For any marked polytope $(Q,A)$ with shortest GKZ vector $\sigma_{A}\in\Sigma(A)$,
we have 

(1) $\sigma_{A}$ is convex, that means the restriction of the convex
envelope $\sigma^{\vee}_{A}|_{A}=\sigma_{A}$. 

(2) $\sigma_{A}:A\ra\bbr$ is positive at each point. 
\end{thm}

The convexity reminds us the concavity of $\Theta$, see the remark
after \cite[Proposition 9]{szekelyhidi_optimal_2008}. By (2), we
know that if $\sigma_{A}$ is a vertex of $\Sigma(A)$, then the corresponding
triangulation must use all points in $A$. 

\subsection*{Connect the maximal destabilizers for Chow and K-stability}

The problem is reduced to connecting $\Theta_P$
with the shortest GKZ vectors $\sigma_{k}$ attached to $P_{k}$. 

A natural idea is to connect the variational problem behind $\Theta$
and $\sigma_{k}$. In Appendix \ref{Appendix: Riemann sum}, we verify
the following folklore expansion of Riemann sum 
\[
\sum_{u\in P_{k}}\eta(u)=\int_{P}\eta\ \mathrm{d}x\cdot k^{n}+\frac{1}{2}\int_{\partial P}\eta\ \mathrm{d}\sigma\cdot k^{n-1}+\mathrm{o}(k^{n-1}),\ \eta\in\mathcal{C}(P),
\]
which may fail if $\eta$ is merely continuous. Consider a quantum
functional 
\[
\Psi_{k}(\eta)\coloneqq\frac{1}{2k^{n}}\sum_{u\in P_{k}}\eta(u)^{2}+\frac{2}{k^{n-1}}\left(\sum_{u\in P_{k}}\eta(u)-k^{n}\int_{P}\eta\ \mathrm{d}x\right),\ \eta\in\mathcal{C}(P),
\]
which is the toric version of the general functional $\breve{\mathbb{Q}}_{k}$
in \cite{yao_maximal_2025}. We have 
\[
\lim_{k\rightarrow\infty}\Psi_{k}(\eta)=\Psi(\eta),\ \forall\eta\in\mathcal{C}(P).
\]
Theorem \ref{thm: min Psi_k} shows that $\Psi_{k}$ admits a unique
minimizer 
\begin{equation}
\psi_{k}\coloneqq2V_{P}k^{n+1}\sigma^{\vee}_{k}-2k.\label{eq:quantum mini INTRO}
\end{equation}
Recall that $-\Theta_{P}$ is the minimizer of $\Psi$. Our main result
is the following. 
\begin{thm}[Convergence §\ref{subsec:Convergence}]
\label{thm:main converge} (1) When $P$ is a Delzant lattice polytope,
the quantum minimizer $\psi_{k}$ (\ref{eq:quantum mini INTRO}) converges
to $-\Theta_{P}$ in $L^{2}(P)$ and uniformly on every compact set
$K\subset\mathrm{int}P$. In particular, at any point $x\in\mathrm{int}P$,
we have 
\begin{equation}
V_{P}k^{n}\sigma^{\vee}_{k}(x)=1-\frac{\Theta_{P}(x)}{2k}+\mathrm{o}\left(\frac{1}{k}\right).
\end{equation}
Moreover, we have 
\[
\min\Psi_{k}\ra\min\Psi,\ \lim_{k}\mathbb{M}^{\mathrm{rel}}_{P}(\psi_{k})=\mathbb{M}^{\mathrm{rel}}_{P}(-\Theta_{P})=-\left\Vert \Theta_{P}-\ell_{P}\right\Vert ^{2}_{L^{2}}.
\]

(2) For a general lattice polytope $P$, $\psi_{k}$ uniformly converges
to $-\Theta_{P}$ over $P$ if and only if $\{\psi_{k}\}$ is $C^{0}$-precompact,
i.e., uniformly bounded and equicontinuous. 
\end{thm}

Hence for smooth K-unstable toric varieties, the maximal Chow-destabilizers
quantize the maximal K-destabilizers in the above sense. The proof
uses a $\Gamma$-convergence argument, the keystone is the estimate
(\ref{eq:key estimate}), which simplifies a result of Lee and Yotsutani
\cite{lee_asymptotic_2026}. This estimate is the only part relying
on Delzant condition, it may fail for non-Delzant polytopes, see Remark
\ref{rem: fail nonDelzant}. 
\begin{cor}[sGKZ vectors can detect K-unstability]
If $P$ is a relatively K-unstable Delzant polytope (i.e. $\Theta_{P}\neq\ell_{P}$),
then $\mathbb{M}^{\mathrm{rel}}_{P}(\sigma^{\vee}_{k})<0$ for sufficiently
large $k$. 
\end{cor}

When $n=2$, to check the relative K-semistability of a polygon, we
only need to take simple convex functions of the form $\max\{0,\alpha\}$
($\alpha$ is affine), see \cite{donaldson_scalar_2002,wang_existence_2011}.
This rule is not available in higher dimensions, but at least it suffices
to take all sGKZ vectors.

\subsection*{Organizations}

In §\ref{sec:2}, we review the structure of Berkovich analytification
of toric varieties. Translate the Fubini--Study and super-norm operators
into simple operations in convex-analysis, in order to derive the
formula of Chow-weights in §\ref{subsec:Toric-reduction-of Chow}.
In §\ref{sec:3}, we study the minimization problem for Chow-weight
via the secondary polytopes, and investigate the shortest GKZ vectors
in §\ref{subsec:Properties-of-shortest GKZ}. In §\ref{sec:4}, we
collect several properties of $\Theta_{P}$. §\ref{subsec:  min Qk proof}
is devoted to the main result: Theorem \ref{thm:main converge}. In
Appendix, we prove a folklore expansion of Riemann sums, and provide 
an algorithm to compute the shortest GKZ vectors. 

\subsection*{AI disclosure}

Generative AI tools (GPT, Gemini) were used to help construct the proof
of minimizer convergence and the Riemann sum expansion (\ref{eq:EM}).
Computer program $\mathtt{sGKZ}$ \cite{yao_program_2025} was built
by GPT 5.6 and tested by Sagemath function \cite{whitney_sagemath_nodate}. The authors verified all AI generated contents and take full responsibility for the paper.

\subsection*{Acknowledges }

We are grateful to Siarhei Finski, Jiyuan Han, Minghao Miao, Yalong Shi, Naoto Yotsutani and Bin Zhou for inspiring discussions. Special thanks to Mingchen Xia for his attention and suggestions on this project since its early stage. The third author is supported by the National Natural Science Foundation of China (NSFC) with Grant No. 12231006. 

\showsubsectionsinTOC

\section{\label{sec:2}Non-Archimedean quantization on toric varieties}

In this section, we review the non-Archimedean pluripotential theory
on toric varieties. We reduce the Fubini--Study map and super-norm
map to some convex-analysis operations. The main references are \cite{burgos_gil_arithmetic_2014,boucksom_non-archimedean_2024}. 

\subsection{Toric varieties}

We review the structure of toric varieties, refer to \cite{cox_toric_2011}
for details. 

Let $M$ be a lattice of rank $n$. It induce a group algebra
\[
\mathbb{C}[M]\coloneqq\left\{ \textrm{finite sum}\sum_{m\in M}c_{m}\mathsf{x}^{m}\mid c_{m}\in\bbc\right\} 
\]
whose product is defined by $\mathsf{x}^{m}\cdot\mathsf{x}^{m'}=\mathsf{x}^{m+m'}$.
There is a coproduct 
\[
\mathbb{C}[M]\ra\mathbb{C}[M]\otimes\mathbb{C}[M],\ \mathsf{x}^{m}\mapsto\mathsf{x}^{m}\otimes\mathsf{x}^{m},
\]
which makes $\mathbb{T}\coloneqq\mathrm{Spec}\mathbb{C}[M]$ a group
scheme. All $\bbc$-points of $\mathbb{T}$ constitute the complex
torus group $T_{\bbc}\coloneqq\mathrm{Hom}_{\bbz}(M,\mathbb{C}^{*})=N\otimes_{\mathbb{Z}}\mathbb{C^{*}}$,
where $N$ is the dual lattice of $M$. The characters of $T_{\bbc}$
are given by 
\[
\mathsf{x}^{m}:T_{\bbc}\rightarrow\mathbb{C^{*}},\ \mathsf{x}^{m}(n\otimes z)=z^{\left\langle m,n\right\rangle },\ n\in N,\ z\in\bbc^{*}.
\]
The cocharacters $\mathbb{C^{*}}\rightarrow T_{\bbc}$ are given by
$z\mapsto n\otimes z$, $\forall n\in N$. Let $T_{\bbr}\coloneqq N\otimes_{\mathbb{Z}}\mathbb{S}^{1}$
be the compact torus. 

A toric variety $X$ is a complex normal variety equipped with a faithful
$\mathbb{T}$-action and $X$ has an open dense orbit. It turns out
that $X$ can be described by a combinatorial structure called a fan.
A fan $\mathcal{F}$ is a finite collection of rational polyhedral
convex cones in $N_{\bbr}\coloneqq N\otimes_{\mathbb{Z}}\mathbb{R}$,
see \cite[Definition 3.1.2]{cox_toric_2011}. It induces a toric variety
$X_{\mathcal{F}}$, which is briefly reviewed below. 

First, each cone $\sigma\in\mathcal{F}$ induces a finitely-generated
semigroup $M_{\sigma}\coloneqq M\cap\sigma^{\vee}$, where $\sigma^{\vee}=\{x\in M_{\bbr}\mid\left\langle x,\cdot\right\rangle \geq0\ \textrm{on }\sigma\}$
is the dual cone. Then $M_{\sigma}$ induces a finitely generated
algebra $\mathbb{C}[M_{\sigma}]$ in the same way as $\mathbb{C}[M]$.
We obtain affine variety $X_{\sigma}\coloneqq\mathrm{Spec}\mathbb{C}[M_{\sigma}]$,
which carries a $\mathbb{T}$-action defined by the morphism 
\[
\mathbb{C}[M_{\sigma}]\ra\mathbb{C}[M]\otimes\mathbb{C}[M_{\sigma}],\ \mathsf{x}^{m}\ra\mathsf{x}^{m}\otimes\mathsf{x}^{m},\ \forall m\in M_{\sigma}.
\]
For two cones $\sigma\subset\sigma'$ in $\mathcal{F}$, since $M_{\sigma}\supset M_{\sigma'}$,
there is a $\mathbb{T}$-equivariant embedding $X_{\sigma}\hookrightarrow X_{\sigma'}$.
Then glueing $\left\{ X_{\sigma}\mid\sigma\in\mathcal{F}\right\} $
together via these embeddings, we obtain toric variety $X_{\mathcal{F}}$,
which is proper if and only if $N_{\bbr}$ is the union of all cones
in $\mathcal{F}$. 

Let $\mathcal{F}(i)\subset\mathcal{F}$ be the subset of $i$-dimensional
cones. These cones correspond one-to-one to the $\mathbb{T}$-orbits
in $X_{\mathcal{F}}$ of dimension $n-i$, see \cite[§3.2]{cox_toric_2011}.
Denote by $O(\sigma)$ the orbit corresponding to $\sigma\in\mathcal{F}$.
Its Zariski closure $V(\sigma)$ is also a toric variety of dimension
$n-i$, and we have $V(\sigma)=\bigsqcup_{\tau\succ\sigma}O(\tau)$.
For each 1-dimensional cone $\rho\in\mathcal{F}(1)$, let $n_{\rho}\in N$
be the primitive generator and $D_{\rho}\coloneqq V(\rho)$ the associated
prime divisor. For any $m\in M$, $\mathsf{x}^{m}$ defines a rational
function on $X_{\mathcal{F}}$ and we have 
\begin{equation}
\mathrm{div}(\mathsf{x}^{m})=\sum_{\rho\in\mathcal{F}(1)}\left\langle m,n_{\rho}\right\rangle D_{\rho}.\label{eq:div formu}
\end{equation}

Let $P$ be a full-dimensional lattice polytope in $M_{\mathbb{R}}$.
It induces an (inner) normal fan 
\[
\mathcal{F}_{P}\coloneqq\left\{ \sigma_{F}:F\ \text{is a face of }P\right\} ,
\]
where 
\[
\sigma_{F}\coloneqq\left\{ y\in N_{\bbr}\mid\min_{P}\left\langle \cdot,y\right\rangle =\left\langle x,y\right\rangle ,\ \forall x\in F\right\} .
\]
Let $\left\{ n_{\rho}\right\} $ be the primitive inner normal vectors
of the facets, then we have 
\[
P=\left\{ x\in M_{\mathbb{R}}\mid\left\langle x,n_{\rho}\right\rangle +a_{\rho}\geq0,\ \forall\rho\in\mathcal{F}_{P}(1)\right\} ,
\]
 where $a_{\rho}\in\mathbb{Z}$ since the vertices of $P$ are lattice
point. 

Consider the toric variety $X_{P}$ induced by the fan $\mathcal{F}_{P}$.
Define 
\[
D_{P}\coloneqq\sum_{\rho\in\mathcal{F}_{P}(1)}a_{\rho}D_{\rho},
\]
which is an ample and base-point-free Cartier divisor on $X_{P}$,
see \cite[Proposition 6.1.10]{cox_toric_2011}. Let $L_{P}=\mathcal{O}_{X_{P}}(D_{P})$
be the associated line bundle. In the next, subscript ``$P$'' will
be omitted since $P$ is fixed. We will fix a rational section $s_{D}$
of $L$ such that $\mathrm{div}(s_{D})=D$. The $\mathbb{T}$-action
on $X$ can be lifted to $L$ in a unique way such that $s_{D}$ is
equivariant. Let $m_{\sigma}\in M$ be the vertex of $P$ corresponding
to a maximal cone $\sigma\in\mathcal{F}_{P}(n)$. By the formula (\ref{eq:div formu}),
\begin{equation}
s_{\sigma}\coloneqq\mathsf{x}^{m_{\sigma}}s_{D}\in\mathrm{H}^{0}(X,L)\label{eq:frame}
\end{equation}
 is a global section and nonvanishing on $X_{\sigma}$. We will take
$s_{\sigma}$ as a frame of $L$ over $X_{\sigma}$. 

For any $k\geq1$, we regard the rescaled lattice set 
\[
P_{k}\coloneqq P\cap k^{-1}M
\]
as the ``quantization'' of $P$ at level $k$. For each $u\in P_{k}$,
$s_{u,k}\coloneqq\mathsf{x}^{ku}s^{k}_{D}$ is a global section of
$kL$. The action of $t\in T_{\bbc}$ on $s_{u,k}$ is $t\cdot s_{u,k}=\mathsf{x}^{ku}(t^{-1})s_{u,k}$.
With respect to the $T_{\bbc}$-action, we have a weight decomposition
\begin{equation}
R_{k}\coloneqq\mathrm{H}^{0}(X,kL)=\bigoplus_{u\in P_{k}}\mathbb{C}\cdot s_{u,k}.\label{eq:eigen basis}
\end{equation}
In particular, $N_{k}\coloneqq\dim R_{k}=\left|P_{k}\right|=V_{P}k^{n}+\frac{1}{2}V_{\partial P}k^{n-1}+\mathrm{O}\left(k^{n-2}\right)$. 

Since $kL$ is globally generated for any $k\geq1$, its global sections
induce a morphism $X\rightarrow\mathbb{P}R^{*}_{k}$. By \cite[Proposition 6.1.10]{cox_toric_2011},
it is an embedding ($kL$ is very ample) when $k\ge n-1$. 

\subsection{Torus-invariant NA norms }

Recall a NA norm $\left\Vert \right\Vert :R_{k}\rightarrow\mathbb{R}_{\geq0}$
is a norm satisfying that $\left\Vert as\right\Vert =\left\Vert s\right\Vert $
for $\forall a\in\mathbb{C}^{*}$ and $\left\Vert s+s'\right\Vert \leq\max\left(\left\Vert s\right\Vert ,\left\Vert s'\right\Vert \right)$.
We prefer to use the valuative notation $\chi=-\log\left\Vert \right\Vert $.
A NA norm $\chi$ is equivalent to a left-continuous non-increasing
filtration $(F^{\lam}_{\chi})_{\lam\in\bbr}$ on $R_{k}$, that is
$F^{\lam}_{\chi}=\{s\in R_{k}\mid\chi(s)\geq\lam\}$. Let $\mathcal{N}(R_{k})$
be the space of NA norms, and $\mathcal{N}(R_{k})^{T}$ be the subspace
of $T_{\bbc}$-invariant NA norms, i.e. those satisfy $\chi(t\cdot s)=\chi(s)$
for any $t\in T_{\bbc}$ and $s\in R_{k}$. 
\begin{prop}
A NA norm $\chi$ on $R_{k}$ is $T_{\bbc}$-invariant if and only
if $\chi$ is diagonalized by the eigen-sections $\left\{ s_{u,k}\mid u\in P_{k}\right\} $. 
\end{prop}

\begin{proof}
For any $\chi\in\mathcal{N}(R_{k})^{T}$, the associated filtration
$F^{\lambda}=\left\{ s\in R_{k}\mid\chi(s)\geq\lambda\right\} $ is
$T_{\bbc}$-invariant for any $\lam$. Since $R_{k}$ is multiplicity-free,
$F^{\lambda}$ is spanned by some $s_{u,k}$. This implies that $\chi$
is diagonalized by $(s_{u,k})$. The converse direction is trivial. 
\end{proof}

By this fact, we identify $\mathcal{N}(R_{k})^{T}$ with the space
$\bbr^{P_{k}}\coloneqq\{f:P_{k}\rightarrow\mathbb{R}\}$ via the map
\begin{equation}
f\mapsto\chi_{f},\ \textrm{where }\chi_{f}\left(\sum_{u\in P_{k}}c_{u}s_{u,k}\right)=\min\left\{ -kf(u)\mid c_{u}\neq0\right\} .\label{eq:identify N W}
\end{equation}
There is a Tits building structure on $\mathcal{N}(R_{k})$ such that
$\mathcal{N}(R_{k})^{T}$ is an apartment, see \cite{boucksom_variational_2019}. 

\subsection{Berkovich analytification }

Let $(X,L)$ be the polarized toric variety associated to lattice
polytope $P$. Endowing $\bbc$ with the trivial valuation. As a set,
Berkovich's analytification $X^{\an}$ is 
\[
X^{\an}\coloneqq\bigcup_{\sigma\in\mathcal{F}_{P}}X^{\an}_{\sigma},\ X^{\an}_{\sigma}\coloneqq\left\{ v:\mathbb{C}[M_{\sigma}]\rightarrow\mathbb{R}\cup\{+\infty\}\right\} ,
\]
where $v$ is semivaluation extending the trivial valuation on $\mathbb{C}$.
For $f\in\mathbb{C}[M_{\sigma}]$ and $v\in X^{\an}_{\sigma}$, we
write $\left|f\right|(v)\coloneqq e^{-v(f)}$. When $\sigma=\{0\}$,
$X^{\an}_{\sigma}$ is the NA analytic torus $\mathbb{T}^{\an}$.
Endowed with the Berkovich topology, $X^{\an}$ is a compact Hausdorff
space. There is a kernel map $X^{\an}\ra X$ and a center map $X^{\an}\ra X$,
which are continuous and anti-continuous respectively. 

The \emph{tropicalization} of $X^{\an}$ is a topological space $N_{P}$
that is homeomorphic to $P$ and contains $N_{\mathbb{R}}$ as an
open dense subset, see \cite[§4.1]{burgos_gil_arithmetic_2014}. There
are two continuous maps 
\[
\varrho:X^{\an}\rightarrow N_{P},\ \vartheta:N_{P}\rightarrow X^{\an}
\]
called the tropicalization and the Gauss section respectively. We
only give their definitions on $\mathbb{T}^{\an}$ and $N_{\mathbb{R}}$.
For $v\in\mathbb{T}^{\an}$, i.e. a semivaluation on $\mathbb{C}[M]$,
$\varrho(v)\in N_{\mathbb{R}}$ is defined as 
\[
\varrho(v):M\rightarrow\mathbb{R},\ m\mapsto-v(\mathsf{x}^{m}).
\]
Note that $v(\mathsf{x}^{m})$ must be finite since $v(\mathsf{x}^{m})+v(\mathsf{x}^{-m})=v(1)=0$. 

For the Gauss section, given a $y\in N_{\mathbb{R}}$, $\vartheta_{y}\coloneqq\vartheta(y)\in\mathbb{T}^{\an}$
is defined by 
\[
\vartheta_{y}\left(\sum_{m\in M}c_{m}\mathsf{x}^{m}\right)=\min_{c_{m}\neq0}\left\langle m,-y\right\rangle ,\ \textrm{if }\exists c_{m}\neq0.
\]
It is a valuation on the function field of $X$. For instance, $\vartheta(0)=v_{tr}$
the trivial valuation. Moreover, the center of $\vartheta_{y}$ is
the generic point of toric subvariety $V(\sigma_{-y})$, where $\sigma_{-y}\in\mathcal{F}_{P}$
is the minimal cone containing $-y$. 

The image of $\vartheta$ is called the skeleton of $X^{\an}$, denoted
by $\mathrm{Sk}(X^{\an})$. We have $\varrho\circ\vartheta=\mathrm{Id}$,
so $\mathscr{R}\coloneqq\vartheta\circ\varrho$ is a retraction to
the skeleton. For $v\in\mathbb{T}^{\an}$, we have 
\begin{equation}
\left(\mathscr{R}v\right)\left(\sum_{m\in M}c_{m}\mathsf{x}^{m}\right)=\min_{c_{m}\neq0}v(\mathsf{x}^{m})\leq v\left(\sum_{m\in M}c_{m}\mathsf{x}^{m}\right).\label{eq: Rv}
\end{equation}

\begin{center}
\begin{tikzcd}
\mathbb{T}^\mathrm {an} \arrow[r, hook] \arrow[d, "\varrho"'] & X^{\mathrm {an}} \arrow[d, "\varrho"'] \arrow[r, "\mathscr{R}", bend left] & \mathrm{Sk}(X^{\mathrm {an}}) \arrow[l, hook'] \\
N_{\mathbb R} \arrow[r, "\text{dense}", hook] & N_P \arrow[ru, "\vartheta"'] &
\end{tikzcd}
\end{center}
\begin{rem}
Although $\mathbb{T}^{\an}$ has no usual group structure, one can
define a multi-valued product and an action on $X^{\an}$, refer to
\cite[§4.2]{burgos_gil_arithmetic_2014}. The NA compact torus $\mathbb{S}^{\an}\subset\mathbb{T}^{\an}$
is the subset of semivaluations such that $v(\mathsf{x}^{m})=0$ for
all $m\in M$. Then the tropicalization $N_{P}$ is the quotient space
of $X^{\an}$ modulo the action of $\mathbb{S}^{\an}$. In the Archimedean
setting, $N_{P}$ is the quotient space of $X(\bbc)$ modulo the action
of $T_{\bbr}$. So in both settings, $N_{P}$ is the common stage
for toric reduction. 
\end{rem}

\subsection{Toric metrics }

Let $L^{\an}$ be the analytification of $L$. Since the base field
$\bbc$ is trivially normed, there is a canonical metric on $L^{\an}$
called the trivial metric. Regarding the trivial metric, the norm
function of a section $s\in R_{k}$ is $\left|s\right|_{tr}(v)\coloneqq e^{-v(s)}$.
When $v$ is a divisorial valuation, $v(s)$ is just the vanishing
order of $s$ along the divisor. Any metric on $L^{\an}$ is represented
by a function $\vphi:X^{\an}\ra\bbr\cup\{-\infty\}$. With respect
to metric $\vphi$, the norm of $s\in R_{k}$ is $\left|s\right|_{\varphi}\coloneqq\left|s\right|_{tr}e^{-k\varphi}$. 

A Fubini--Study metric $\varphi$ on $L^{\an}$ is defined to have
the form 
\[
\varphi=\frac{1}{m}\max_{j}\left\{ \log\left|s_{j}\right|_{tr}+\lambda_{j}\right\} ,
\]
where $m\geq1$, $\lambda_{j}\in\mathbb{R}$ and $\{s_{j}\}\subset R_{m}$
is a finite set of sections with no common zeros. A psh metric $\varphi$
(not $\equiv-\infty$) on $L^{\an}$ is defined to be the pointwise
limit of a decreasing sequence of Fubini--Study metrics. Denote by
$\mathrm{PSH}^{\na}(L)$ the set of psh metrics, and $\mathrm{CPSH}^{\na}(L)$
the set of continuous psh metrics. We may omit ``$L$'' when it
is fixed. 
\begin{defn}
Suppose that $(X,L)$ is a polarized variety and $kL$ is globally
generated. The NA Fubini--Study map $\mathrm{FS}_{k}:\mathcal{N}(R_{k})\rightarrow\mathrm{CPSH}^{\na}(L)$
is defined by 
\begin{equation}
\mathrm{FS}_{k}(\chi)\coloneqq\frac{1}{k}\sup_{s\in R_{k}\backslash\{0\}}\left\{ \log\left|s\right|_{tr}+\chi(s)\right\} =\frac{1}{k}\max_{1\leq j\leq N_{k}}\left\{ \log\left|s_{j}\right|_{tr}+\chi(s_{j})\right\} ,\label{def FSk}
\end{equation}
where $(s_{j})$ is any basis diagonalizing $\chi$. 
\end{defn}

When $(X,L)$ is the toric variety associated to lattice polytope
$P$, $kL$ is always globally generated. If $\chi$ is $T_{\bbc}$-invariant,
then $\mathrm{FS}_{k}(\chi)$ is torus-invariant in the following
sense. 
\begin{defn}
Let $(X,L)$ be the toric variety associated to a lattice polytope
$P\subset M_{\bbr}$. 

(1) For any $U\subset N_{P}$, function $\vphi:\varrho^{-1}(U)\rightarrow\bbr$
is \textit{torus-invariant} if $\mathscr{R}^{*}\vphi=\varphi$, i.e.
$\varphi$ is constant along the fibers of tropicalization map. Let
$\mathrm{PSH}^{\na}_{T}(L)$ ($\mathrm{CPSH}^{\na}_{T}(L)$) be the
space of torus-invariant (continuous) psh metrics on $L^{\an}$. 

(2) For a torus-invariant metric $\vphi$ on $L^{\an}$, its tropical
potential $F_{\varphi}:N_{\mathbb{R}}\rightarrow\mathbb{R}$ is defined
as 
\begin{equation}
F_{\varphi}(y)\coloneqq-\log\left|s_{D}\right|_{\varphi}(\vartheta_{y})=\vartheta_{y}(s_{D})+\varphi(\vartheta_{y})=F_{P}(y)+\varphi(\vartheta_{y}).\label{eq:def of trop poten}
\end{equation}
The last $=$ is due to Proposition \ref{prop: trop of FS} (1) below,
and $F_{P}(y)\coloneqq\max_{u\in P}\left\langle u,y\right\rangle $
is the support function of $P$. 
\end{defn}

For instance, as the functions on $\mathbb{T}^{\an}$, $v\mapsto v(\mathsf{x}^{m})$
and $\left|s_{D}\right|_{tr}$ are torus-invariant. Thus for $\forall u\in P_{k}$,
the norm function $\left|s_{u,k}\right|_{tr}=\left|\mathsf{x}^{ku}\right|\left|s_{D}\right|^{k}_{tr}$
is torus-invariant. For any $\chi\in\mathcal{N}(R_{k})^{T}$, since
it is diagonalized by $(s_{u,k})$, $\mathrm{FS}_{k}(\chi)$ is torus-invariant. 
\begin{prop}
\label{prop: trop of FS}(1) For any $y\in N_{\bbr}$, we have $\vartheta_{y}(s_{D})=F_{P}(y)\coloneqq\max_{u\in P}\left\langle u,y\right\rangle $. 

(2) Let $\chi\in\mathcal{N}(R_{k})^{T}$ given by $f:P_{k}\ra\bbr$
such that $\chi(s_{u,k})=-kf(u)$. The tropical potential for $\mathrm{FS}_{k}(\chi)$
is 
\begin{equation}
F(y)=\max_{u\in P_{k}}\left\{ \left\langle u,y\right\rangle -f(u)\right\} ,\ \forall y\in N_{\mathbb{R}}.\label{eq:trop of FS}
\end{equation}
\end{prop}

\begin{proof}
(1) Take a full-dimensional cone $\sigma\in\mathcal{F}_{P}$ containing
$-y$. Let $m_{\sigma}$ be the corresponding vertex of $P$, then
we have $\left\langle m_{\sigma},y\right\rangle =\max_{u\in P}\left\langle u,y\right\rangle =F_{P}(y)$.
Note that $X_{\sigma}$ contains the center of valuation $\vartheta_{y}$,
and $s_{\sigma}\coloneqq\mathsf{x}^{m_{\sigma}}s_{D}$ (\ref{eq:frame})
is a frame of $L$ over $X_{\sigma}$. We represent $s_{D}=\mathsf{x}^{-m_{\sigma}}s_{\sigma}$.
By the definition of $\vartheta_{y}(s_{D})$, we have 
\[
\vartheta_{y}(s_{D})=\vartheta_{y}(\mathsf{x}^{-m_{\sigma}})=\left\langle m_{\sigma},y\right\rangle =F_{P}(y).
\]

(2) Since $\chi$ is diagonalized by $\left\{ s_{u,k}\coloneqq\mathsf{x}^{ku}s^{k}_{D}\mid u\in P_{k}\right\} $,
we have 
\[
\mathrm{FS}_{k}(\chi)=\max_{u\in P_{k}}\left\{ k^{-1}\log\left|s_{u,k}\right|_{tr}-f(u)\right\} .
\]
By (1), we have $\log\left|s_{u,k}\right|_{tr}(\vartheta_{y})=-\vartheta_{y}(s_{u,k})=\left\langle ku,y\right\rangle -kF_{P}(y)$.
Put it into the definition (\ref{eq:def of trop poten}), then (\ref{eq:trop of FS})
follows. 
\end{proof}

To capture the behavior of tropical potentials at infinity, we consider
their Legendre transforms. The Legendre transform of a function $F:N_{\mathbb{R}}\rightarrow\mathbb{R}$
is defined as 
\begin{align*}
F^{*}(x) & =\sup_{y\in N_{\mathbb{R}}}\left\{ \left\langle x,y\right\rangle -F(y)\right\} :M_{\mathbb{R}}\rightarrow\mathbb{R}\cup\{+\infty\}.
\end{align*}
Conversely, the Legendre transform of $\phi:M_{\mathbb{R}}\rightarrow\mathbb{R}\cup\{+\infty\}$
($\not\equiv+\infty$) is defined as 
\[
\phi^{*}(y)=\sup_{x\in M_{\mathbb{R}}}\left\{ \left\langle x,y\right\rangle -\phi(x)\right\} :N_{\mathbb{R}}\rightarrow\mathbb{R}\cup\{+\infty\}.
\]
Let $\mathcal{C}(P)$ be the space of convex functions on $P$ which
are continuous up to the boundary. We have a bijection 
\[
\left\{ \text{convex }F:N_{\bbr}\ra\bbr\mid\left|F-F_{P}\right|\le C\right\} \xrightarrow{\simeq}\mathcal{C}(P),\ F\mapsto F^{*}.
\]
Since for any $F$ in the LHS, $F^{*}$ is a finite-valued lsc convex
function on $P$. By \cite{gale_convex_1968}, $F^{*}$ is also upper
semicontinuous, thus $F^{*}\in\mathcal{C}(P)$. 
\begin{thm}[{\cite[Theorem 4.8.1]{burgos_gil_arithmetic_2014}}]
 Let $(X,L)$ be the toric variety given by a lattice polytope $P\subset M_{\bbr}$.
Then the map 
\begin{equation}
\mathrm{CPSH}^{\na}_{T}(L)\ra\mathcal{C}(P),\ \varphi\mapsto\left(F_{\varphi}\right)^{*}\label{eq:identify CPSH to conv}
\end{equation}
 is a bijection, where $F_{\vphi}$ (\ref{eq:def of trop poten})
is the tropical potential of $\vphi$. 
\end{thm}

For any $\chi\in\mathcal{N}(R_{k})^{T}$ given by $f:P_{k}\ra\bbr$,
we define $f(x)=+\infty$ for $x\notin P_{k}$, then (\ref{eq:trop of FS})
is just $f^{*}$. Hence the Legendre transform of the tropical potential
of $\mathrm{FS}_{k}(\chi)$ is $\left(f^{*}\right)^{*}$, which is
the convex envelope of $f$, i.e. 
\begin{equation}
f^{\vee}(x)\coloneqq\sup\left\{ \eta(x)\mid\eta\in\mathcal{C}(P),\ \eta\le f\ \mathrm{on}\ P_{k}\right\} ,\ \text{for }x\in P.\label{eq:conv envelop}
\end{equation}
Obviously, $f^{\vee}$ is a piecewise affine convex function over
$P$. For any affine function $\alpha$, we have $\left(f+\alpha|_{P_{k}}\right)^{\vee}=f^{\vee}+\alpha$.
For any $f_{1},f_{2}\in\bbr^{P_{k}}$, we have 
\[
\left(f_{1}+f_{2}\right)^{\vee}\geq f^{\vee}_{1}+f^{\vee}_{2}.
\]

In a summary, under (\ref{eq:identify N W}) and (\ref{eq:identify CPSH to conv}),
Fubini--Study map $\mathrm{FS}_{k}:\mathcal{N}(R_{k})^{T}\ra\mathrm{CPSH}^{\na}_{T}(L)$
is identified with the convex envelope map 
\begin{equation}
\vee:\bbr^{P_{k}}\ra\mathcal{C}(P),\ f\mapsto f^{\vee}.\label{eq:identify FS_k}
\end{equation}

\subsection{The super-norm map}

The super-norm map is introduced in \cite{boucksom_spaces_2021,boucksom_non-archimedean_2018},
which plays the role of the Hilbert map in the Archimedean geometric
quantization picture. 
\begin{defn}
Let $(X,L)$ be a polarized variety. Denote by $\mathcal{L}^{\infty}$
the space of bounded functions on $X^{\an}$. The super-norm map $\mathrm{SN}_{k}:\mathcal{L}^{\infty}\rightarrow\mathcal{N}(R_{k})$
is defined as
\begin{equation}
\mathrm{SN}_{k}(\varphi)(s)\coloneqq-\log\left(\sup_{X^{\an}}\left|s\right|_{tr}e^{-k\varphi}\right)=\inf_{v\in X^{\an}}\left\{ v(s)+k\varphi(v)\right\} .\label{def SNk}
\end{equation}
\end{defn}

In the toric setting, we show the super-norms induced by torus-invariant
metrics are $T_{\bbc}$-invariant. 
\begin{prop}
\label{prop:SN is inv}(1) Suppose that metric $\varphi$ is torus-invariant,
then for any $s\in R_{k}$ and $v\in X^{\an}$ we have $\left|s\right|_{\varphi}(v)\leq\left|s\right|_{\varphi}(\mathscr{R}v)$,
where $\mathscr{R}:X^{\an}\ra\mathrm{Sk}(X^{\an})$ is (\ref{eq: Rv}).
This implies 
\begin{equation}
\sup_{X^{\an}}\left|s\right|_{\varphi}=\sup_{\mathrm{Sk}(X^{\an})}\left|s\right|_{\varphi}.\label{eq:sup skeleton}
\end{equation}

(2) If bounded metric $\varphi$ is torus-invariant, then the induced
super-norm $\mathrm{SN}_{k}(\varphi)$ is $T_{\bbc}$-invariant. 
\end{prop}

\begin{proof}
(1) By (\ref{eq: Rv}) and $\varphi(v)=\varphi(\mathscr{R}v)$ for
any $v\in X^{\an}$, we have 
\[
\left|s\right|_{\varphi}(v)=\left|s\right|_{tr}(v)e^{-k\varphi(v)}\leq\left|s\right|_{tr}(\mathscr{R}v)e^{-k\varphi(\mathscr{R}v)}=\left|s\right|_{\varphi}(\mathscr{R}v).
\]
(2) By (\ref{eq:sup skeleton}), we only need to show that $\left|t\cdot s\right|_{\varphi}(v)=\left|s\right|_{\varphi}(v)$
for any $t\in T_{\bbc}$, $s\in R_{k}$ and $v\in\mathrm{Sk}(X^{\an})$.
It amounts to show $v(t\cdot s)=v(s)$. Let $s=\sum_{u\in P_{k}}c_{u}s_{u,k}$,
then $t\cdot s=\sum c_{u}t^{-ku}s_{u,k}$, where $t^{-ku}\coloneqq\mathsf{x}^{-ku}(t)\in\mathbb{C}^{*}$.
Take a maximal affine open set $X_{\sigma}$ containing the center
of $v$, and represent $s_{u,k}=\mathsf{x}^{ku-km_{\sigma}}s^{k}_{\sigma}$
by the local frame $s_{\sigma}$ of $L$, then we have 
\[
v(t\cdot s)=v\left(\sum_{u\in P_{k}}c_{u}t^{-ku}\mathsf{x}^{ku-km_{\sigma}}\right)=\min_{c_{u}\neq0}v\left(\mathsf{x}^{ku-km_{\sigma}}\right)=v(s),
\]
where the second $=$ is due to $\mathscr{R}v=v$ and (\ref{eq: Rv}). 
\end{proof}

By the above proposition, the super-norms induced by invariant metrics
are determinated by their values on eigen-sections $s_{u,k}$. 
\begin{prop}
\label{prop: SN is res}Assume that $\varphi:X^{\an}\rightarrow\bbr$
is torus-invariant, bounded and upper semi-continuous, then we have
$\mathrm{SN}_{k}(\varphi)(s_{u,k})=-kF^{*}_{\varphi}(u)$ for all
$k\geq1$ and $u\in P_{k}$, where $F^{*}_{\varphi}$ is the Legendre
transform of the tropical potential $F_{\varphi}$ of $\varphi$. 
\end{prop}

\begin{proof}
For $s\in R_{k}$, the function $v\mapsto\left|s\right|_{\varphi}=e^{-v(s)-k\varphi(v)}$
is lsc on $X^{\an}$, and $T^{\an}\subset X^{\an}$ is dense, so we
have $\sup_{X^{\an}}\left|s\right|_{\varphi}=\sup_{T^{\an}}\left|s\right|_{\varphi}$.
By Proposition \ref{prop:SN is inv} (1), it suffices to take the
supremum over the skeleton, so $\sup_{T^{\an}}\left|s\right|_{\varphi}=\sup_{y\in N_{\bbr}}\left|s\right|_{\varphi}(\vartheta_{y})$.
Hence 
\begin{align*}
\mathrm{SN}_{k}(\varphi)(s_{u,k}) & =\inf_{y\in N_{\mathbb{R}}}\left\{ \vartheta_{y}(s_{u,k})+k\varphi(\vartheta_{y})\right\} =\inf_{y\in N_{\mathbb{R}}}\left\{ \vartheta_{y}\left(\mathsf{x}^{ku}\right)+\vartheta_{y}\left(s^{k}_{D}\right)+k\varphi(\vartheta_{y})\right\} \\
 & =k\inf_{y\in N_{\mathbb{R}}}\left\{ F_{\varphi}(y)-\left\langle u,y\right\rangle \right\} =-kF^{*}_{\varphi}(u),
\end{align*}
where in the second row we used $\vartheta_{y}\left(s_{D}\right)=F_{P}(y)$
(Proposition \ref{prop: trop of FS}) and (\ref{eq:def of trop poten}). 
\end{proof}

In a summary, under (\ref{eq:identify N W}) and (\ref{eq:identify CPSH to conv}),
super-norm map $\mathrm{SN}_{k}:\mathrm{CPSH}^{\na}_{T}(L)\ra\mathcal{N}(R_{k})^{T}$
is identified with the restriction map 
\[
\mathcal{C}(P)\ra\bbr^{P_{k}},\ \eta\mapsto\eta|_{P_{k}}.
\]

\subsection{Asymptotics of the volumes of super-norms}

Let $(X,L)$ be a polarized variety. A NA norm $\chi$ on $R_{k}$
corresponds to a filtration $F^{\lambda}_{\chi}=\{s\in R_{k}\mid\chi(s)\geq\lam\}$
with jumping numbers: 
\[
\lambda_{i}(\chi)\coloneqq\max\{\lambda\mid\dim F^{\lambda}_{\chi}\geq i\},\ 1\leq i\leq N_{k}\coloneqq\dim R_{k}.
\]
If $\{s_{i}\}\subset R_{k}$ is a basis diagonalizing $\chi$, then
$\lam_{1}\geq\cdots\geq\lam_{N_{k}}$ is the re-arrangement of $\left(\chi(s_{i})\right)_{i}$. 

The volume of $\chi$ is defined as 
\begin{equation}
E_{k}(\chi)\coloneqq\frac{1}{kN_{k}}\sum^{N_{k}}_{i=1}\lambda_{i}(\chi).\label{eq: def Ek}
\end{equation}
For $\varphi\in\mathrm{CPSH}^{\na}(L)$, Boucksom--Eriksson \cite[Theorem A]{boucksom_spaces_2021}
gives the leading term of volume asymptotics, that is 
\begin{equation}
E_{k}\left(\mathrm{SN}_{k}(\varphi)\right)=\mathbb{E}(\varphi)+\mathrm{o}(1),\textrm{ as }k\ra\infty,\label{eq: BE expand}
\end{equation}
where $\mathbb{E}(\varphi)$ is the NA Monge--Ampère energy. The
next-order term is expected to be the NA K-energy, which is readily
verified when $\vphi$ is induced by an ample test-configuration,
see \cite[Theorem 46]{yao_maximal_2025}. 

Suppose that $(X,L)$ is the polarized toric variety given by lattice
polytope $P$, and $\varphi\in\mathrm{CPSH}^{\na}_{T}(L)$ with tropical
potential $F_{\vphi}$. By Proposition \ref{prop: SN is res}, we
have 
\[
E_{k}\left(\mathrm{SN}_{k}(\varphi)\right)=-\frac{1}{\left|P_{k}\right|}\sum_{u\in P_{k}}F^{*}_{\varphi}(u).
\]
Since $F^{*}_{\varphi}$ is a continuous convex function over $P$,
by the asymptotics of Riemann sum (\ref{eq:EM}), we obtain 
\[
E_{k}\left(\mathrm{SN}_{k}(\varphi)\right)=-\frac{1}{V_{P}}\int_{P}F^{*}_{\varphi}\ \mathrm{d}x-\frac{1}{2V_{P}k}\mathbb{M}_{P}\left(F^{*}_{\varphi}\right),
\]
where 
\[
\mathbb{M}_{P}\left(\eta\right)\coloneqq\int_{\partial P}\eta\ \mathrm{d}\sigma-\int_{P}\eta\bar{S}\ \mathrm{d}x,\ \bar{S}\coloneqq\frac{V_{\partial P,\mathrm{d}\sigma}}{V_{P}}
\]
is the toric NA K-energy, see §\ref{subsec:Szekelyhidi's-optimal-test}
for more. Compare with (\ref{eq: BE expand}), we obtain 
\begin{equation}
\mathbb{E}(\varphi)=-\frac{1}{V_{P}}\int_{P}F^{*}_{\varphi}\ \mathrm{d}x,\ \forall\varphi\in\mathrm{CPSH}^{\na}_{T}(L).\label{eq:MA toric formula}
\end{equation}

\section{\label{sec:3}Maximum Chow-destabilizers and secondary polytopes}

In this section, we show that the maximal Chow-destabilizer of toric
variety $(X_{P},L_{P})$ is determined by the shortest GKZ vector,
which is the least norm point on the secondary polytope associated
to point set $P_{k}$. 

\subsection{\label{subsec:Toric-reduction-of Chow}Toric reduction of the maximal
Chow-destabilizers}

By the torus symmetry, the minimization problem for Chow-weights can
be reduced into a convex optimization problem. 

Given a polarized variety $(X,L)$, suppose $k$ is large enough such
that $\iota_{k}:X\hookrightarrow\mathbb{P}R^{*}_{k}$ is an embedding.
Then we can talk about the Chow-stability of $\iota_{k}(X)$ in the
sense of GIT, see \cite[§4.1]{yao_maximal_2025} for a recap of Chow-stability. 

Chow-stability is usually defined via the Chow-weights assigned to
the 1-ps of $\mathrm{SL}(R_{k})$. The 1-ps with minimum normalized
Chow-weight can differ from each other by conjugations. To eliminate
this ambiguity, \cite{yao_maximal_2025} replace 1-ps by the NA norms
on $R_{k}$. Then Chow-weight turns into a functional $M_{k}$ on
$\mathcal{N}(R_{k})$ which has a NA expression (see \cite[Theorem 3]{yao_maximal_2025})
\begin{equation}
M_{k}(\chi)=\mathbb{E}\left(\mathrm{FS}_{k}(\chi)\right)-E_{k}(\chi),\ \forall\chi\in\mathcal{N}(R_{k}).\label{eq: Chow wt}
\end{equation}
We focus on the \emph{Chow-unstable} case, i.e. $\inf M_{k}<0$. Since
$M_{k}$ is homogeneous, we consider the normalized Chow-weight: 
\[
\bar{M}_{k}(\chi)\coloneqq\frac{M_{k}(\chi)}{\left\Vert \chi\right\Vert _{2}},\ \textrm{for }\chi\in\mathcal{N}(R_{k})\backslash\{\chi_{tr}\},\ \textrm{where }\left\Vert \chi\right\Vert ^{2}_{2}\coloneqq\frac{1}{N_{k}}\sum^{N_{k}}_{i=1}\lambda^{2}_{i}(\chi).
\]
By \cite[Theorem 1]{yao_maximal_2025}, in the unstable case, $\bar{M}_{k}$
admits a unique (modulo scaling) minimizer $\chi^{*}_{k}$, which
is called the maximal Chow-destabilizer for $\iota_{k}(X)$. Moreover,
uniqueness implies that $\chi^{*}_{k}$ inherits all symmetries of
$(X,L)$, see \cite[Corollary 2]{yao_maximal_2025}. 

Let $(X,L)$ be the toric variety given by lattice polytope $P$.
When $k\ge n-1$, $\iota_{k}$ is an embedding. Suppose $\iota_{k}(X)$
is Chow-unstable, by the symmetry principle mentioned above, the maximal
Chow-destabilizer is $T_{\bbc}$-invariant. Hence, to determine it,
it suffices to minimize $\bar{M}_{k}$ over $\mathcal{N}(R_{k})^{T}$.
Under identification $\mathcal{N}(R_{k})^{T}\cong\bbr^{P_{k}}$ and
(\ref{eq:identify FS_k}), and use (\ref{eq:MA toric formula}), we
see that Chow-weight (\ref{eq: Chow wt}) is identified with 
\begin{equation}
M_{P_{k}}:\bbr^{P_{k}}\ra\bbr,\ M_{P_{k}}(f)\coloneqq\frac{1}{N_{k}}\sum_{u\in P_{k}}f(u)-\frac{1}{V_{P}}\int_{P}f^{\vee}\ \mathrm{d}x.\label{eq:chow weight Pk}
\end{equation}
The normalized Chow-weight is identified with 
\begin{equation}
\bar{M}_{P_{k}}(f)\coloneqq\frac{M_{P_{k}}(f)}{\left\Vert f\right\Vert _{2}},\ \left\Vert f\right\Vert ^{2}_{2}\coloneqq\sum_{u\in P_{k}}f(u)^{2}.\label{eq: Chow weight formula}
\end{equation}
The next task is to find the minimizer of $\bar{M}_{P_{k}}$. We generalize
this problem by replacing $P_{k}$ with a more general point configuration. 

\subsection{Marked polytopes }
\begin{defn}
(1) Let $V$ be a $n$-dimensional real vector space. A \emph{marked
polytope} in $V$ is a pair $(Q,A)$, where $A\subset V$ is a finite
subset (mark points) and its convex hull $Q=\mathrm{conv}(A)$ is
full-dimensional. Our primary example is $(P,P_{k})$. 

(2) For $(Q,A)$, we define an inner product space 
\[
\bbr^{A}\coloneqq\left\{ f:A\rightarrow\mathbb{R}\right\} ,\ \left\langle f,g\right\rangle \coloneqq\sum_{u\in A}f(u)g(u).
\]
The convex envelope map $\vee:\bbr^{A}\rightarrow\mathcal{C}(Q)$
is defined as before. 
\end{defn}

\begin{rem}
Even if $A$ is not $P_{k}$, $\vee$ also relates to the NA Fubini--Study
map. Let $M$ be a lattice, and $(Q,A)$ is a marked polytope in $M_{\bbr}$.
Suppose $A=\{m_{1},\cdots,m_{s}\}\subset M$, then the Zariski closure
of the image of map 
\[
T_{\bbc}\ra\mathbb{P}^{s-1},\ t\mapsto\left[\mathsf{x}^{m_{1}}(t):\cdots:\mathsf{x}^{m_{s}}(t)\right]
\]
is a toric variety $X_{A}$ (may not isomorphic to $X_{Q}$). Take
the Berkovich analytification of $X_{A}$, and consider the Fubini--Study
map from the space of invariant NA norms on $\bbc^{a}$ to the space
of NA metrics on $\mathcal{O}(1)|_{X_{A}}$, we will arrive at the
envelope map $\vee$. 
\end{rem}

There are several interactive properties between Fubini--Study map
and super-norm map, see \cite[Lemma 7.23]{boucksom_spaces_2021} and
\cite[Proposition 35]{yao_maximal_2025}. The following are its toric
concretizations. 
\begin{prop}
\label{prop:FS vs SN} Given a marked polytope $(Q,A)$. 
\end{prop}

\begin{enumerate}
\item For $\forall f\in\bbr^{A}$ and $\eta\in\mathcal{C}(Q)$, we have
\[
f^{\vee}|_{A}\le f,\ \left(\eta|_{A}\right)^{\vee}\ge\eta.
\]
\item Consider the image of operation $\vee$, we have 
\[
\mathrm{PL}(Q,A)\coloneqq\left\{ f^{\vee}\mid f\in\bbr^{A}\right\} =\left\{ \eta\in\mathcal{C}(Q)\mid\left(\eta|_{A}\right)^{\vee}=\eta\right\} .
\]
Consider the restrictions of convex functions, we have 
\[
\mathcal{C}(A)\coloneqq\left\{ \eta|_{A}\mid\eta\in\mathcal{C}(Q)\right\} =\left\{ f\in\bbr^{A}\mid f^{\vee}|_{A}=f\right\} .
\]
\item We have a bijection $\vee:\mathcal{C}(A)\rightarrow\mathrm{PL}(Q,A)$,
whose inverse is the restriction map. 
\[
\begin{array}{ccc}
\bbr^{A} & \xrightleftharpoons[\mathrm{restriction}]{\vee} & \mathcal{C}(Q)\\
\cup &  & \cup\\
\mathcal{C}(A) & \stackrel{1-1}{\longleftrightarrow} & \mathrm{PL}(Q,A)
\end{array}
\]
\end{enumerate}
\begin{proof}
(1) is obvious. (2) For the first equality, $\supset$ is trivial.
For $\forall f\in\bbr^{A}$, let $\eta=f^{\vee}$. By (1) we have
$\eta|_{A}\le f$, so $\left(\eta|_{A}\right)^{\vee}\le f^{\vee}=\eta$.
By (1) we also have $\left(\eta|_{A}\right)^{\vee}\ge\eta$, hence
$\left(\eta|_{A}\right)^{\vee}=\eta$, so the $\subset$ direction
holds. The second equality can be proved similarly. (3) follows from
(2). 
\end{proof}

\begin{defn}
\label{def: Chow-weight}For marked polytope $(Q,A)$, we define the
``Chow-weight'' functional $M_{A}:\bbr^{A}\rightarrow\mathbb{R}$
by 
\begin{equation}
M_{A}(f)\coloneqq\frac{1}{\left|A\right|}\sum_{u\in A}f(u)-\frac{1}{\left|Q\right|_{\mathrm{d}x}}\int_{Q}f^{\vee}\ \mathrm{d}x,\label{eq: M_A}
\end{equation}
which generalizes (\ref{eq:chow weight Pk}). It is independent of
the choices of Lebesgue measure $\mathrm{d}x$. We say $(Q,A)$ is
\emph{Chow-semistable} if $M_{A}(f)\geq0$ for all $f\in\bbr^{A}$.
Otherwise, we say $(Q,A)$ is \emph{Chow-unstable}. 
\end{defn}

Suppose $(Q,A)$ is Chow-unstable, consider the minimization problem:
\begin{equation}
\inf\left\{ \bar{M}_{A}(f)\coloneqq\frac{M_{A}(f)}{\left\Vert f\right\Vert }\Bigm|0\not\equiv f:A\ra\bbr\right\} .\label{eq: min M_A}
\end{equation}
We analyze this problem via the secondary polytopes. 

\subsection{Secondary polytopes }

Roughly speaking, the secondary polytope $\Sigma(A)$ associated to
a marked polytope $(Q,A)$ is the ``moduli space'' of polyhedral
subdivisions of $Q$ with nodes in $A$. The main references are \cite[§7]{gelfand_discriminants_1994}
and \cite[§5]{de_loera_triangulations_2010}. 
\begin{defn}
Given a $n$-dimensional marked polytope $(Q,A)$. 

(1) A \emph{subdivision} of $(Q,A)$ is a family $S=\left\{ (Q_{i},A_{i})\mid i\in I\right\} $
of $n$-dimensional marked polytopes such that
\begin{itemize}
\item $A_{i}\subset A$ and $Q=\bigcup_{i\in I}Q_{i}$.
\item $\forall i,j\in I$, $Q_{i}\cap Q_{j}$ is empty or a common face
of $Q_{i}$ and $Q_{j}$.
\item $\forall i,j\in I$, we have $A_{i}\cap Q_{i}\cap Q_{j}=A_{j}\cap Q_{i}\cap Q_{j}$. 
\end{itemize}
Since $Q_{i}=\mathrm{conv}(A_{i})$, we can simply take $S$ as the
family $\{A_{i}\mid i\in I\}$. Subdivision $T=\{A_{i}\mid i\in I\}$
is called a \emph{triangulation} of $(Q,A)$ if all $Q_{i}$ are $n$-simplex
and $A_{i}$ is the vertex set of $Q_{i}$. 

(2) We say subdivision $S$ \emph{refines} another one $S'$ (denoted
as $S\prec S'$, by no means strictly) if each $A_{i}\in S$ is contained
in some $A^{\prime}_{j}\in S'$. It defines a partial order on the
set of all subdivisions of $(Q,A)$. The unique maximum is the trivial
subdivision $\{A\}$, and the triangulations are the minimum. A subdivision
$S$ is called \emph{full} if $\bigcup_{i\in I}A_{i}=A$, which is
equivalent to that $A_{i}=Q_{i}\cap A$ for all $i$.

(3) Regular subdivisions. For any $f:A\ra\bbr$, its convex envelope
$f^{\vee}$ induces a subdivision $S(f)=\{A_{i}\mid i\in I\}$, which
is defined as follows. Let $\{Q_{i}\}_{i\in I}$ be the maximal subpolytopes
on which $f^{\vee}$ is affine, and we set 
\begin{equation}
A_{i}=\left\{ u\in A\cap Q_{i}\mid f^{\vee}(u)=f(u)\right\} .\label{eq: def S(f)}
\end{equation}
Subdivisions of the form $S(f)$ are called \textit{regular}. By \cite[Prop. 2.2.4]{de_loera_triangulations_2010},
$(Q,A)$ always admits a regular triangulation. See \cite[page 219]{gelfand_discriminants_1994}
for an irregular triangulation. 
\end{defn}

A natural way to compute $\int_{Q}f^{\vee}\mathrm{d}x$ is to refine
$S(f)$ by a triangulation $T$, then integrate $f^{\vee}$ over each
simplex of $T$. This procedure yields the GKZ vectors. 
\begin{defn}
(1) Affine interpolation. Given a triangulation $T$ of $(Q,A)$ and
$f:A\ra\bbr$, we can build a function $\varLambda_{T}(f)$ which
is affine over each simplex of $T$ and equals to $f$ at the simplex
vertices. Obviously, $\varLambda_{T}$ is a linear operator, and by
convexity we have 
\begin{equation}
\varLambda_{T}(f)\ge f^{\vee},\ \forall f:A\ra\bbr.\label{eq: interpolate and envelope}
\end{equation}

(2) GKZ (Gelfand--Kapranov--Zelevinsky) vectors. For a triangulation
$T=\{A_{i}\mid i\in I\}$ of $(Q,A)$, the attached \emph{GKZ vector}
$g_{T}:A\ra\bbr$ is 
\begin{equation}
g_{T}(u)\coloneqq\frac{1}{(n+1)\left|Q\right|}\sum_{i:\,u\in\mathrm{vert}(Q_{i})}\left|Q_{i}\right|,\ \textrm{for }u\in\bigcup_{i\in I}A_{i},\label{eq:def GKZ vec}
\end{equation}
that is the sum of the volumes of simplices with $u$ as a vertex.
If $u\in A$ is not used by $T$, we set $g_{T}(u)=0$. Note that
$g_{T}(u)$ is independent of the choices of Lebesgue measure. The
normalization factor is different from \cite{gelfand_discriminants_1994},
our choice makes $\sum_{u\in A}g_{T}(u)=1$. 
\end{defn}

\begin{prop}
\label{prop: T integral formula}For any triangulation $T$ of $(Q,A)$
and $f:A\ra\bbr$, we have 

(1) $\fint_{Q}\varLambda_{T}(f)=\left\langle f,g_{T}\right\rangle $.

(2) $\fint_{Q}f^{\vee}=\left\langle f,g_{T}\right\rangle \ \Leftrightarrow\ f^{\vee}=\varLambda_{T}(f)\ \Leftrightarrow\ T\prec S(f)$.

(3) Let $T$ go over all triangulations of $(Q,A)$, we have 
\begin{equation}
\fint_{Q}f^{\vee}=\min_{T}\left\langle f,g_{T}\right\rangle .\label{eq: int tent T vertex}
\end{equation}
\end{prop}

\begin{proof}
\textbf{(1)} Note that the integral of $\varLambda_{T}(f)$ over a
simplex $\Delta$ of $T$ is equal to $\left|\Delta\right|$ times
the average of $f$ over the vertices of $\Delta$. Then take the
sum over simplexes, the equality follows. \textbf{(2)} The first $\Leftrightarrow$
follows from (1) and (\ref{eq: interpolate and envelope}). Suppose
$T\prec S(f)$, this means that for each simplex $\Delta$ of $T$,
there is a subpolytope $(Q_{i},A_{i})$ of $S(f)$ such that $\mathrm{vert}(\Delta)\subset A_{i}$.
By definition (\ref{eq: def S(f)}) of $S(f)$, $f^{\vee}$ is affine
over $Q_{i}\supset\Delta$ and $f^{\vee}=f$ on $A_{i}$, thus $f^{\vee}=\varLambda_{T}(f)$
on $\Delta$, so on the whole $Q$. Conversely, if $f^{\vee}=\varLambda_{T}(f)$,
then $f^{\vee}$ is affine over each simplex $\Delta$ of $T$, and
$f^{\vee}=\varLambda_{T}(f)=f$ at the vertices of $\Delta$, these
imply $T\prec S(f)$. \textbf{(3)} $\le$ follows by (1) and (\ref{eq: interpolate and envelope}).
On the other hand, there exists a triangulation $T\prec S(f)$, then
(2) implies (\ref{eq: int tent T vertex}). 
\end{proof}

\begin{defn}
The secondary polytope associated to a marked polytope $(Q,A)$ is
the following convex hull in $\bbr^{A}$, 
\[
\Sigma(A)\coloneqq\mathrm{conv}\left\{ g_{T}\in\bbr^{A}\mid T\textrm{ is triangulation of }(Q,A)\right\} .
\]
\end{defn}

Immediately, (\ref{eq: int tent T vertex}) implies (see \cite[Lemma 1.9 (1.7), Chapter 7]{gelfand_discriminants_1994}): 
\begin{equation}
\fint_{Q}f^{\vee}=\min_{g\in\Sigma(A)}\left\langle f,g\right\rangle ,\ \forall f:A\ra\bbr.\label{eq: Int tent formula}
\end{equation}
This formula is crucial to solving minimization problem (\ref{eq: min M_A}).
Before that, we state the structure theorem of secondary polytopes. 
\begin{thm}[{\cite[§7 Proposition 1.11 \& Theorem 2.4]{gelfand_discriminants_1994}}]
\label{thm: structure of Sigma A} Given a $n$-dimensional marked
polytope $(Q,A)$. 

(1) The affine subspace spanned by $\Sigma(A)$ is 
\begin{equation}
\left\{ g\in\bbr^{A}\Bigm|\sum_{u\in A}g(u)=1,\ \sum_{u\in A}g(u)\cdot\vec{u}=\fint_{Q}\vec{x}\ \mathrm{d}x\right\} .\label{eq:position of Sigma_A}
\end{equation}
So $\dim\Sigma(A)=\left|A\right|-n-1$. 

(2) (Subdivision-face correspondence) For a regular subdivision $S$,
we define 
\[
F(S)\coloneqq\mathrm{conv}\left\{ g_{T}\in\bbr^{A}\mid\textrm{triangulation}\ T\prec S\right\} .
\]
Then the map $S\mapsto F(S)$ is an order-preserving bijection from
the poset of regular subdivisions of $(Q,A)$ to the face poset of
$\Sigma(A)$. The regular triangulations are in one-to-one correspondence
with the vertices of $\Sigma(A)$. 

(3) For any $f\in\bbr^{A}$, $\left\langle f,\cdot\right\rangle |_{\Sigma(A)}$
attains its minimum exactly on face $F\left(S(f)\right)$. For any
regular subdivision $S$, the outer-normal cone of $F(S)$ is 
\[
\left\{ f\in\bbr^{A}\mid\left\langle f,\cdot\right\rangle |_{\Sigma(A)}\textrm{ attains minimum on }F(S)\right\} =\left\{ f\in\bbr^{A}\mid S(f)\succ S\right\} .
\]
\end{thm}

\begin{example}
\label{exa: pentagon}Let $A=\left\{ (0,0),(1,0),(2,0),(1,1),(0,1)\right\} \subset\bbr^{2}$,
$Q=\mathrm{conv}(A)$ is a trapezoid. Its secondary polytope $\Sigma(A)$
is a pentagon, see Figure \ref{fig:pentagon} for the face-subdivision correspondence.
\begin{figure}[htbp]
\centering
\includegraphics[scale=0.3]{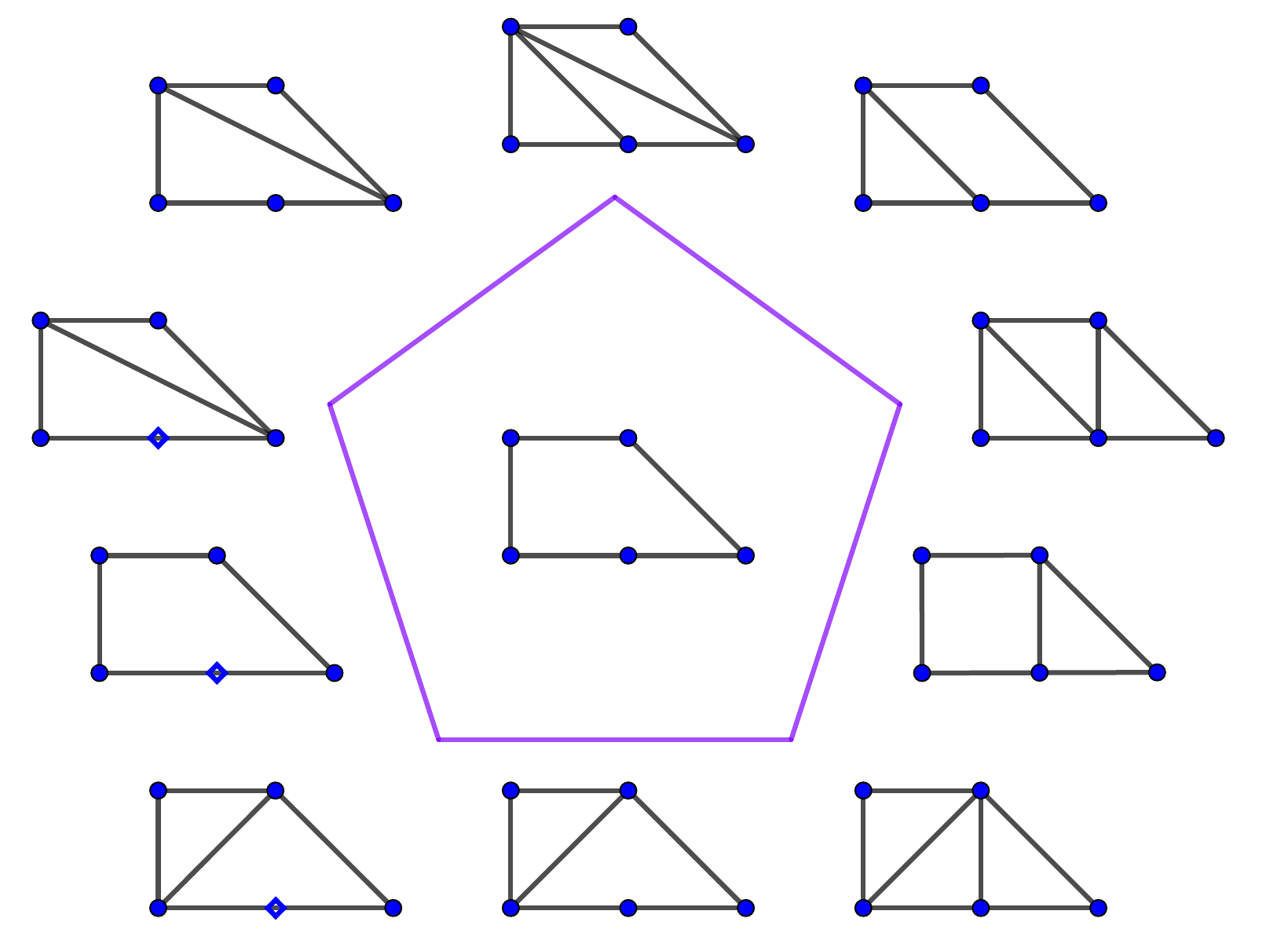}

\caption{\label{fig:pentagon}$\Sigma(A)$ and the corresponding subdivisions. }

\end{figure}
\end{example}

\subsection{\label{subsec: coefficient field}Integral representation of secondary
polytopes}

This section relates to §\ref{subsec:Estimate-the-integrals}. 

There is another gentle description of $\Sigma(A)$ due to \cite{billera_fiber_1992}.
Define a simplex 
\[
\Delta_{A}=\left\{ g\in\mathbb{R}^{A}\mid g\ge0,\sum_{u\in A}g(u)=1\right\} 
\]
in $\mathbb{R}^{A}$. Consider the natural map 
\[
\wp:\Delta_{A}\rightarrow Q,\ g\mapsto\sum_{u\in A}g(u)u.
\]

\begin{defn}
A \textit{coefficient field} $\lam$ on marked polytope $(Q,A)$ is
a Lebesgue-measurable section of $\wp$, which consists of $\left|A\right|$
functions $\lam_{u}(\cdot)\coloneqq\lam(\cdot)(u)$ on $Q$ such that
\[
\lam_{u}(x)\ge0,\ \sum_{u\in A}\lam_{u}(x)=1,\ \sum_{u\in A}\lam_{u}(x)u=x
\]
for almost all $x\in Q$, see Example \ref{exa:model lambda}. The
\textit{generalized GKZ vector} attached to $\lam$ is defined as
\[
g_{\lam}:A\ra\bbr,\ g_{\lam}(u)=\fint_{Q}\lam_{u}(x)\mathrm{d}x.
\]
\end{defn}

A triangulation $T$ of $(Q,A)$ induces a coefficient field $\lam_{T}$
which comes from the unique barycentric coordinates on each simplex.
One may verify that $g_{T}=g_{\lam_{T}}$. By \cite[Theorem 1.12]{gelfand_discriminants_1994},
$\Sigma(A)$ is the collection of all generalized GKZ vectors, namely
\begin{equation}
\Sigma(A)=\left\{ g_{\lam}\mid\text{coefficient field }\lam:Q\rightarrow\Delta_{A}\right\} .\label{eq: SP=00003DgGKZ}
\end{equation}

From this point of view, we can show that the faces of $\Sigma(A)$
are Minkowski sums of smaller secondary polytopes. 
\begin{prop}
\label{prop:Minkowski sum}For a marked polytope $(Q,A)$, consider
a face $F(S)$ of $\Sigma(A)$ corresponding to a regular subdivision
$S=\left\{ (Q_{i},A_{i})\right\} _{i\in I}$. Then we have 
\[
F(S)=\sum_{i\in I}\frac{\left|Q_{i}\right|}{\left|Q\right|}\Sigma(A_{i}),
\]
where the vectors in $\Sigma(A_{i})$ are regarded as vectors in $\mathbb{R}^{A}$
by zero extension. 
\end{prop}

\begin{proof}
We say a coefficient field $\lam$ is \textit{confined} by $S$ if
for any $i\in I$, 
\[
\lam_{u}(\cdot)=0\text{ a.e. on }Q_{i}\text{ for }u\notin A_{i}.
\]
It suffices to show 
\[
F(S)=\left\{ g_{\lam}\mid\lam\text{ is confined by }S\right\} .
\]
Since if $\lam$ is confined by $S$, after a modification on null
set, it induces a field $\lam^{(i)}$ on $(Q_{i},A_{i})$. For any
$u\in A$, 
\[
\left|Q\right|g_{\lam}(u)=\int_{Q}\lam_{u}(x)=\sum_{i\in I}\int_{Q_{i}}\lam^{(i)}_{u}(x)=\sum_{i\in I}\left|Q_{i}\right|g_{\lam^{(i)}}(u).
\]
By (\ref{eq: SP=00003DgGKZ}), $g_{\lam}$ belongs to the Minkowski
sum. The converse direction is similar. 

Take $f\in\bbr^{A}$ such that $S=S(f)$. If $\lam$ is confined by
$S$, since $f^{\vee}$ is affine on $Q_{i}$ and $f^{\vee}=f$ on
$A_{i}$, so 
\[
f^{\vee}(x)=\sum_{u\in A}\lam_{u}(x)f(u),\text{ for almost all }x\in Q.
\]
Integrating both sides, we obtain $\fint_{Q}f^{\vee}=\left\langle g_{\lam},f\right\rangle $.
Hence $\left\langle \cdot,f\right\rangle |_{\Sigma(A)}$ attains the
minimum at $g_{\lam}$. By Theorem \ref{thm: structure of Sigma A}
(3), we know $g_{\lam}\in F(S(f))$. 

Conversely, given $g_{\lam}\in F(S)$, by Jensen inequality, we have
\[
f^{\vee}(x)\le\sum_{u\in A}\lam_{u}(x)f^{\vee}(u)\le\sum_{u\in A}\lam_{u}(x)f(u),\ \forall x\in Q.
\]
Since $\fint_{Q}f^{\vee}=\left\langle g_{\lam},f\right\rangle $,
two $\le$ take equality for $x\in Q\backslash E$, where $E$ is
a null set. Fix $x\in Q^{\circ}_{i}\backslash E$, the equality condition
implies that $f^{\vee}$ is affine over the set 
\[
\mathrm{cvsupp}\lam(x)\coloneqq\mathrm{conv}\left\{ u\in A\mid\lam_{u}(x)>0\right\} 
\]
and $f^{\vee}(u)<f(u)$ only when $\lam_{u}(x)=0$. Since $x\in\mathrm{cvsupp}\lam(x)$,
we have $\mathrm{cvsupp}\lam(x)\subset Q_{i}$. By definition of $S(f)$,
these imply $\lam_{u}(x)=0$ for all $u\notin A_{i}$, thus $\lam$
is confined by $S$. 
\end{proof}

\subsection{Determine Chow-stability by secondary polytopes }

Define subspace 
\[
\bbr^{A}_{0}\coloneqq\left\{ f:A\rightarrow\mathbb{R}\Bigm|\sum_{u\in A}f(u)=0\right\} \subset\bbr^{A}
\]
and orthogonal projection
\[
\pi:\bbr^{A}\ra\bbr^{A}_{0},\ \pi(f)=f-\bar{f},\ \bar{f}\coloneqq\frac{1}{\left|A\right|}\sum_{u\in A}f(u).
\]
By (\ref{eq: Int tent formula}), we rewrite $M_{A}$ (\ref{eq: M_A})
as 
\begin{equation}
M_{A}(f)=-\fint_{Q}\left(\pi f\right)^{\vee}=-\min_{g\in\Sigma(A)}\left\langle \pi f,g\right\rangle =-\min_{h\in\pi\left(\Sigma(A)\right)}\left\langle f,h\right\rangle .\label{eq: M_A as support func}
\end{equation}
The last equality used $\left\langle \pi f,g\right\rangle =\left\langle f,\pi g\right\rangle $.
Now, $M_{A}$ is almost the support function of $\pi\left(\Sigma(A)\right)$.
By the same trick, we can handle the relative Chow-weights. 
\begin{defn}
\label{def: relative Chow-stab}Given a $n$-dimensional marked polytope
$(Q,A)$. 

(1) Let $\mathcal{A}$ be the space of affine functions. Since the
restriction map $\mathcal{A}\rightarrow\bbr^{A}$ is injective, we
can identify $\mathcal{A}$ with the image. Let $R(A)\coloneqq\mathcal{A}^{\perp}$
be the orthogonal complement in $\bbr^{A}$, that is 
\[
R(A)=\left\{ f:A\rightarrow\mathbb{R}\Bigm|\sum_{u\in A}f(u)=0,\ \sum_{u\in A}f(u)\,\overrightarrow{u}=0\right\} .
\]
It records the affine relations between the points in $A$. Let $\Pi:\bbr^{A}\rightarrow R(A)$
be the orthogonal projection. By (\ref{eq:position of Sigma_A}),
the projection image of $\Sigma(A)$ to $\mathcal{A}$ is a single
point, denoted by $\ell_{A}$, which is the unique affine function
satisfying 
\[
\fint_{Q}\alpha=\sum_{u\in A}\alpha(u)\ell_{A}(u),\ \forall\text{ affine }\alpha.
\]

\begin{figure}
\centering
\includegraphics[viewport=0bp 0bp 1150bp 670bp,clip,scale=0.3]{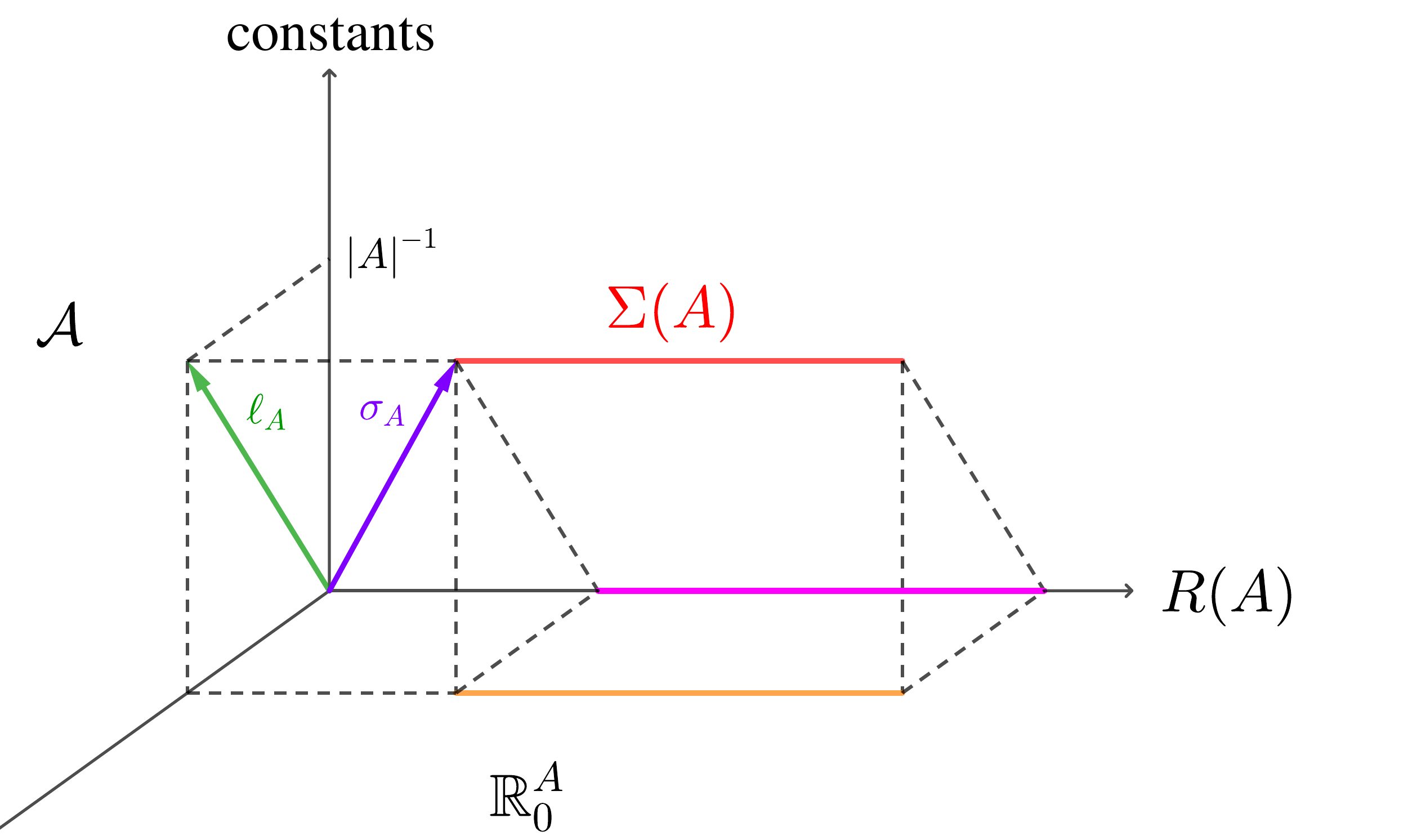}\caption{\label{fig: Location of second poly}Position of $\Sigma(A)$.}
\end{figure}

(2) We define the \emph{relative} Chow-weight functional $M^{\mathrm{rel}}_{A}:\bbr^{A}\rightarrow\mathbb{R}$
is 
\begin{equation}
M^{\mathrm{rel}}_{A}(f)\coloneqq-\fint_{Q}\left(\Pi f\right)^{\vee}=-\fint_{Q}f^{\vee}+\left\langle f,\ell_{A}\right\rangle .
\end{equation}
It vanishes on affine functions. We say $(Q,A)$ is \emph{relatively
Chow-semistable} if $M^{\mathrm{rel}}_{A}(f)\geq0$ for all $f\in\bbr^{A}$.
Otherwise, we call $(Q,A)$ is \emph{relatively Chow-unstable}. 
\end{defn}

\begin{rem}
When $(Q,A)=(P,P_{k})$ and $f$ takes rational numbers, $M^{\mathrm{rel}}_{A}(f)$
is the genuinely relative Chow-weight for toric variety $X_{P}\hookrightarrow\mathbb{P}R^{*}_{k}$
relative to the automorphism group $T_{\bbc}$, in the sense of \cite{szekelyhidi_extremal_2006}.
Note that our ``relative Chow-semistability'' differs from that
in \cite[Definition 5.3]{ono_algebro-geometric_2013}, which involves
the extremal affine function. 
\end{rem}

By (\ref{eq: Int tent formula}), we also have support function formula
\begin{equation}
M^{\mathrm{rel}}_{A}(f)=-\min_{g\in\Sigma(A)}\left\langle \Pi f,g\right\rangle =-\min_{h\in\Pi\left(\Sigma(A)\right)}\left\langle f,h\right\rangle .\label{eq: rel M_A as support}
\end{equation}
A consequence is that the (relative) Chow-semistability is determined
by the location of $\Sigma(A)$. 
\begin{thm}
\label{thm: when semistable}(1) Marked polytope $(Q,A)$ is Chow-semistable
(Definition \ref{def: Chow-weight}) if and only if the constant function
$\left|A\right|^{-1}\in\Sigma(A)$, equivalently $\ell_{A}\equiv\left|A\right|^{-1}$.
In addition, $(Q,A)$ is relatively Chow-semistable (Definition \ref{def: relative Chow-stab})
if and only if $\ell_{A}\in\Sigma(A)$. Refer to Figure  \ref{fig: stable}. 

(2) For toric variety $(X,L)$ given by a lattice polytope $P$, the
projective embedding $\iota_{k}(X)$ induced by $\left|kL\right|$
is Chow-semistable if and only if $\left|P_{k}\right|^{-1}\in\Sigma(P_{k})$,
where Chow-semistability is with regard to the whole linear group
action. 
\begin{figure}[h]
\centering
\includegraphics[viewport=150bp 200bp 1140bp 650bp,clip,scale=0.3]{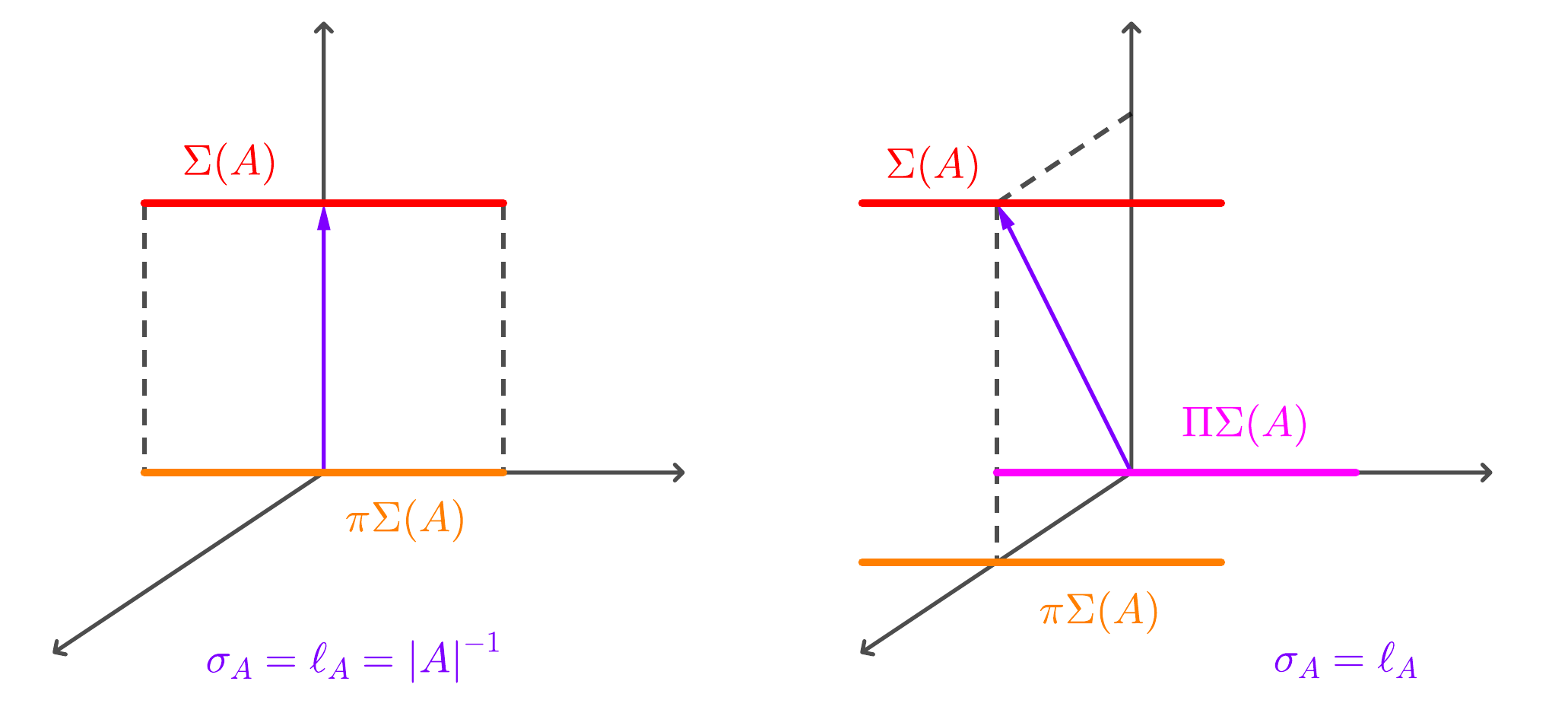}\caption{\label{fig: stable}Left: Chow-semistable. Right: Chow-unstable but relatively Chow-semistable. Refer to Figure \ref{fig: Location of second poly}. }
\end{figure}
\end{thm}

\begin{proof}
(1) follows from (\ref{eq: M_A as support func}) and (\ref{eq: rel M_A as support})
respectively. 

(2) By \cite[Corollary 2.3, 2.7]{lee_asymptotic_2019} or \cite[Corollary 25]{yao_maximal_2025},
Chow-semistability is equivalent to the semistability with respect
to the action of maximal torus. The latter one is the same as the
Chow-semistability of $(P,P_{k})$ in the sense of Definition \ref{def: Chow-weight}. 
\end{proof}

\begin{rem}
If $\left|P_{k}\right|^{-1}\in\Sigma(P_{k})$, by (\ref{eq:position of Sigma_A}),
we obtain a barycenter equality
\[
\frac{1}{\left|P_{k}\right|}\sum_{u\in P_{k}}\vec{u}=\frac{1}{V_{P}}\int_{P}\vec{x}\ \mathrm{d}x.
\]
This is a necessary condition for $X_{P}\hookrightarrow\mathbb{P}R^{*}_{k}$
being Chow-semistable, which was discovered by Ono \cite{ono_necessary_2011}.
Hence Chow-semistability is a very stringent property for toric varieties.
In contrast, relative Chow-semistability is far easier to satisfy.
\end{rem}

\begin{example}
Marked polytope $(Q,A)$ in Example \ref{exa: pentagon} is relatively
Chow-semistable, since 
\[
\ell_{A}=\left(\frac{5}{27},\frac{5}{27},\frac{5}{27},\frac{2}{9},\frac{2}{9}\right)\in\Sigma(A),
\]
which is the restriction of affine function $\frac{1}{27}(y+5)$. 
\end{example}

\subsection{Minimize the normalized Chow-weights }

Suppose $(Q,A)$ is Chow-unstable, we determine the minimizers of
$\bar{M}_{A}$ (\ref{eq: min M_A}). It follows from the basic fact
below. 
\begin{lem}
\label{lem: min-max}Let $K$ be a compact convex subset in an inner-product
space $W$. Suppose $0\notin K$, let $x^{*}\in K$ be the unique
point on $K$ with minimal norm. 

(1) $\left\langle x^{*},x\right\rangle \geq\left\langle x^{*},x^{*}\right\rangle $
for all $x\in K$. Moreover, $x^{*}$ is the unique point on $K$
satisfying this property. 

(2) Let $H(x)=\min_{y\in K}\left\langle x,y\right\rangle $, then
we have 
\begin{equation}
\max_{x\in W\backslash\{0\}}\frac{H(x)}{\left\Vert x\right\Vert }=\min_{y\in K}\left\Vert y\right\Vert =\left\Vert x^{*}\right\Vert .\label{eq: min-max}
\end{equation}
The maximizers for the left-hand side are $\left\{ ax^{*}\mid a>0\right\} $. 
\end{lem}

\begin{proof}
(1) For any $x\in K$, let $x_{t}\coloneqq tx+(1-t)x^{*}\in K$ for
$t\in[0,1]$. Since $x^{*}$ has minimal norm, we have
\[
0\leq\frac{d}{dt}\left\Vert x_{t}\right\Vert ^{2}|_{t=0}=2\left\langle x-x^{*},x^{*}\right\rangle .
\]
Conversely, if $y\in K$ satisfies this condition. It implies $\left\Vert y\right\Vert ^{2}\leq\left\langle y,x^{*}\right\rangle \leq\left\Vert y\right\Vert \left\Vert x^{*}\right\Vert $,
so $\left\Vert y\right\Vert \leq\left\Vert x^{*}\right\Vert $. Since
the minimal-norm point is unique, we have $y=x^{*}$. 

(2) For any $x\in W\backslash\{0\}$ and $y\in K$, Cauchy--Schwarz
inequality $\left\langle x,y\right\rangle \le\left\Vert x\right\Vert \left\Vert y\right\Vert $
implies 
\[
\frac{\left\langle x,y\right\rangle }{\left\Vert x\right\Vert }\le\left\Vert y\right\Vert .
\]
Take $\inf_{y\in K}$ on both sides, we obtain $\frac{H(x)}{\left\Vert x\right\Vert }\le\min_{y\in K}\left\Vert y\right\Vert =\left\Vert x^{*}\right\Vert $.
Hence $\le$ side of (\ref{eq: min-max}) follows. On the other hand,
by (1), we have 
\[
\frac{H(x^{*})}{\left\Vert x^{*}\right\Vert }=\frac{\min_{y\in K}\left\langle x^{*},y\right\rangle }{\left\Vert x^{*}\right\Vert }=\frac{\left\langle x^{*},x^{*}\right\rangle }{\left\Vert x^{*}\right\Vert }=\left\Vert x^{*}\right\Vert .
\]
Hence (\ref{eq: min-max}) holds and $x^{*}$ is a maximizer of $\bar{H}(x)\coloneqq H(x)/\left\Vert x\right\Vert $. 

Uniqueness. Suppose $x_{0},x_{1}$ are both maximizers for $\bar{H}$.
Scaling them such that $\left\Vert x_{0}\right\Vert =\left\Vert x_{1}\right\Vert =1$,
then $H(x_{0})=H(x_{1})=\max\bar{H}>0$. Consider $x_{m}=(x_{0}+x_{1})/2\neq0$.
If $x_{0}\neq x_{1}$, then we have $\left\Vert x_{m}\right\Vert <1$.
Since $H$ is concave, 
\[
H(x_{m})\ge\frac{1}{2}\left(H(x_{0})+H(x_{1})\right)=\max\bar{H}.
\]
This follows that 
\[
\bar{H}(x_{m})=\frac{H(x_{m})}{\left\Vert x_{m}\right\Vert }\ge\frac{\max\bar{H}}{\left\Vert x_{m}\right\Vert }>\max\bar{H}.
\]
That is absurd, so $x_{0}=x_{1}$. 
\end{proof}
\begin{rem}\label{rem-quadmin}
    The same proof shows an essentially equivalent statement: for any $\lambda>0$, the functional
    \begin{equation*}
       x\mapsto\;\; -\min_{y\in K}\langle x, y\rangle+\frac{1}{2\lambda}\|x\|^2,
    \end{equation*}
    obtains the minimum at the point $\lambda x^*\in \lambda K$. 
\end{rem}

\begin{defn}[shortest GKZ vector]
 For a marked polytope $(Q,A)$, since $0\notin\Sigma(A)$, we define
the \emph{shortest GKZ vector} $\sigma_{A}\in\Sigma(A)$ is the unique
least-norm point on $\Sigma(A)$. By the subdivision-face correspondence
(Theorem \ref{thm: structure of Sigma A}), there is a unique regular
subdivision $S_{A}$ of $(Q,A)$ such that $F(S_{A})$ is the minimal
face containing $\sigma_{A}$. Refer to Figure \ref{fig: Location of second poly}. 
\end{defn}

\begin{rem}
In general, $\sigma_{A}$ is a convex combination of GKZ vectors,
or a generalized GKZ vector by (\ref{eq: SP=00003DgGKZ}). Refer to
Example \ref{exa: sGKZ}. 
\end{rem}

Now, we determine the maximal Chow-destabilizers by shortest GKZ vectors. 
\begin{thm}
\label{thm: Chow destabi} Suppose marked polytope $(Q,A)$ is Chow-unstable
(i.e. $\inf M_{A}<0$), then we have 
\[
\min\left\{ \frac{M_{A}(f)}{\left\Vert f\right\Vert }\Bigm|0\not\equiv f:A\ra\bbr\right\} =-\left\Vert \sigma_{A}-\left|A\right|^{-1}\right\Vert .
\]
The minimizers are $a\left(\sigma_{A}-\left|A\right|^{-1}\right)$,
$\forall a>0$. Applying to the case $(Q,A)=(P,P_{k})$, we obtain
the part (2) of Theorem \ref{thm: Chow-desta INTRO}. 
\end{thm}

\begin{proof}
Apply Lemma \ref{lem: min-max} to the formula (\ref{eq: M_A as support func})
of $M_{A}$. Take $K=\pi\left(\Sigma(A)\right)$ and the least-norm
point is $\sigma_{A}-\left|A\right|^{-1}$. 

For part (2) of Theorem \ref{thm: Chow-desta INTRO}. Suppose $(X,L)$
is the toric variety induced by $P$ and $\iota_{k}(X)$ is Chow-unstable.
By the symmetry principle \cite[Corollary 2]{yao_maximal_2025}, the
minimizers of (\ref{eq: min Chow  INTRO}) belong to $\mathcal{N}(R_{k})^{T}$
over which the Chow-weight has formula (\ref{eq: Chow weight formula}).
Then the statement follows. 
\end{proof}

\subsection{\label{subsec:Properties-of-shortest GKZ}Properties of shortest
GKZ vectors}

We establish two properties of $\sigma_{A}$, the convexity and positivity. 
\begin{thm}
\label{thm: convex of mGKZ} Let $\sigma_{A}:A\ra\bbr$ be the shortest
GKZ vector associated to marked polytope $(Q,A)$, then $\sigma^{\vee}_{A}|_{A}=\sigma_{A}$,
which means that $\sigma_{A}$ is convex. 
\end{thm}

\begin{proof}
Consider functional: 
\[
H(f)\coloneqq\fint_{Q}f^{\vee}=\min_{g\in\Sigma(A)}\left\langle f,g\right\rangle ,\ f:A\ra\bbr.
\]
By Lemma \ref{lem: min-max}, maximizers of $\bar{H}(f)\coloneqq H(f)/\left\Vert f\right\Vert $
are $a\sigma_{A},\ \forall a>0$. 

Let $\rho\coloneqq\sigma^{\vee}_{A}|_{A}$, we have $0\le\rho\le\sigma_{A}$,
thus $\left\Vert \rho\right\Vert \le\left\Vert \sigma_{A}\right\Vert $.
By Proposition \ref{prop:FS vs SN} (2), we have $\rho^{\vee}=\sigma^{\vee}_{A}$.
These imply 
\[
\frac{H(\rho)}{\left\Vert \rho\right\Vert }=\frac{H(\sigma_{A})}{\left\Vert \rho\right\Vert }\ge\frac{H(\sigma_{A})}{\left\Vert \sigma_{A}\right\Vert }.
\]
Hence $\rho$ is also a maximizer of $\bar{H}$. We have $\rho=a\sigma_{A}$
for some $a>0$. But $\rho^{\vee}=\sigma^{\vee}_{A}$ implies $a=1$,
so $\rho=\sigma_{A}$. 
\end{proof}

It is obvious that $\sigma_{A}\ge0$ from its definition, further
we can show that $\sigma_{A}>0$ at every point. 
\begin{thm}
\label{thm: positive of mGKZ}Let $(Q,A)$ be a marked polytope. 

(1) Let $S_{A}$ be the unique regular subdivision such that \textup{$F(S_{A})$}
is the minimal face of $\Sigma(A)$ containing $\sigma_{A}$, then
we have $S(\sigma_{A})\succ S_{A}$. We cannot exclude the possibility
that $S(\sigma_{A})\neq S_{A}$. Conversely, if $g\in\Sigma(A)$ satisfies
$g\in F\left(S(g)\right)$, then $g=\sigma_{A}$. 

(2) We have $\sigma_{A}>0$ at every point in $A$. Subdivision $S_{A}$
must be full, so is $S(\sigma_{A})$, which implies $\sigma^{\vee}_{A}|_{A}=\sigma_{A}$. 
\end{thm}

\begin{proof}
(1) By Lemma \ref{lem: min-max} (1), function $\left\langle \sigma_{A},\cdot\right\rangle |_{\Sigma(A)}$
attains its minimum at $\sigma_{A}$, so on the whole face $F(S_{A})$,
since $\sigma_{A}$ is contained in the relative interior of $F(S_{A})$.
On the other hand, Theorem \ref{thm: structure of Sigma A} (3) says
that the minimizer set is face $F\left(S(\sigma_{A})\right)$. Hence
$F\left(S(\sigma_{A})\right)\supset F(S_{A})$, so $S(\sigma_{A})\succ S_{A}$.
For the last statement, $g\in F\left(S(g)\right)$ implies that $\left\langle g,\cdot\right\rangle |_{\Sigma(A)}$
attains its minimum at $g$, by Lemma \ref{lem: min-max} (1), we
can conclude $g=\sigma_{A}$. 

(2) First we handle a simple case, namely $\sigma_{A}$ is a vertex
of $\Sigma(A)$, then $\sigma_{A}=g_{T}$ for a regular triangulation
$T$ and $S_{A}=T$. By (1), we have $S(\sigma_{A})\succ T$. Use
Proposition \ref{prop: T integral formula} (2), we have 
\[
\sigma^{\vee}_{A}=\varLambda_{T}(\sigma_{A})=\varLambda_{T}(g_{T}).
\]
By definition of GKZ vector, $g_{T}>0$ at every vertex of $T$, thus
$\varLambda_{T}(g_{T})>0$ on $Q$, i.e. $\sigma^{\vee}_{A}>0$. Since
$\sigma_{A}\ge\sigma^{\vee}_{A}|_{A}$, it implies $\sigma_{A}>0$
at every point. 

General case. Since $\sigma_{A}\in F(S_{A})$, there exists triangulations
$T_{i}\prec S_{A}$ ($1\le i\le p$) such that 
\[
\sigma_{A}=\sum^{p}_{i=1}\lambda_{i}g_{T_{i}},\ \sum^{p}_{i=1}\lambda_{i}=1,\ \lam_{i}>0.
\]
Since $S(\sigma_{A})\succ S_{A}\succ T_{j}$, by Proposition \ref{prop: T integral formula}
(2), we have 
\[
\sigma^{\vee}_{A}=\varLambda_{T_{j}}(\sigma_{A}),\ \forall1\leq j\leq p.
\]
Since operator $\varLambda_{T_{j}}$ is linear, for each $j$ we have
\[
\varLambda_{T_{j}}(\sigma_{A})=\sum^{p}_{i=1}\lambda_{i}\varLambda_{T_{j}}(g_{T_{i}})\ge\lambda_{j}\varLambda_{T_{j}}(g_{T_{j}}),
\]
where we used $\varLambda_{T_{j}}(g_{T_{i}})\ge0$. Since $\varLambda_{T_{j}}(g_{T_{j}})>0$
on $Q$, so $\sigma^{\vee}_{A}=\varLambda_{T_{j}}(\sigma_{A})>0$.
Since $\sigma_{A}\ge\sigma^{\vee}_{A}|_{A}$, we obtain $\sigma_{A}>0$
at every point. 

Next we show $S_{A}$ is full. For any $a\in A$, we have $0<\sigma_{A}(a)=\sum^{p}_{i=1}\lambda_{i}g_{T_{i}}(a)$,
so there exists $j$ such that $g_{T_{j}}(a)>0$. By definition of
GKZ vector, it means that $a$ is a vertex of $T_{j}$. Since $T_{j}\prec S_{A}$,
$a$ is also a mark point of $S_{A}$, thus $S_{A}$ is full. Then
$S(\sigma_{A})$ is also full since $S(\sigma_{A})\succ S_{A}$. 
\end{proof}

\begin{rem}
Notice that $\sigma_{A}>0$ indicates to us how complicated $\sigma_{A}$
is. Represent $\sigma_{A}=\sum^{p}_{i=1}\lam_{i}g_{T_{i}}$ ($\sum\lam_{i}=1,\lam_{i}>0$)
by some triangulations $T_{i}$ of $(Q,A)$, then $\sigma_{A}>0$
implies that every point in $A$ is a vertex of some $T_{i}$. Furthermore,
if $\sigma_{A}=g_{T}$ is a vertex of $\Sigma(A)$, then $T$ must
use all the points in $A$. 
\end{rem}

\begin{example}
\label{exa: sGKZ} 
Consider marked polytope: 
\begin{figure}[h]
\centering
\includegraphics[viewport=200bp 120bp 800bp 620bp,clip,scale=0.3]{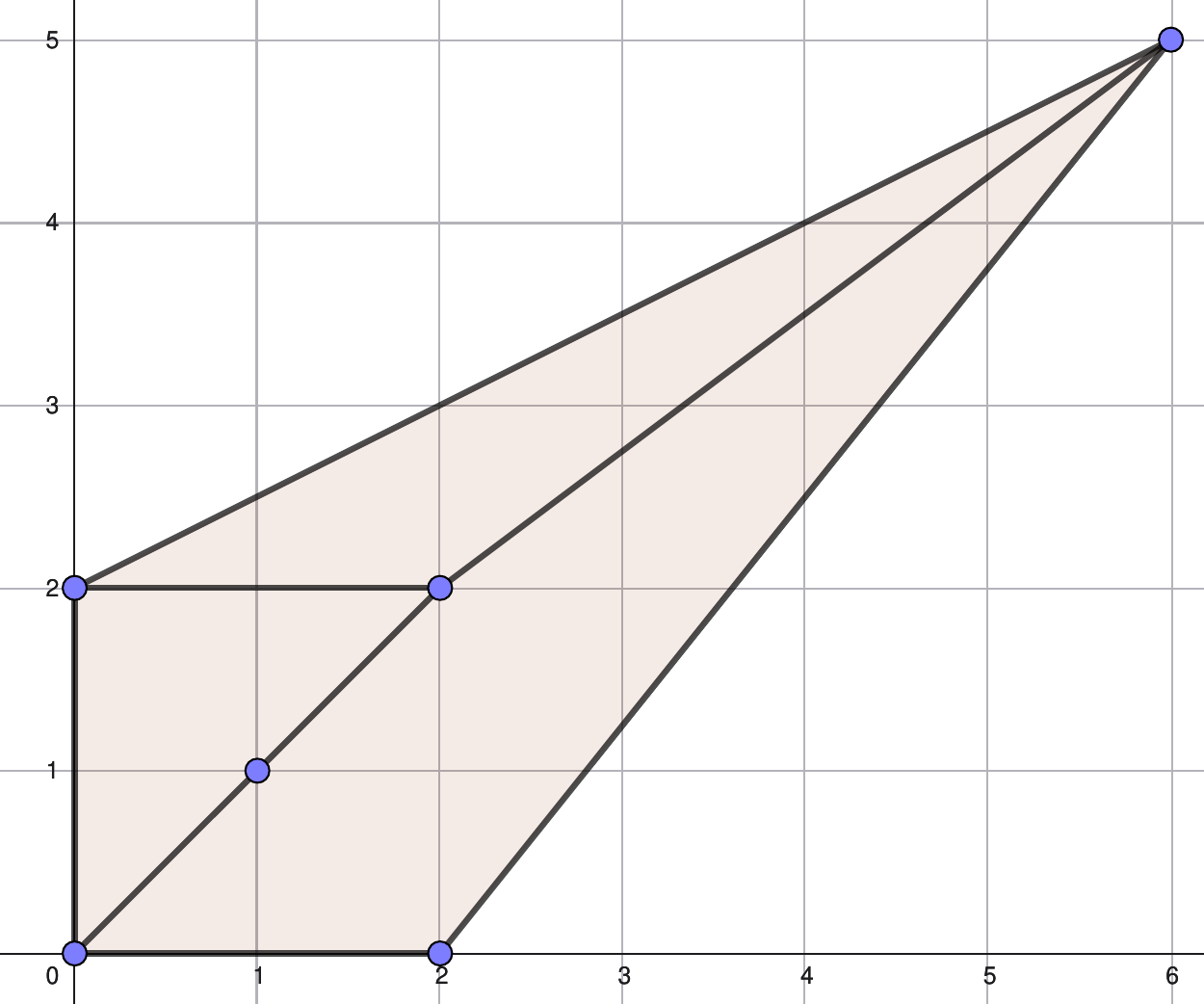}

\caption{\label{fig: SA}subdivision $S(\sigma_{A})=S_{A}$}
\end{figure}
\[
A=\left\{ (0,0),(0,2),(2,0),(2,2),(1,1),(6,5)\right\} \subset\bbr^{2},\ Q=\mathrm{conv}(A).
\]
There is a SageMath function \cite{whitney_sagemath_nodate} can output
$\Sigma(A)$. It turns out $\Sigma(A)$ is $3$-dimensional and has
$7$ facets, $15$ edges and $10$ vertices, see its \href{https://yeeyoe.github.io/mGKZ-example/}{ 3Dview}. 

Use the program $\mathtt{sGKZ}$ \cite{yao_program_2025}, which is
based on the algorithm in Appendix \ref{Appendix:Algorithm}, we obtain
\begin{align*}
\sigma_{A} & =\left(\frac{15}{121},\frac{5}{33},\frac{493}{3025},\frac{1511}{9075},\frac{1318}{9075},\frac{2267}{9075}\right)\\
 & \approx\left(0.1240,\,0.1515,\,0.1630,\,0.1665,\,0.1452,\,0.2498\right).
\end{align*}
It shows that $(Q,A)$ is relatively Chow-unstable, and $\sigma_{A}$
lies in the relative interior of a facet. In this example, we have
$S(\sigma_{A})=S_{A}$. Figure \ref{fig: SA} shows the subdivision $S_{A}$.
\end{example}

\section{\label{sec:4}Optimal test functions for K-stability}

\subsection{\label{subsec:Szekelyhidi's-optimal-test} Szekelyhidi's optimal test function}

We recap the optimal test-function for K-unstable toric varieties
in \cite{szekelyhidi_optimal_2008}. For an introduction to K-stability
of toric varieties, refer to \cite{donaldson_scalar_2002,szekelyhidi_extremal_2006,wang_k-stability_2014}. 

Let $P\subset M_{\bbr}$ be a lattice polytope. Given a rational piecewise
affine convex function $\eta$ over $P$, one can build a toric test-configuration
for $(X_{P},L_{P})$, see \cite{donaldson_scalar_2002}. Donaldson
shows that the attached DF-invariant is equal to 
\begin{equation}\label{eq-defMP}
\mathbb{M}_{P}(\eta)\coloneqq\int_{\partial P}\eta\ \mathrm{d}\sigma-\int_{P}\eta\bar{S}\ \mathrm{d}x,\ \bar{S}=\frac{V_{\partial P,\mathrm{d}\sigma}}{V_{P}},
\end{equation}
see (\ref{eq:EM}) for the definition of $\mathrm{d}\sigma$. When
$\eta$ is affine, it recovers the classical Futaki invariant with
respect to the automorphism torus action. 

Since we focus on the unstable case, we only recall the K-semistability
and its relative version. The relative DF-invariant is defined as
\begin{equation}
\mathbb{M}^{\mathrm{rel}}_{P}(\eta)\coloneqq\int_{\partial P}\eta\ \mathrm{d}\sigma-\int_{P}\eta\ell_{P}\ \mathrm{d}x,
\end{equation}
where $\ell_{P}$ is the unique affine function such that $\mathbb{M}^{\mathrm{rel}}_{P}$
vanishes on all affine functions. Conventionally, $\ell_{P}$ is called
the \textit{extremal affine function}, since it can give the Hamiltonian
function for the extremal vector field.
\begin{defn}[K-semistabilities]
 Given a lattice polytope $P\subset M_{\mathbb{R}}$. 

(1) we say $P$ is K-semistable if $\mathbb{M}_{P}(\eta)\geq0$ for
all $\eta\in\mathcal{C}(P)$. Otherwise, we say $P$ is K-unstable. 

(2) we say $P$ is relatively K-semistable if $\mathbb{M}^{\mathrm{rel}}_{P}(\eta)\geq0$
for all $\eta\in\mathcal{C}(P)$. Otherwise, we say $P$ is relatively
K-unstable. 
\end{defn}

\begin{example}
Using a random search program in \cite{yao_program_2025}, we have
found many unstable non-Delzant polygons, Figure \ref{fig:A-relatively-K-unstable} is an example.  
\begin{figure}[h]
\begin{centering}
\includegraphics[viewport=130bp 155bp 720bp 435bp,clip,scale=0.4]{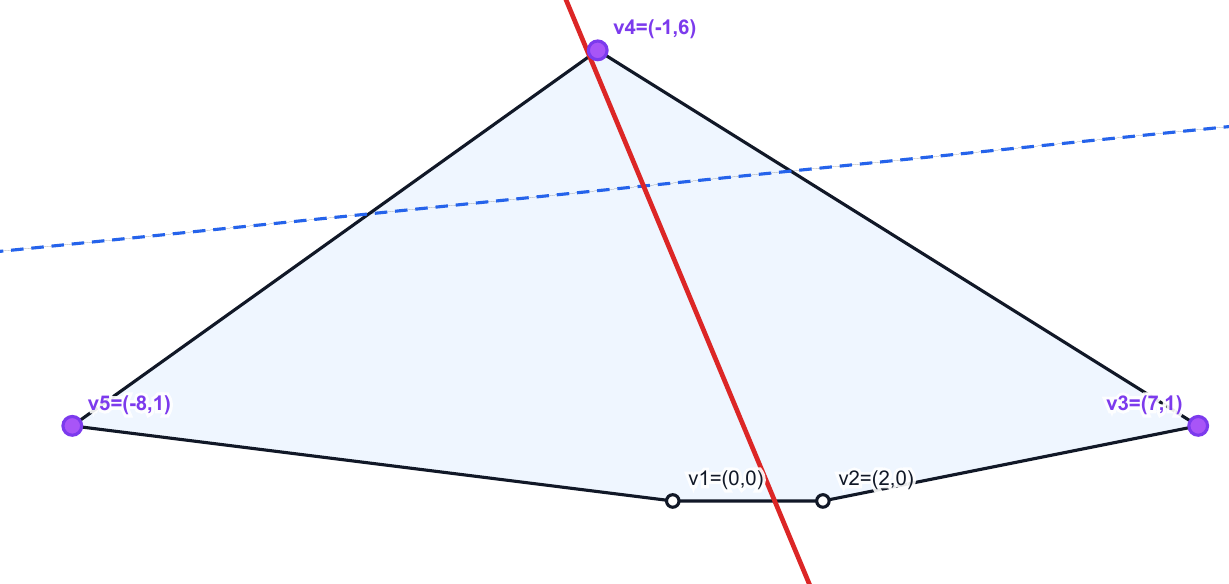}
\par\end{centering}
\caption{\label{fig:A-relatively-K-unstable}A relatively K-unstable pentagon. }
\end{figure}
The purple vertices are singular. The dashed line is
$\{\ell_{P}=0\}$. The red line is $\{\eta=0\}$, where 
\begin{align*}
\ell_{P} & =\frac{2509191}{8585347}+\frac{60144}{8585347}x-\frac{592758}{8585347}y,\\
\eta & =\max\left\{ \frac{9}{10}x+\frac{3}{8}y-\frac{11}{9},\;0\right\} \text{ such that }\mathbb{M}^{\mathrm{rel}}_{P}(\eta)<0.
\end{align*}
\end{example}

In the K-unstable case, Székelyhidi \cite{szekelyhidi_optimal_2008}
showed that the normalized invariant 
\[
\bar{\mathbb{M}}_{P}(\eta)\coloneqq\frac{\mathbb{M}_{P}(\eta)}{\left\Vert \eta\right\Vert _{L^{2}}}
\]
admits a minimizer in an enlarged space of convex functions, which
is defined as follows. Let $P^{*}$ equals to $P$ excluding all the
faces with codimension $\geq2$, which is also a convex set. Donaldson
introduced space 
\begin{equation}
\mathcal{C}_{1}(P)=\left\{ \eta:P^{*}\ra\bbr\text{ continuous convex }\Big|\int_{\partial P}\eta\,\mathrm{d}\sigma<\infty\right\} ,\label{eq:space C1}
\end{equation}
whose elements can be extended to lsc convex functions on $P$ (may
take $+\infty$ on $P\backslash P^{*}$). By \cite[Lemma 5.1.3]{donaldson_scalar_2002},
they are integrable over $P$, thus $\mathbb{M}_{P}$ is well-defined
on $\mathcal{C}_{1}(P)$. 
\begin{thm}[\cite{szekelyhidi_optimal_2008}]
\label{thm: Gabor paper} Given a lattice polytope $P\subset M_{\mathbb{R}}$. 

(1) Suppose $P$ is K-unstable, then 
\[
\inf\left\{ \bar{\mathbb{M}}_{P}(\eta)\mid0\neq\eta\in\mathcal{C}_{1}(P)\cap L^{2}(P)\right\} =-\left\Vert \Theta_{P}-\bar{S}\right\Vert _{L^2}
\]
admits a unique (up to scaling) minimizer $\bar{S}-\Theta_{P}\in\mathcal{C}_{1}(P)\cap L^{2}(P)$.
We call concave function $\Theta_{P}$ the optimal test-function for
$P$ (denoted by $B$ in \cite{szekelyhidi_optimal_2008}), which
satisfies variational inequality: 
\begin{equation}
\int_{\partial P}\eta\ge\int_{P}\eta\Theta_{P},\ \forall\eta\in\mathcal{C}_{1}(P)\cap L^{2}(P)\ \textrm{and}\ \int_{\partial P}\Theta_{P}=\int_{P}\Theta^{2}_{P}.\label{eq: property of Theta}
\end{equation}
When $P$ is K-semistable, we define $\Theta_{P}=\bar{S}$, above
conditions still hold. 

When $P$ is relatively K-semistable, we have $\Theta_{P}=\ell_{P}$.
If $P$ is relatively K-unstable, $\ell_{P}-\Theta_{P}$ is the unique
minimizer of $\mathbb{M}^{\mathrm{rel}}_{P}(\eta)/\left\Vert \eta\right\Vert $. 

(2) Consider the space of ``scalar curvatures'' 
\begin{equation}
\mathcal{S}(P)=\left\{ \theta\in L^{2}(P)\ \Big|\int_{\partial P}\eta\ge\int_{P}\eta\theta,\ \forall\eta\in\mathcal{C}_{1}(P)\cap L^{2}(P)\right\} ,
\end{equation}
then $\Theta_{P}$ is the unique element in $\mathcal{S}(P)$ with
minimal $L^{2}$-norm. Recall that if $u:P\ra\bbr$ is the symplectic
potential of an invariant Kähler metric, its scalar curvature is $S(u)=-\sum_{i,j}\partial_{i}\partial_{j}u^{ij}\in\mathcal{S}(P)$.

(3) Consider the Calabi flow on $X_{P}$ with invariant initial data,
which can be reduced into evolution equation on $P$: 
\[
\dot{u}_{t}=-S(u_{t})=\sum_{i,j}\partial_{i}\partial_{j}u^{ij}_{t}.
\]
Assume it admits long-time solutions, then $S(u_{t})\stackrel{L^{2}}{\rightarrow}\Theta_{P}$
as $t\rightarrow\infty$. 

(4) (canonical semi-stable decomposition) Assume $P$ is relatively
K-unstable and $\Theta_{P}$ is piecewise affine. Let $Q\subset P$
be a maximal subpolytope over which $\Theta_{P}$ is affine, then
$Q$ is relatively K-semistable in the following sense: 
\begin{equation}
\int_{\partial Q\cap\partial P}\eta\ \mathrm{d}\sigma\ge\int_{Q}\eta\Theta_{P}\ \mathrm{d}x,\ \forall\eta\in\mathcal{C}(Q).\label{eq: semi-decom}
\end{equation}
\end{thm}

\begin{rem}[no island]
 Under the assumptions in (4), we can deduce that each subpolytope
$Q$ must have a facet which is entirely contained in $\partial P$.
If not, we have $\left|\partial Q\cap\partial P\right|_{\mathrm{d}\sigma}=0$.
Take $\eta$ in (\ref{eq: semi-decom}) to be the affine functions
extending $\pm\Theta|_{Q}$, it implies 
\[
\int_{Q}\Theta^{2}=\int_{\partial Q\cap\partial P}\Theta\ \mathrm{d}\sigma=0.
\]
So $\Theta\equiv0$ on $Q$. Since $\Theta$ is concave, this implies
$\Theta\leq0$ on $P$, but this contradicts with $\int_{P}\Theta=V_{\partial P,\mathrm{d}\sigma}>0$,
which follows from (\ref{eq: property of Theta}). 
\end{rem}

\begin{example}[A potential example for semistable decomposition]
\begin{figure}[H]
\includegraphics[clip,scale=0.4]{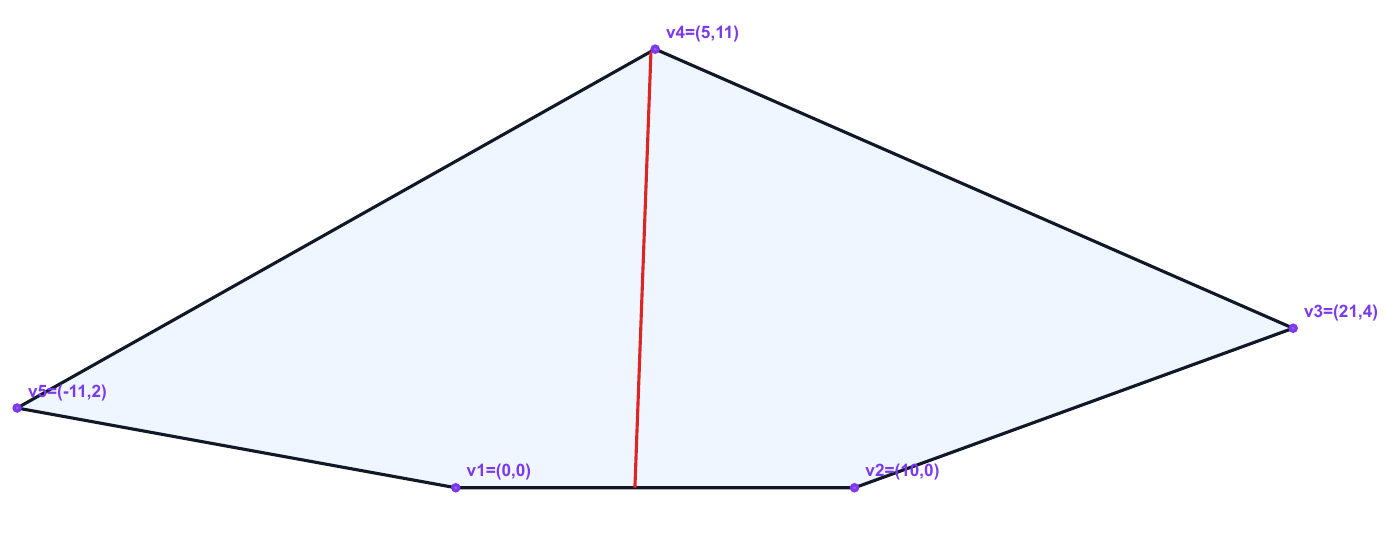}

\caption{\label{fig: exam decompose}The red line does not pass the top vertex! }
\end{figure}
See Figure (\ref{fig: exam decompose}). First we triangulate $P$
into many pieces, then minimize $\Psi$ among the piecewise affine
convex functions. The minimizer shows the rough location of the crease
line whose endpoints are $(10s,0)$ and $tv_{5}+(1-t)v_{4}$ $(0\le s,t\le1)$.
We have $\ell_{Q_{1}}=\ell_{Q_{2}}$ at two endpoints. Solve this
equation, we obtain 
\[
s\approx0.44922255,\ t\approx0.00652993,
\]
which are algebraic numbers whose minimal polynomials have degree
$28$. The semistability of $Q_{1},Q_{2}$ has only passed a numerical
test. Once verified, we can conclude $\Theta=\min\{\ell_{Q_{1}},\ell_{Q_{2}}\}$. 
\end{example}

\subsection{\texorpdfstring{More properties of $\Theta_{P}$}{More properties of Theta\_P}}

A central question is whether $\Theta$ is piecewise affine. The answer
is still unknown. A closely related result is Theorem \ref{thm: Li Zhou}.
In this section, we give some properties about $\Theta$. First, we
re-characterize $\Theta$ as the unique minimizer of a simpler convex
functional, which is inspired by the work \cite{inoue_entropies_2022-1}
of Inoue. 
\begin{prop}
\label{prop: recharac Theta}Convex function $-\Theta_{P}$ is the
unique minimizer of functional 
\begin{equation}
\Psi(\eta)\coloneqq\frac{1}{2}\int_{P}\eta^{2}\mathrm{d}x+\int_{\partial P}\eta\ \mathrm{d}\sigma,\ \eta\in\mathcal{C}_{1}(P)\cap L^{2}(P).\label{eq: Inoue functional}
\end{equation}
\end{prop}

\begin{proof}
For any $\eta\in\mathcal{C}_{1}(P)\cap L^{2}(P)$, (\ref{eq: property of Theta})
implies 
\[
\int_{\partial P}\eta\geq\int_{P}\eta\Theta\geq-\left\Vert \eta\right\Vert _{L^2}\left\Vert \Theta\right\Vert _{L^2}\geq-\frac{1}{2}\left\Vert \eta\right\Vert ^{2}_{L^2}-\frac{1}{2}\left\Vert \Theta\right\Vert ^{2}_{L^2}.
\]
Thus $\Psi(\eta)\geq-\frac{1}{2}\left\Vert \Theta\right\Vert ^{2}_{L^2}$.
This lower bound is achieved by $\eta=-\Theta$ since $\int_{P}\Theta^{2}=\int_{\partial P}\Theta$.
The uniqueness follows from the strict convexity of $L^{2}$-norm. 
\end{proof}

\begin{rem}
\begin{enumerate}
\item
Since $-\Theta_{P}$ is a minimizer, for any convex $\eta$, we have
$\frac{\mathrm{d}}{\mathrm{d}t}\Psi(t\eta-\Theta_{P})|_{t=0}\ge0$
and $\frac{\mathrm{d}}{\mathrm{d}t}\Psi(-t\Theta_{P})|_{t=1}=0$.
These give us (\ref{eq: property of Theta}). Conversely, (\ref{eq: property of Theta})
can imply that $-\Theta_{P}$ is the minimizer of $\Psi$, since for any
convex $\eta$, we have 
\[
\Psi\left(\eta\right)-\Psi\left(-\Theta\right)=\frac{1}{2}\int_{P}\left(\eta+\Theta\right)^{2}+\int_{\partial P}\eta-\int_{P}\eta\Theta\ge0.
\]

\item 
The functional below is more closely related to $\mathbb{M}_{P}$
than $\Psi$, 
\begin{equation}
\breve{\mathbb{M}}_{P}(\eta)\coloneqq\mathbb{M}_{P}(\eta)+\frac{1}{2}\left\Vert \eta\right\Vert ^{2}_{L^{2}}=\Psi(\eta-\bar{S})+\frac{1}{2}\bar{S}^{2}V_{P}.\label{eq: horn M}
\end{equation}
Its relative version is 
\[
\breve{\mathbb{M}}^{\mathrm{rel}}_{P}(\eta)\coloneqq\mathbb{M}^{\mathrm{rel}}_{P}(\eta)+\frac{1}{2}\left\Vert \eta\right\Vert ^{2}_{L^{2}}=\Psi(\eta-\ell_{P})+\frac{1}{2}\int_{P}\ell^{2}_{P}.
\]
Apparently, minimizing these functionals is equivalent to minimizing
$\Psi$. 
\end{enumerate}
\end{rem}

\begin{prop}
We have $\sup_{P}\Theta_{P}=\sup_{\partial P}\Theta_{P}>0$.
\end{prop}

\begin{proof}
Set $u=-\Theta$, we show $\inf_{P}u=\inf_{\partial P}u$. If this
is not true, let $m\coloneqq\inf_{P}u<\inf_{\partial P}u$. Since
$\int_{P}\Theta=\left|\partial P\right|_{\mathrm{d}\sigma}>0$, we
know $m<0$. Take $\delta>0$ such that $0>m+\delta<\inf_{\partial P}u$.
Consider $u'=\max\left\{ u,m+\delta\right\} $. Since $u\geq\inf_{\partial P}u>m+\delta$
on $\partial P$, we have $u'=u$ on $\partial P$. Then it is easy
to see
\[
\Psi(u')=\frac{1}{2}\int_{P}u'^{2}+\int_{\partial P}u'<\frac{1}{2}\int_{P}u^{2}+\int_{\partial P}u=\Psi(u).
\]
It is absurd since $u$ is a minimizer of $\Psi$. 
\end{proof}

In \cite{li_extremal_2023}, the authors showed the boundedness of
$\Theta_{P}$. We find that their proof can yield stronger results.
For readers' convenience, we reproduce the proof with some simplifications. 
\begin{proof}[Proof of Theorem \ref{thm: theta is continuous-INTRO}]
\label{proof of Li-Lian-Sheng}For readers to make comparisons, we
use the same notations as \cite[Theorem 6.2]{li_extremal_2023}. For
example, coordinates on $N_{\bbr}$ are denoted by $x$, and $\xi$
for $M_{\bbr}$. 

Translate $P$ such that $0\in\mathrm{int}P$, then $a_{\rho}>0$.
Let 
\[
u\coloneqq-\Theta_{P}:P\ra\bbr\cup\{+\infty\},
\]
which is lsc, convex and taking finite values on $P^{*}$. We set
$u=+\infty$ outside $P$. Fix any $p^{o}\in\partial u(0)$, let $\ell(\xi)\coloneqq u(0)+\left\langle p^{o},\xi\right\rangle $
be the attached support function. Define $u^{o}\coloneqq u-\ell$,
we have $u^{o}(\xi)\geq u^{o}(0)=0$ for $\forall\xi\in P$. Let 
\[
f^{o}(x)\coloneqq\left(u^{o}\right)^{*}(x)\coloneqq\sup_{\xi\in P}\left\langle x,\xi\right\rangle -u^{o}(\xi),\ x\in N_{\bbr}
\]
be the Legendre dual of $u^{o}$. We also have $f^{o}(x)\geq f^{o}(0)=0$.
Since $0\in\mathrm{int}P$, there exists a closed ball $B\subset\mathrm{int}P$
centered at $0$, then we have 
\[
I_{P}\leq u^{o}\leq I_{B}+\sup_{B}u^{o},
\]
where $I_{B}$ is the convex characteristic function of $B$, i.e.
$I_{B}=0$ on $B$ and $=+\infty$ otherwise. Taking the Legendre
dual, we obtain 
\[
\sup_{\xi\in P}\left\langle x,\xi\right\rangle \geq f^{o}(x)\geq\sup_{\xi\in B}\left\langle x,\xi\right\rangle -\sup_{B}u^{o},\ x\in N_{\bbr}.
\]
It implies that the sublevel set 
\[
S_{f^{o}}(h)\coloneqq\left\{ x\in N_{\bbr}\mid f^{o}(x)\leq h\right\} 
\]
is compact and convex. 

We introduce the key construction in \cite[before Lem. 3.9]{li_extremal_2023}.
For $h>0$, we define the \textit{Legendre truncation} of $u^{o}$
is 
\[
u^{o}_{h}(\xi)\coloneqq\sup_{x:f^{o}(x)\leq h}\left\langle x,\xi\right\rangle -f^{o}(x)=\left(\left(u^{o}\right)^{*}+I_{S_{f^{o}}(h)}\right)^{*}(\xi).
\]
It is the tropical analogue of the Fourier truncation. Intuitionally,
for $h\gg1$, $u^{o}_{h}$ is the flattening of $u$ over the region
with large slope. Since $f^{o}\geq0$, it is easy to see 
\begin{equation}
\sup_{f^{o}(x)\leq h}\left\langle x,\xi\right\rangle -h\leq u^{o}_{h}\leq u^{o},\ \sup_{f^{o}(x)\leq h}\left\langle x,\xi\right\rangle .\label{eq:bound of uo_h}
\end{equation}
Set 
\[
u_{h}\coloneqq u^{o}_{h}+\ell\leq u^{o}+\ell=u,
\]
which is nondecreasing in $h$. We will show that 
\[
u_{h}=u,\ \text{on }P\text{ when }h\gg1.
\]
The proof proceeds in several steps. For $h>0$, let 
\[
W_{h}\coloneqq\partial f^{o}\left\{ f^{o}\leq h\right\} \subset P,
\]
which is compact by \cite[Lemma A.22]{figalli_monge-ampere_2017},
and may not be convex.

\textbf{(a)} We show $u^{o}_{h}=u^{o}$ on $W_{h}$. Suppose $\xi\in\partial f^{o}(x)$
with $f^{o}(x)\leq h$, then we have 
\[
u^{o}_{h}(\xi)\geq\left\langle x,\xi\right\rangle -f^{o}(x)=u^{o}(\xi),
\]
where the equality is due to $\xi\in\partial f^{o}(x)$. Since $u^{o}_{h}\leq u^{o}$,
it forces $u^{o}_{h}(\xi)=u^{o}(\xi)$.

\textbf{(b)} We show that $W_{h}$ is a star set, i.e. if $\xi\in W_{h}$
then $t\xi\in W_{h}$ for any $t\in[0,1)$.

Suppose $\xi\in\partial f^{o}(x)$ with $f^{o}(x)\leq h$. Take any
$x_{t}\in\partial u^{o}(t\xi)$, since $t\xi\in\partial f^{o}(x_{t})$,
it suffices to show $f^{o}(x_{t})\leq h$. Since $x\in\partial u^{o}(\xi)$
and $x_{t}\in\partial u^{o}(t\xi)$, by the definition of subdifferentials,
we have 
\begin{align*}
u^{o}(t\xi) & \geq u^{o}(\xi)+\left\langle x,t\xi-\xi\right\rangle ,\\
u^{o}(\xi) & \geq u^{o}(t\xi)+\left\langle x_{t},\xi-t\xi\right\rangle .
\end{align*}
They imply $\left\langle x_{t},\xi\right\rangle \leq\left\langle x,\xi\right\rangle $.
Then we have 
\begin{align*}
f^{o}(x_{t}) & =\left\langle x_{t},t\xi\right\rangle -u^{o}(t\xi)\leq t\left\langle x,\xi\right\rangle -u^{o}(t\xi)\\
 & \leq t\left\langle x,\xi\right\rangle -\left[u^{o}(\xi)+\left\langle x,t\xi-\xi\right\rangle \right]\\
 & =\left\langle x,\xi\right\rangle -u^{o}(\xi)=f^{o}(x)\leq h.
\end{align*}
Since $W_{h}$ is a star set, we can describe it as 
\[
W_{h}=\left\{ tq\mid q\in\partial P,0\leq t\leq r_{h}(q)\leq1\right\} ,
\]
where $r_{h}:\partial P\ra[0,1]$ is a function. Since $W_{h}$ is
closed, $r_{h}(q)$ is usc in $q$ and nondecreasing in $h$. We simply
write $r_{q}\coloneqq r_{h}(q)$. 

\textbf{(c)} We show that $\lim_{h\ra\infty}r_{h}(q)=1$ uniformly
for $q\in\partial P$. 

For any $\varepsilon\in(0,1)$, let $P_{\varepsilon}\coloneqq(1-\varepsilon)P$.
By \cite[Lemma A.22]{figalli_monge-ampere_2017}, $\partial u^{o}\left(P_{\varepsilon}\right)$
is a compact set in $N_{\bbr}$. We claim that 
\[
P_{\varepsilon}\subset W_{h}\ \text{when }h>H_{\varepsilon}\coloneqq\sup_{\partial u^{o}\left(P_{\varepsilon}\right)}f^{o}.
\]
For any $\xi\in P_{\varepsilon}$, take $x\in\partial u^{o}(\xi)\subset\partial u^{o}\left(P_{\varepsilon}\right)$,
we have $\xi\in\partial f^{o}(x)$ and $f^{o}(x)\leq H_{\varepsilon}$,
so $\xi\in W_{h>H_{\varepsilon}}$. By the definition of $r_{h}(q)$,
we have 
\[
1-\varepsilon\leq r_{h}(q)\leq1,\ \forall q\in\partial P
\]
when $h>H_{\varepsilon}$. Hence \textbf{(c)} follows. 

\textbf{(d)} For any $q\in\partial P$, $u^{o}_{h}$ is affine restricting
on the ray $\left\{ tq\mid t\geq r_{q}\coloneqq r_{h}(q)\right\} $.
Explicitly, we have 
\[
u^{o}_{h}(tq)=\frac{t}{r_{q}}\left(u^{o}(r_{q}q)+h\right)-h,\ \text{for }t\geq r_{q}.
\]
Proof: For any $q\in\partial P$, fix a $t>r_{q}$. Since $S_{f^{o}}(h)$
is compact, the supremum 
\[
u^{o}_{h}(tq)\coloneqq\sup_{x\in S_{f^{o}}(h)}\left\langle x,\xi\right\rangle -f^{o}(x)=\sup_{x\in N_{\bbr}}\left\langle x,tq\right\rangle -\left(f^{o}(x)+I_{S_{f^{o}}(h)}(x)\right)
\]
is attained by some $x^{*}\in S_{f^{o}}(h)$, and this occurs if and
only if 
\begin{equation}
tq\in\partial\left(f^{o}+I_{S_{f^{o}}(h)}\right)(x^{*})=\partial f^{o}(x^{*})+\partial I_{S_{f^{o}}(h)}(x^{*}).\label{eq:Moreau-Rockafellar}
\end{equation}
The last equality is due to Moreau--Rockafellar theorem.

If $x^{*}$ does not lie on the boundary of $S_{f^{o}}(h)$, (\ref{eq:Moreau-Rockafellar})
implies $tq\in\partial f^{o}(x^{*})\subset W_{h}$. That is absurd
since $t>r_{q}$. Hence $x^{*}$ must lie on the boundary and $f^{o}(x^{*})=h$.

By the definition of subdifferentials, we note that 
\begin{align*}
\partial I_{S_{f^{o}}(h)}(x^{*}) & =\left\{ \xi\mid\left\langle x-x^{*},\xi\right\rangle \leq0,\ \forall x\in S_{f^{o}}(h)\right\} \\
 & =\left\{ \lam\eta\mid\lam\geq0,\ \eta\in\partial f^{o}(x^{*})\right\} .
\end{align*}
Thus by (\ref{eq:Moreau-Rockafellar}), we can write $tq=\eta_{1}+\lam\eta_{2}$
with $\eta_{i}\in\partial f^{o}(x^{*})$ and $\lam>0$. 

For any $a\geq1$, we find 
\[
atq=\eta_{1}+\left(a-1+a\lam\right)\frac{\left(a-1\right)\eta_{1}+a\lam\eta_{2}}{a-1+a\lam}\in\partial f^{o}(x^{*})+\partial I_{S_{f^{o}}(h)}(x^{*}),
\]
where the convexity of $\partial f^{o}(x^{*})$ is used. This inclusion
relation implies that the supremum 
\[
u^{o}_{h}(atq)\coloneqq\sup_{x\in N_{\bbr}}\left\langle x,atq\right\rangle -\left(f^{o}(x)+I_{S_{f^{o}}(h)}(x)\right)
\]
is still attained by $x^{*}$. So $u^{o}_{h}(atq)=\left\langle x^{*},atq\right\rangle -h$
for all $a\geq1$. Since $t$ can be arbitrarily close to $r_{q}$,
\textbf{(d)} follows. 

\textbf{(e)} It is ready to show $u=u_{h}$ for $h\gg1$. Let 
\[
v_{h}\coloneqq u-u_{h}=u^{o}-u^{o}_{h}\geq0.
\]
By property \textbf{(c)}, for any $q\in\partial P$, the function
$[r_{q},1]\ni t\mapsto v_{h}(tq)$ is convex and vanishes at $t=r_{q}$
by \textbf{(a)}. It implies that this function is nondecreasing, so
\begin{equation}
v_{h}(tq)\leq v_{h}(q),\ \forall t\in[r_{q},1].\label{eq:controll v_h}
\end{equation}
Note that $v_{h}(q)<\infty$, when $q$ lies in the relative interior
of facets, since $u<+\infty$ on $P^{*}$. 

To estimate the integrals over $P$, we decompose $P=\bigcup_{\rho\in\mathcal{F}_{P}(1)}P_{\rho}$,
where $P_{\rho}$ is the cone with apex $0$ and base $F_{\rho}$
(facet of $P$ associated to $\rho$). Consider the map 
\[
[0,1]\times F_{\rho}\ra P_{\rho},\ \left(t,q\right)\mapsto tq.
\]
One can check that the pullback of Lebesgue measure $\mathrm{d}\mu$
along this map is $a_{\rho}t^{n-1}\mathrm{d}t\wedge\mathrm{d\sigma}$. 

By \textbf{(a)}, $v_{h}(tq)=0$ when $t\in[0,r_{q}]$. Use (\ref{eq:controll v_h}),
we have 
\[
\int_{P_{\rho}}v_{h}\mathrm{d}\mu=a_{\rho}\int_{F_{\rho}}\int^{1}_{r_{q}}v_{h}(tq)t^{n-1}\mathrm{d}t\mathrm{d\sigma}\leq\frac{a_{\rho}}{n}\int_{F_{\rho}}\left(1-r^{n}_{q}\right)v_{h}(q)\mathrm{d\sigma}.
\]
Take the sum over $\rho$, we obtain 
\begin{equation}
0\leq\int_{P}v_{h}\mathrm{d}\mu\leq\frac{\max_{\rho}a_{\rho}}{n}\left(1-\inf_{q\in\partial P}r^{n}_{q}\right)\int_{\partial P}v_{h}\mathrm{d\sigma}.\label{eq:controll vh int}
\end{equation}
Compare the $L^{2}$-norms, 
\begin{align*}
\int_{P}u^{2}_{h}-\int_{P}u^{2} & =\int_{P}\left(u^{o}_{h}+\ell\right)^{2}-\int_{P}\left(u^{o}+\ell\right)^{2}=\int_{P}\left(u^{o}_{h}\right)^{2}-\int_{P}\left(u^{o}\right)^{2}-2\int_{P}v_{h}\ell\\
 & \leq2\max_{P}\left|\ell\right|\int_{P}v_{h}\leq\frac{2}{n}\max_{P}\left|\ell\right|\max_{\rho}a_{\rho}\left(1-\inf_{q\in\partial P}r^{n}_{q}\right)\int_{\partial P}v_{h}\mathrm{d\sigma}.
\end{align*}
where $\int_{P}\left(u^{o}_{h}\right)^{2}\leq\int_{P}\left(u^{o}\right)^{2}$
is due to $0\leq u^{o}_{h}\leq u^{o}$. 

We obtain a comparison of functional values, 
\[
\left(\frac{1}{2}\int_{P}u^{2}_{h}+\int_{\partial P}u_{h}\right)-\left(\frac{1}{2}\int_{P}u^{2}+\int_{\partial P}u\right)\leq\left[C\left(1-\inf_{q\in\partial P}r^{n}_{q}\right)-1\right]\int_{\partial P}v_{h}\mathrm{d\sigma}.
\]
By \textbf{(c)}, we know $\inf_{q\in\partial P}r^{n}_{q}\ra1$. So
there exists $h_{0}>0$, when $h>h_{0}$, the RHS $\leq-\frac{1}{2}\int_{\partial P}v_{h}$.
Since $u$ is a minimizer of (\ref{eq: Inoue functional}), we must
have $\int_{\partial P}v_{h}\mathrm{d\sigma}=0$ for $h>h_{0}$. By
(\ref{eq:controll vh int}), it implies $\int_{P}v_{h}=0$, so $u=u_{h}$
almost everywhere. Since $u$ is continuous on $P^{*}$ and $u_{h}$
is continuous everywhere, so $u=u_{h}$ on $P^{*}$. Note $u$ is
also lsc, so 
\[
u_{h}=u,\ \text{on }P\text{ when }h>h_{0}.
\]
Moreover, by (\ref{eq:bound of uo_h}), we have 
\[
\partial u\left(\mathrm{int}P\right)=\partial u_{h}\left(\mathrm{int}P\right)\subset\partial u_{h}\left(M_{\bbr}\right)=S_{f^{o}}(h)+p^{o}
\]
when $h>h_{0}$. Since $S_{f^{o}}(h)$ is bounded, so is $\partial u\left(\mathrm{int}P\right)$. 
\end{proof}

The following property is due to Li and Zhou. They deal with a general
problem and assume a smoothness condition (P2) which is not satisfied
by $\Theta_{P}$, but the proof is still valid. 
\begin{thm}[{\cite[Proposition 3.2]{li_k-stability_2022}}]
\label{thm: Li Zhou} Convex function $u=-\Theta_{P}$ satisfies
the homogenous Monge--Ampère equation in the sense of Alexandrov.
Precisely, for any compact set $K\subset\mathrm{int}P$, the Lebesgue
measure $\left|\partial u(K)\right|=0$. 
\end{thm}

\section{\label{subsec:  min Qk proof} Connect shortest GKZ vector with optimal
test-function}

In this section, we address the main question: the maximal Chow-destabilizers
quantize the maximal K-destabilizers. For toric varieties, we have
shown that it can be reduced to the relation between $\Theta_{P}$
and the shortest GKZ vectors $\sigma_{k}$ for $(P,P_{k})$. First
we introduce a functional whose minimizer is given by $\sigma_{k}$,
meanwhile, it converges to the functional $\Psi$. 

\subsection{\texorpdfstring{Quantum functional $\Psi_{k}$}{Quantum functional Psi\_k}}

Consider the expansion (\ref{eq:EM}) of Riemann sums, we introduce
a functional 
\[
\Psi_{k}(\eta)\coloneqq\frac{1}{2k^{n}}\sum_{u\in P_{k}}\eta(u)^{2}+\frac{2}{k^{n-1}}\left(\sum_{u\in P_{k}}\eta(u)-k^{n}\int_{P}\eta\ \mathrm{d}x\right),\ \eta\in\mathcal{C}(P),
\]
which pointwisely converges to $\Psi$.

\begin{thm}
\label{thm: min Psi_k} Functional $\Psi_{k}$ admits a unique minimizer
\begin{equation}
\psi_{k}\coloneqq2V_{P}k^{n+1}\sigma^{\vee}_{k}-2k.\label{eq:quan minimizer}
\end{equation}
\end{thm}

\begin{proof}
\textbf{Step 1}. Reduce the minimization problem 
\begin{equation}
\inf\left\{ \Psi_{k}(\eta)\mid\eta\in\mathcal{C}(P)\right\} \label{eq: min Psi_k}
\end{equation}
to (\ref{eq: minimize Mk horn}). For any $\eta\in\mathcal{C}(P)$,
let $\eta_{k}=\left(\eta|_{P_{k}}\right)^{\vee}$. By Proposition
\ref{prop:FS vs SN}, we have $\eta_{k}\ge\eta$ and $\eta_{k}|_{P_{k}}=\eta|_{P_{k}}$,
these imply $\Psi_{k}(\eta_{k})\leq\Psi_{k}(\eta)$ with equality
only if $\eta_{k}=\eta$. Hence problem (\ref{eq: min Psi_k}) can
be reduced to 
\begin{equation}
\inf\left\{ \Psi_{k}(\eta)\mid\eta\in\mathrm{PL}(P,P_{k})\right\} ,\label{eq: inf Q horn PL}
\end{equation}
recall $\mathrm{PL}(P,P_{k})$ is the image of envelope map $\vee:\bbr^{P_{k}}\ra\mathcal{C}(P)$. 

By Proposition \ref{prop:FS vs SN} (3), $\vee:\mathcal{C}(P_{k})\ra\mathrm{PL}(P,P_{k})$
is a bijection. Hence we can change the domain of (\ref{eq: inf Q horn PL})
to $\mathcal{C}(P_{k})$. Since for any $f\in\mathcal{C}(P_{k})$,
we have $f^{\vee}|_{P_{k}}=f$, it follows that 
\[
\Psi_{k}\left(f^{\vee}\right)=\frac{1}{2k^{n}}\sum_{u\in P_{k}}f(u)^{2}+\frac{2}{k^{n-1}}\left(\sum_{u\in P_{k}}f(u)-k^{n}\int_{P}f^{\vee}\right)\eqqcolon\breve{M}_{P_{k}}(f).
\]
By changing the domain, problem (\ref{eq: inf Q horn PL}) is equivalent
to 
\begin{equation}
\inf\left\{ \breve{M}_{P_{k}}(f)\mid f\in\mathcal{C}(P_{k})\right\} .\label{eq: minimize Mk horn}
\end{equation}

\textbf{Step 2}. We claim that 
\[
f^{*}\coloneqq2V_{P}k^{n+1}\left(\sigma_{k}-\frac{1}{V_{P}k^{n}}\right)
\]
is actually the unique minimizer of $\breve{M}_{P_{k}}$ over larger
space $\bbr^{P_{k}}$. The uniqueness is obvious, since $\breve{M}_{P_{k}}$
is strictly convex. 

Once this claim holds, since $f^{*}\in\mathcal{C}(P_{k})$ by Theorem
\ref{thm: convex of mGKZ} and Proposition \ref{prop:FS vs SN} (2),
it is also the unique minimizer of (\ref{eq: minimize Mk horn}).
Then trace back the reductions in Step 1, we see $\left(f^{*}\right)^{\vee}$
is the unique minimizer of $\Psi_{k}$, that is (\ref{eq:quan minimizer}). 

Next we verify the above claim. For any $f\in\bbr^{P_{k}}$, by (\ref{eq: Int tent formula}),
we have 
\[
\frac{1}{V_{P}}\int_{P}f^{\vee}=\min_{g\in\Sigma(P_{k})}\left\langle f,g\right\rangle \le\left\langle f,\sigma_{k}\right\rangle =\left\langle f,\sigma_{k}-\frac{1}{V_{P}k^{n}}\right\rangle +\frac{1}{V_{P}k^{n}}\sum_{u\in P_{k}}f(u),
\]
\[
\left\langle f,\sigma_{k}-\frac{1}{V_{P}k^{n}}\right\rangle \le\left\Vert f\right\Vert \left\Vert \sigma_{k}-\frac{1}{V_{P}k^{n}}\right\Vert \le\frac{1}{4V_{P}k^{n+1}}\left\Vert f\right\Vert ^{2}+V_{P}k^{n+1}\left\Vert \sigma_{k}-\frac{1}{V_{P}k^{n}}\right\Vert ^{2}.
\]
Combine these inequalities together, we obtain 
\[
\breve{M}_{P_{k}}(f)\geq-2V^{2}_{P}k^{n+2}\left\Vert \sigma_{k}-\frac{1}{V_{P}k^{n}}\right\Vert ^{2}.
\]
Next we check that this lower bound is achieved by $f^{*}$. Denote
$f^{*}=a(\sigma_{k}-c)$ for short, where $a=2V_{P}k^{n+1}$ and $c=\frac{1}{V_{P}k^{n}}$.
Using identities 
\[
\sum_{u\in P_{k}}\sigma_{k}(u)=1\text{ and }\int_{P}\sigma^{\vee}_{k}=V_{P}\left\Vert \sigma_{k}\right\Vert ^{2},
\]
we compute 
\begin{align*}
\sum_{u\in P_{k}}f^{*}(u)-k^{n}\int_{P}\left(f^{*}\right)^{\vee} & =a\left(2-c\left|P_{k}\right|-c^{-1}\left\Vert \sigma_{k}\right\Vert ^{2}\right)=-\frac{a}{c}\left\Vert \sigma_{k}-c\right\Vert ^{2}.
\end{align*}
Hence 
\[
\breve{M}_{P_{k}}(f^{*})=\frac{a^{2}}{2k^{n}}\left\Vert \sigma_{k}-c\right\Vert ^{2}-\frac{2}{k^{n-1}}\frac{a}{c}\left\Vert \sigma_{k}-c\right\Vert ^{2}=-2V^{2}_{P}k^{n+2}\left\Vert \sigma_{k}-c\right\Vert ^{2}.
\]
\end{proof}

\begin{rem}
The toric version of general functional $\breve{\mathbb{Q}}_{k}$
in \cite{yao_maximal_2025} is  
\[
\frac{1}{2k^{n}}\sum_{u\in P_{k}}\eta(u)^{2}+2kV_{P}\left(\frac{1}{\left|P_{k}\right|}\sum_{u\in P_{k}}\eta(u)-\fint_{P}\eta\right),
\]
which pointwisely converges to $\breve{\mathbb{M}}_{P}(\eta)$ (\ref{eq: horn M}).
However, it equals to 
\[
\Psi_{k}\left(\eta-\frac{s_{k}}{r_{k}}\right)+\frac{V_{P}s^{2}_{k}}{2r_{k}},\ r_{k}=\frac{\left|P_{k}\right|}{V_{P}k^{n}}\rightarrow1,\ s_{k}=2k(r_{k}-1)\rightarrow\bar{S}.
\]
Hence for fixed $k$, minimizing the above functional amounts to minimizing
$\Psi_{k}$. Indeed, by using Remark \ref{rem-quadmin} and the proof of Theorem \ref{thm: Chow destabi}, we know that $\breve{\mathbb{Q}}_{k}$ obtains the minimum at $\eta_k=2k^{n+1}V_P(\sigma_k^\vee-\frac{1}{|P_k|})=2k^{n+1}V_P\sigma_k^\vee-\frac{2k}{r_k}$ and one sees that $\eta_k-\frac{s_k}{r_k}=\psi_k$ from \eqref{eq:quan minimizer}. 
\end{rem}

The following task is to show the minimizer of $\Psi_{k}$ converges
to that of $\Psi$ under a certain topology. This is a typical variational
problem involving $\Gamma$-convergence. 

\subsection{\label{subsec:Estimate-the-integrals}Uniform bounds for integrals
via lattice sums}

This section is devoted to the key estimates in the proof of minimizer
convergence. 
\begin{prop}[Rough estimate]
Given a lattice polytope $P$, there exists a constant $C_{P}>0$
(depending on $P$) such that for any $k\ge1$, there exists a triangulation
$T$ of $(P,P_{k})$ with GKZ vector $g_{T}\le C_{P}k^{-n}$. Then
for any convex function $\phi\ge0$ on $P$, we have 
\begin{equation}
\fint_{P}\phi\mathrm{d}x\le\frac{C_{P}}{k^{n}}\sum_{u\in P_{k}}\phi(u).\label{eq:rough bound}
\end{equation}
\end{prop}

\begin{proof}
First we triangulate $P$ using diagonals. For each resulting simplex
$\Delta$, take an affine bijection from the standard simplex $\Delta_{n}=\mathrm{conv}\{0,e_{1},\dots,e_{n}\}$
to $\Delta$. The Freudenthal--Kuhn triangulation of $\Delta_{n}$
uses the points in $k^{-1}\mathbb{Z}^{n}$ and consists of $k^{n}$
simplices. Via the affine map, it induces a triangulation $T$ of
$(P,P_{k})$, see Figure \ref{fig:Kuhn}. The number of simplices
adjacent to a point $u\in P_{k}$ is bounded above by a constant which
is independent of $k$. This implies $g_{T}\le C_{P}k^{-n}$ for a
constant $C_{P}$. 
\begin{figure}[H]
\begin{centering}
\includegraphics[clip,scale=0.2]{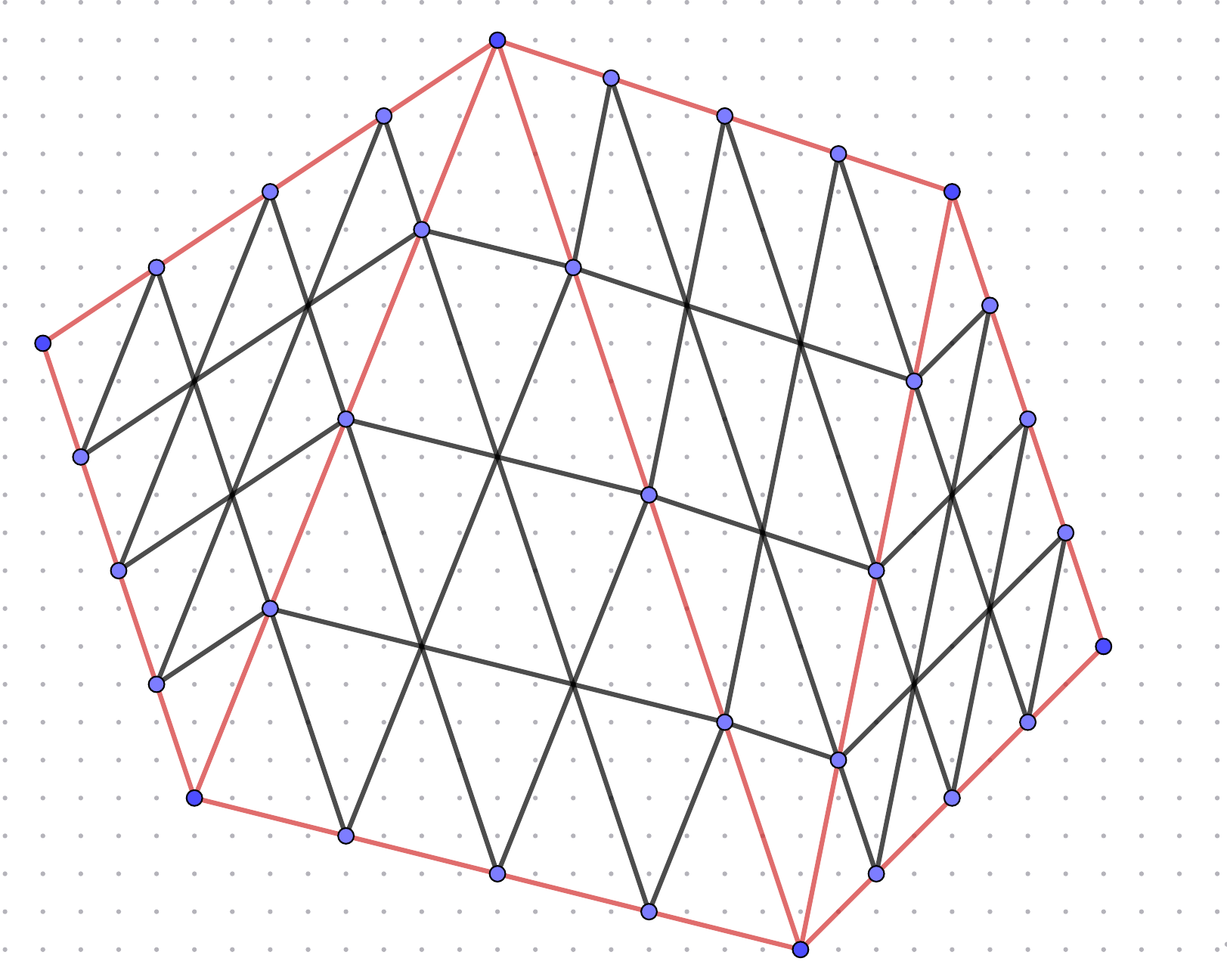}
\par\end{centering}
\caption{\label{fig:Kuhn}Freudenthal--Kuhn triangulation, $k=4$}
\end{figure}

Finally, by (\ref{eq: Int tent formula}), we have 
\[
\fint_{P}\phi\le\fint_{P}\left(\phi|_{P_{k}}\right)^{\vee}\le\sum_{u\in P_{k}}g_{T}(u)\phi(u)\le\frac{C_{P}}{k^{n}}\sum_{u\in P_{k}}\phi(u).
\]
\end{proof}

When $P$ is a Delzant polytope, we derive a sharper bound than (\ref{eq:rough bound}).
Let $\lam:P\ra\Delta_{P_{k}}$ be a coefficient field for $(P,P_{k})$,
which is defined in §\ref{subsec: coefficient field}. For any $\phi\in\mathcal{C}(P)$,
Jensen inequality gives 
\[
\phi(x)\le\sum_{u\in P_{k}}\lam_{u}(x)\phi(u),\ \forall x\in P.
\]
Integrating over $P$, we obtain an upper bound in the form of a lattice
sum, 
\begin{equation}
\int_{P}\phi\mathrm{d}x\le\sum_{u\in P_{k}}\left(\int_{P}\lam_{u}(x)\mathrm{d}x\right)\phi(u).\label{eq:lam give bound}
\end{equation}

We will derive a sharp bound by selecting a special $\lam$, which is based on the following model.
\begin{example}
\label{exa:model lambda} Consider an $n$-dimensional lattice rectangular box $Q=\prod^{n}_{i=1}[0,m_{i}]\subset\mathbb{R}^{n}$
with $m_{i}\in\mathbb{Z}_{>0}$. There is a canonical division of
$Q$ into $\left(\prod^{n}_{i=1}m_{i}\right)$ number of $n$-dimensional
unit lattice square boxes. For $x=(x_{1},\dots,x_{n})\in Q$ belongs
to box $\prod^{n}_{i=1}[j_{i},j_{i}+1]$ with $j_{i}\in\{0,1,\dots,m_{i}-1\}$,
we have 
\[
x_{i}=\left(1-(x_{i}-j_{i})\right)\cdot j_{i}+(x_{i}-j_{i})(j_{i}+1)=\sum^{m_{i}}_{j=0}(1-|x_{i}-j|)_{+}\cdot j,
\]
where $t_{+}=\max\{t,0\}$. Define a coefficient field on $Q$: for
any $u\in Q\cap\mathbb{Z}^{n}$, set 
\[
\lambda_{u}(x)=\prod^{n}_{i=1}(1-|x_{i}-u_{i}|)_{+}.
\]
Then for $x\in Q$, the following equalities are satisfied: 
\begin{eqnarray*}
\sum_{u\in Q\cap\mathbb{Z}^{n}}\lambda_{u}(x) & = & \prod^{n}_{i=1}\sum^{m_{i}}_{u_{i}=0}(1-|x_{i}-u_{i}|)_{+}=1,\\
\sum_{u\in Q\cap\mathbb{Z}^{n}}\lambda_{u}(x)u_{k} & = & \sum^{m_{k}}_{u_{k}=0}(1-|x_{k}-u_{k}|)_{+}u_{k}\cdot\prod_{i\neq k}\sum^{m_{i}}_{u_{i}=0}(1-|x_{i}-u_{i}|)_{+}\\
 & = & x_{k}.
\end{eqnarray*}
It is easy to see that for each $x\in Q$, $\lambda_{u}(x)>0$ only
when $d_{\infty}(u,x)<1$. In other words, the support of $\lambda(x)$
is contained in a single unit square. Moreover, due to the separation
of variables, we easily calculate: 
\[
\int_{Q}\lambda_{u}(x)\mathrm{d}x=2^{-\mathrm{cd}(u)},
\]
where $\mathrm{cd}(u)$ is the the number of facets of $Q$ that contain
$u$. For any $\phi\in\mathcal{C}(Q)$, (\ref{eq:lam give bound})
gives the estimate: 
\[
\int_{Q}\phi\mathrm{d}x\le\sum_{u\in Q\cap\mathbb{Z}^{n}}2^{-\mathrm{cd}(u)}\phi(u).
\]
\end{example}

The following
result is a simplification of \cite[Corollary 2.16, 2.18]{lee_asymptotic_2026}. 
The basic idea of proof is that,
up to rescaling, near each vertex of the Delzant polytope, we can use uni-modular linear transformation to reduce the estimates to the model case as in the above example. Then we will use partition of unity to cover the whole polytope using open subsets associated to all vertices and finally sum up the estimates.  
\begin{thm}[Fine estimate]\label{thm-fine}
For a Delzant lattice polytope $P$, there exists constant $C_{1},C_{2},C_{b}>0$
and $k_{0}$ (all only depend on $P$) such that for any $0\le\phi\in\mathcal{C}(P)$
and $k>k_{0}$, we have 
\begin{equation}
k^{n}\int_{P}\phi\le\sum_{u\in P_{k}}2^{-\mathrm{cd}(u)}\phi(u)+\frac{C_{1}}{k}\sum_{u\in P_{k}\cap\partial P}\phi(u)+\frac{C_{2}}{k^{2}}\sum_{u\in P_{k}}\phi(u),\label{eq:key estimate}
\end{equation}
where $\mathrm{cd}(u)$ is the number of facets containing $u$. Apply it to each facet, we obtain 
\begin{equation}
k^{n-1}\int_{\partial P}\phi\mathrm{d}\sigma\le\left(1+\frac{C_{b}}{k}\right)\sum_{u\in P_{k}\cap\partial P}\phi(u).\label{eq: b key estimate}
\end{equation}
\end{thm}

\begin{proof}
For a vertex $v$ of $P$, let $e^{1}_{v},\dots,e^{n}_{v}\in M$ be
the primitive generators of the edges emanating from $v$, which form
a lattice basis since $P$ is Delzant. Let $C_{v}$ be the vertex
cone at $v$, define a map 
\[
\Gamma_{v}:\mathbb{R}^{n}_{\ge0}\rightarrow C_{v},\ \Gamma_{v}(t)=v+\frac{t_{1}}{k}e^{1}_{v}+\cdots+\frac{t_{n}}{k}e^{n}_{v},
\]
which is bijective between $\mathbb{\mathbb{Z}}^{n}_{\ge0}$ and $C_{v}\cap k^{-1}M$.

First we define a coefficient field for $\left(\mathbb{R}^{n}_{\ge0},\mathbb{\mathbb{Z}}^{n}_{\ge0}\right)$,
then transport it to $P$ via the map $\Gamma_{v}$, and finally patch
them together. For $n=1$, let 
\[
b_{0}(t)=(1-t)_{+},\ b_{j}(t)=\left(1-\left|t-j\right|\right)_{+}\text{ for }j\ge1.
\]
Easy to check 
\[
\sum_{j\ge0}b_{j}(t)=1,\ \sum_{j\ge0}b_{j}(t)\cdot j=t,\ \forall t\ge0.
\]
We extend it to $\mathbb{R}^{n}_{\ge0}$ by taking the tensor product,
\[
B_{m}(t)\deff b_{m_{1}}(t_{1})\cdots b_{m_{n}}(t_{n}),\text{ for }m\in\mathbb{\mathbb{Z}}^{n}_{\ge0}\text{ and }t\in\mathbb{R}^{n}_{\ge0}.
\]
Then we have 
\begin{equation}
\sum_{m\in\mathbb{\mathbb{Z}}^{n}_{\ge0}}B_{m}(t)=1,\ \sum_{m\in\mathbb{\mathbb{Z}}^{n}_{\ge0}}B_{m}(t)\cdot m=t,\ \forall t\in\mathbb{R}^{n}_{\ge0}.\label{eq: field B}
\end{equation}
Next we patch 
\[
\left\{ B_{\Gamma^{-1}_{v}u}(\Gamma^{-1}_{v}\cdot):P\ra\bbr\mid u\in P_{k}\right\} _{v}
\]
 together via a partition of unity. Suppose 
\[
P=\left\{ x\in M_{\bbr}\mid l_{F}(x)\ge0,\text{ for all facet }F\right\} .
\]
Fix a small $\varepsilon>0$, for a vertex $v$ of $P$, set 
\[
U_{v}\coloneqq\left\{ x\in P\mid l_{F}(x)>\varepsilon,\ \forall F\not\ni v\right\} .
\]
Then $\left\{ U_{v}\right\} _{v}$ forms an open cover of $P$. Take
a partition of unity $\left\{ \rho_{v}:P\ra\bbr\right\} _{v}$ which
are smooth up to $\partial P$ and satisfy 
\[
\sum_{v}\rho_{v}=1,\ \mathrm{supp}\rho_{v}\subset U_{v}.
\]
For $x\in P$, let $F_{x}$ be the smallest face containing $x$.
Note that $\rho_{v}(x)=0$ if $v$ is not a vertex of $F_{x}$.

We define 
\[
\lam_{u}(x)=\sum_{v\in\mathrm{vert}P}\rho_{v}(x)B_{\Gamma^{-1}_{v}u}\left(\Gamma^{-1}_{v}x\right),\ \forall x\in P,\ u\in P_{k}.
\]
For any $x\in P$, we check 
\[
\sum_{u\in P_{k}}\lam_{u}(x)=\sum_{v\in\mathrm{vert}F_{x}}\rho_{v}(x)\sum_{u\in P_{k}}B_{\Gamma^{-1}_{v}u}\left(\Gamma^{-1}_{v}x\right)=\sum_{v\in\mathrm{vert}F_{x}}\rho_{v}(x)=1,
\]
where $\sum_{u\in P_{k}}B_{\Gamma^{-1}_{v}u}\left(\Gamma^{-1}_{v}x\right)=1$
holds since (\ref{eq: field B}), $x\in\mathrm{relint}F_{x}$ and
$k$ is sufficiently large (depending on $\varepsilon$). Similarly,
we have 
\[
\sum_{u\in P_{k}}\lam_{u}(x)u=x,\ \forall x\in P.
\]
Thus $\lam$ is a coefficient field for $(P,P_{k})$. It yields the
upper bound (\ref{eq:lam give bound}). Next we give an upper bound
for $\int_{P}\lam_{u}(x)$.

For a vertex $v$ and $u\in P_{k}$, changing the variables, we have
\begin{equation}
\int_{P}\rho_{v}(x)B_{\Gamma^{-1}_{v}u}\left(\Gamma^{-1}_{v}x\right)\mathrm{d}x=k^{-n}\int_{\mathbb{R}^{n}_{\ge0}}\rho_{v}\left(\Gamma_{v}t\right)B_{\Gamma^{-1}_{v}u}(t)\mathrm{d}t,\label{eq: int rho}
\end{equation}
where we extend $\rho_{v}$ by zero to the entire cone $C_{v}$. Note
the support of $B_{\Gamma^{-1}_{v}u}(t)$ is 
\[
\left\{ t\in\mathbb{R}^{n}_{\ge0}\mid\left|t-\Gamma^{-1}_{v}u\right|_{\infty}\le1\right\} .
\]
Consider the Taylor expansion of $\rho_{v}\left(\Gamma_{v}t\right)$
near the point $\Gamma^{-1}_{v}u$, we have 
\[
\rho_{v}\left(\Gamma_{v}t\right)=\rho_{v}(u)+\frac{1}{k}\sum^{n}_{i=1}D_{e^{i}_{v}}\rho_{v}(u)\left(t_{i}-\left(\Gamma^{-1}_{v}u\right)_{i}\right)+R^{u}_{v}(t),
\]
where $\left|R^{u}_{v}(t)\right|\le Ck^{-2}\left|t-\Gamma^{-1}_{v}u\right|^{2}_{2}$
and $C$ depends on the $C^{2}$-norm of $\rho_{v}$. Bring this expansion
back to (\ref{eq: int rho}). Via direct computation, we have 
\[
\int_{\mathbb{R}^{n}_{\ge0}}B_{\Gamma^{-1}_{v}u}(t)\mathrm{d}t=2^{-\mathrm{cd}(u)},
\]
\[
\int_{\mathbb{R}^{n}_{\ge0}}\left(t_{i}-\left(\Gamma^{-1}_{v}u\right)_{i}\right)B_{\Gamma^{-1}_{v}u}(t)\mathrm{d}t=\begin{cases}
2^{-\mathrm{cd}(u)}/3, & \left(\Gamma^{-1}_{v}u\right)_{i}=0,\\
0, & \left(\Gamma^{-1}_{v}u\right)_{i}>0.
\end{cases}
\]
Finally, we obtain upper bound, 
\[
k^{n}\int_{P}\lam_{u}(x)\le2^{-\mathrm{cd}(u)}+\frac{C_{1}}{k}I_{\partial P}(u)+\frac{C_{2}}{k^{2}},
\]
where $C_{i}$ depends on the $C^{i}$-norm of $\{\rho_{v}\}$, and
$I_{\partial P}(u)$ is the indicator function of $\partial P$. Apply
this bound to (\ref{eq:lam give bound}), (\ref{eq:key estimate})
follows. 

Apply (\ref{eq:key estimate}) to each facet $F$, we obtain 
\[
k^{n-1}\int_{F}\phi\le\sum_{u\in F_{k}}2^{1-\mathrm{cd}(u)}\phi(u)+\frac{C^{F}_{1}}{k}\sum_{u\in\partial F\cap k^{-1}M}\phi(u)+\frac{C^{F}_{2}}{k^{2}}\sum_{u\in F_{k}}\phi(u).
\]
For Delzant polytopes, a boundary lattice point $u$ belongs to $\mathrm{cd}(u)$
facets. Take the sum over $F$, we obtain 
\[
k^{n-1}\int_{\partial P}\phi\le\sum_{u\in P_{k}\cap\partial P}\mathrm{cd}(u)\cdot2^{1-\mathrm{cd}(u)}\phi(u)+\frac{C^{\prime}_{1}}{k}\sum_{u\in\partial^{>1}P\cap k^{-1}M}\phi(u)+\frac{C^{\prime}_{2}}{k^{2}}\sum_{u\in P_{k}\cap\partial P}\phi(u),
\]
where $\partial^{>1}P$ is the union of faces with codimension $>1$.
Since $\mathrm{cd}(u)\cdot2^{1-\mathrm{cd}(u)}\le1$, (\ref{eq: b key estimate})
follows. 
\end{proof}

\subsection{Control the normalized convex functions}

For $f,g:P_{k}\ra\bbr$ and $\eta,\xi:P\ra\bbr$, we define 
\begin{gather*}
\left\langle f,g\right\rangle _{k}\coloneqq\frac{1}{k^{n}}\sum_{u\in P_{k}}f(u)g(u),\ \left\Vert f\right\Vert ^{2}_{k}\coloneqq\left\langle f,f\right\rangle _{k},\\
\left\langle \eta,\xi\right\rangle _{k}\coloneqq\frac{1}{k^{n}}\sum_{u\in P_{k}}\eta(u)\xi(u),\ \left\Vert \eta\right\Vert ^{2}_{k}\coloneqq\left\langle \eta,\eta\right\rangle _{k}.
\end{gather*}
Let $\pi_{k}:\bbr^{P_{k}}\ra\mathcal{A}$ be the orthogonal projection
with respect to $\left\langle \cdot,\cdot\right\rangle _{k}$. For
$\eta:P\ra\bbr$, we write $\pi_{k}\eta\coloneqq\pi_{k}(\eta|_{P_{k}})$. 

Fix $x_{0}\in\mathrm{int}P$, we say a convex function $\phi:P\ra\bbr$
is normalized if 
\[
\phi\ge0\text{ and }\phi(x_{0})=0.
\]

\begin{prop}
Fix a point $x_{0}\in\mathrm{int}P$. There exists a constant $C_{\pi}(P,x_{0})>0$
such that 
\begin{equation}
\left\Vert \phi\right\Vert _{k}\le C_{\pi}\left\Vert \phi-\pi_{k}\phi\right\Vert _{k}\label{eq:projection control}
\end{equation}
holds for all normalized convex functions at $x_{0}$. 
\end{prop}

\begin{proof}
If not, there exists a sequence of normalized convex functions $\phi_{k}$
so that 
\[
\left\Vert \phi_{k}\right\Vert _{k}=1,\text{ but }\left\Vert \phi_{k}-\pi_{k}\phi_{k}\right\Vert _{k}\ra0.
\]
By (\ref{eq:rough bound}) and C-S inequality, $\int_{P}\phi_{k}$
is uniformly bounded. A compactness result \cite[Corollary 5.2.5]{donaldson_scalar_2002}
implies that, after passing to a subsequence, $\phi_{k}$ locally
uniformly converges to a normalized convex function $\phi$. 

Let $\alpha_{k}=\pi_{k}\phi_{k}$, we have $\left\Vert \alpha_{k}\right\Vert _{k}\le\left\Vert \phi_{k}\right\Vert _{k}=1$.
Since $\mathcal{A}$ is finite-dimensional, via a contradiction argument,
we can find a constant $C>0$ such that 
\begin{equation}
C^{-1}\left\Vert \alpha\right\Vert _{k}\le\sup_{P}\left|\alpha\right|\le C\left\Vert \alpha\right\Vert _{k},\ \forall\alpha\in\mathcal{A},\ k\ge1.\label{eq: finite dim norms}
\end{equation}
By the compactness in finite-dimensional spaces, we may assume $\alpha_{k}\ra\alpha\in\mathcal{A}$. 

Let $\eta_{k}=\phi_{k}-\alpha_{k}$ and $\eta=\phi-\alpha$, so $\left\Vert \eta\right\Vert _{k}\ra0$.
For any subpolytope $K\subset\mathrm{int}P$, by the uniform convergence
on $K$, we have 
\[
\frac{1}{k^{n}}\sum_{u\in P_{k}\cap K}\eta_{k}(u)^{2}-\frac{1}{k^{n}}\sum_{u\in P_{k}\cap K}\eta(u)^{2}\ra0,\ \frac{1}{k^{n}}\sum_{u\in P_{k}\cap K}\eta(u)^{2}\ra\int_{K}\eta^{2}.
\]
It implies 
\[
\int_{K}\eta^{2}=\lim_{k}\frac{1}{k^{n}}\sum_{u\in P_{k}\cap K}\eta_{k}(u)^{2}\le\lim_{k}\left\Vert \eta\right\Vert ^{2}_{k}=0
\]
for all $K$. Thus $\phi=\alpha$, and $\alpha=0$ since $\phi$ is
normalized. But this contradicts with 
\[
\int_{P}\alpha^{2}\leftarrow\left\Vert \alpha_{k}\right\Vert ^{2}_{k}=\left\Vert \phi_{k}\right\Vert ^{2}_{k}-\left\Vert \eta\right\Vert ^{2}_{k}\ra1.
\]
\end{proof}

\subsection{Bound the minimizers }

We rewrite 
\[
\Psi_{k}(\eta)=\frac{1}{2}\left\Vert \eta\right\Vert ^{2}_{k}+L_{k}(\eta),\ L_{k}(\eta)\coloneqq\frac{2}{k^{n-1}}\left(\sum_{u\in P_{k}}\eta(u)-k^{n}\int_{P}\eta\ \mathrm{d}x\right).
\]
Since $L_{k}$ is linear, there is a unique affine function $\beta_{k}\in\mathcal{A}$,
such that 
\begin{equation}
L_{k}(\alpha)=\left\langle \alpha,\beta_{k}\right\rangle _{k},\ \forall\alpha\in\mathcal{A}.\label{eq: def beta_k}
\end{equation}
Since $\psi_{k}$ minimizes $\Psi_{k}$, for any $\alpha\in\mathcal{A}$,
we have 
\[
0=\frac{\mathrm{d}}{\mathrm{d}t}\Psi_{k}(\psi_{k}+t\alpha)|_{t=0}=\left\langle \alpha,\psi_{k}+\beta_{k}\right\rangle _{k}.
\]
Hence 
\[
\beta_{k}=-\pi_{k}(\psi_{k})=2k-2V_{P}k^{n+1}\ell_{P_{k}}.
\]

\begin{prop}
$\lim_{k}\beta_{k}=\ell_{P}$. 
\end{prop}

\begin{proof}
Let $\beta_{k}=\sum^{n}_{i=1}c^{k}_{i}x_{i}+c^{k}_{0}$. Take $\alpha=1$
and $x_{j}$ ($1\le j\le n$) in (\ref{eq: def beta_k}), we obtain
a system of linear equations satisfying by $c^{k}_{i}$, 
\begin{align*}
\sum^{n}_{i=1}\left\langle 1,x_{i}\right\rangle _{k}c^{k}_{i}+\left\langle 1,1\right\rangle _{k}c^{k}_{0} & =L_{k}(1)\\
\sum^{n}_{i=1}\left\langle x_{j},x_{i}\right\rangle _{k}c^{k}_{i}+\left\langle x_{j},1\right\rangle _{k}c^{k}_{0} & =L_{k}(x_{j}),\ 1\le j\le n.
\end{align*}
By the expansion of Riemann sums, we have 
\[
\left\langle x_{j},x_{i}\right\rangle _{k}\ra\int_{P}x_{j}x_{i},\ L_{k}(x_{j})\ra\int_{\partial P}x_{j},\ \text{etc.}
\]
Since the coefficients of system tends to that of the system for $\ell_{P}$,
we have $\beta_{k}\ra\ell_{P}$. 
\end{proof}

We introduce 
\[
L^{\mathrm{rel}}_{k}(\eta)\coloneqq L_{k}(\eta)-\left\langle \eta,\beta_{k}\right\rangle _{k}.
\]
Since (\ref{eq: def beta_k}), $L^{\mathrm{rel}}_{k}(\eta)$ is affine
invariant. The minimizer of 
\[
\Psi^{\mathrm{rel}}_{k}(\eta)\coloneqq\frac{1}{2}\left\Vert \eta\right\Vert ^{2}_{k}+L^{\mathrm{rel}}_{k}(\eta)
\]
is $\psi^{\perp}_{k}\coloneqq\psi_{k}-\pi_{k}(\psi_{k})=\psi_{k}+\beta_{k}$.

Fix a point $x_{0}\in\mathrm{int}P$, choose a support affine function
$\alpha_{k}$ of $\psi^{\perp}_{k}$ at $x_{0}$, and set 
\[
\phi_{k}\coloneqq\psi^{\perp}_{k}-\alpha_{k}.
\]
Then $\phi_{k}$ is normalized at $x_{0}$. 
\begin{prop}
\label{prop: boundedness}Suppose $P$ is a Delzant lattice polytope,
then 
\[
\left\Vert \psi^{\perp}_{k}\right\Vert _{k},\ \left\Vert \phi_{k}\right\Vert _{k},\ \sup_{P}\left|\alpha_{k}\right|,\ \int_{\partial P}\phi_{k},\ \int_{P}\phi^{2}_{k}
\]
are all uniformly bounded. 
\end{prop}

\begin{proof}
Since $\psi^{\perp}_{k}$ minimizes $\Psi^{\mathrm{rel}}_{k}$, we
have 
\begin{equation}
0=\Psi^{\mathrm{rel}}_{k}(0)\ge\Psi^{\mathrm{rel}}_{k}(\psi^{\perp}_{k})=\frac{1}{2}\left\Vert \psi^{\perp}_{k}\right\Vert ^{2}_{k}+L^{\mathrm{rel}}_{k}(\phi_{k}),\label{eq:rel psi bound}
\end{equation}
where we used the affine invariance of $L^{\mathrm{rel}}_{k}$. Employ
the fine estimate (\ref{eq:key estimate}), note that $1-2^{-\mathrm{cd}(u)}\ge\frac{1}{2}$
for $u\in P_{k}\cap\partial P$, we have 
\begin{align}
L_{k}(\phi_{k}) & \ge\frac{2}{k^{n-1}}\left(\left(\frac{1}{2}-\frac{C_{1}}{k}\right)\sum_{u\in P_{k}\cap\partial P}\phi_{k}(u)-\frac{C_{2}}{k^{2}}\sum_{u\in P_{k}}\phi_{k}(u)\right)\label{eq: estimate Lk}\\
\left(\text{suppose }\frac{C_{1}}{k}<\frac{1}{4}\right)\quad & \ge\frac{1}{2k^{n-1}}\sum_{u\in P_{k}\cap\partial P}\phi_{k}(u)-\frac{2C_{2}}{k^{n+1}}\sum_{u\in P_{k}}\phi_{k}(u)\nonumber \\
\text{C-S inequality}\quad & \ge\frac{1}{2k^{n-1}}\sum_{u\in P_{k}\cap\partial P}\phi_{k}(u)-\frac{C^{\prime}_{2}}{k}\left\Vert \phi_{k}\right\Vert _{k}.\nonumber 
\end{align}
Since $\beta_{k}\ra\ell_{P}$, it is uniformly bounded, 
\[
\left|\left\langle \phi_{k},\beta_{k}\right\rangle _{k}\right|\le\frac{C_{\beta}}{k^{n}}\sum_{u\in P_{k}}\phi_{k}(u)\le C^{\prime}_{\beta}\left\Vert \phi_{k}\right\Vert _{k}.
\]
Bring these estimates to (\ref{eq:rel psi bound}), 
\[
0\ge\Psi^{\mathrm{rel}}_{k}(\psi^{\perp}_{k})=\frac{1}{2k^{n-1}}\sum_{u\in P_{k}\cap\partial P}\phi_{k}(u)+\frac{1}{2}\left\Vert \psi^{\perp}_{k}\right\Vert ^{2}_{k}-\left(C^{\prime}_{\beta}+\frac{C^{\prime}_{2}}{k}\right)\left\Vert \phi_{k}\right\Vert _{k}.
\]
Since $\phi_{k}-\pi_{k}\phi_{k}=\psi^{\perp}_{k}$, (\ref{eq:projection control})
gives 
\[
\left\Vert \phi_{k}\right\Vert _{k}\le C_{\pi}\left\Vert \psi^{\perp}_{k}\right\Vert _{k}.
\]
Since $\phi_{k}\ge0$, we obtain
\[
0\ge\frac{1}{2}\left\Vert \psi^{\perp}_{k}\right\Vert ^{2}_{k}-\left(C^{\prime}_{\beta}+\frac{C^{\prime}_{2}}{k}\right)C_{\pi}\left\Vert \psi^{\perp}_{k}\right\Vert _{k}.
\]
It implies the boundedness of $\left\Vert \psi^{\perp}_{k}\right\Vert _{k}$,
$\left\Vert \phi_{k}\right\Vert _{k}$ and $\frac{1}{k^{n-1}}\sum_{u\in P_{k}\cap\partial P}\phi_{k}(u).$
Then $\int_{\partial P}\phi_{k}\mathrm{d}\sigma$ is bounded since
(\ref{eq: b key estimate}). Since $\alpha_{k}=\psi^{\perp}_{k}-\phi_{k}$,
$\left\Vert \alpha_{k}\right\Vert _{k}$ is bounded, thus $\{\alpha_{k}\}$
are uniformly bounded by (\ref{eq: finite dim norms}). 

Finally, since $\phi_{k}\ge0$, $\phi^{2}_{k}$ is also convex. Employ
the rough estimate (\ref{eq:rough bound}), 
\[
\fint_{P}\phi^{2}_{k}\le C_{P}\left\Vert \phi_{k}\right\Vert ^{2}_{k}\le C.
\]
\end{proof}

\subsection{Convergence \label{subsec:Convergence}}

First we recall a compactness result for normalized convex functions. 
\begin{prop}[{\cite[Proposition 5.2.6]{donaldson_scalar_2002}}]
\label{prop: Donaldson compact}Let $P$ be a lattice polytope with
a point $x_{0}\in\mathrm{int}P$. If $\phi_{k}\in\mathcal{C}(P)$
is a sequence of normalized convex functions such that $\int_{\partial P}\phi_{k}\mathrm{d}\sigma\le C$,
then there exists a subsequence locally uniformly converges to a function
$\phi\in\mathcal{C}_{1}(P)$, and 
\begin{equation}
\int_{\partial P}\phi\mathrm{d}\sigma\le\liminf_{k}\int_{\partial P}\phi_{k}\mathrm{d}\sigma.\label{eq: Donald liminf}
\end{equation}
\end{prop}

The following lemma will be used latter. 
\begin{lem}
\label{lem: little}Let $P$ be a lattice polytope, $\psi_{k}\in C^{0}(P)$
locally uniformly converges to $\psi\in L^{1}(P)$ on $\mathrm{int}P$,
and $\left\Vert \psi_{k}\right\Vert _{k}\le C_{\psi}$ for all $k$.
Another sequence $\gamma_{k}\in C^{0}(P)$ uniformly converges to
$\gamma$. Then we have 
\[
\lim_{k}\left\langle \psi_{k},\gamma_{k}\right\rangle _{k}=\int_{P}\psi\gamma\mathrm{d}x.
\]
\end{lem}

\begin{proof}
Suppose $\sup_{P}\left|\gamma_{k}\right|\le C_{\gamma}$ for all $k$.
For any $\varepsilon>0$, we can find a subpolytope $K\subset\mathrm{int}P$,
such that 
\[
2C_{\gamma}C_{\psi}\left|P\backslash K\right|^{1/2}<\varepsilon\text{ and }\left|\int_{P\backslash K}\psi\gamma\right|<\varepsilon.
\]
We have 
\[
\left|\left\langle \psi_{k},\gamma_{k}\right\rangle _{k}-\int_{P}\psi\gamma\right|\le\left|\frac{1}{k^{n}}\sum_{P_{k}\cap K}\psi_{k}\gamma_{k}-\int_{K}\psi\gamma\right|+\left|\frac{1}{k^{n}}\sum_{P_{k}\backslash K}\psi_{k}\gamma_{k}\right|+\left|\int_{P\backslash K}\psi\gamma\right|.
\]
The first term on the RHS tends to zero since $\psi_{k}\gamma_{k}\ra\psi\gamma$
uniformly on $K$. By C-S inequality, 
\[
\text{middle term }\le C_{\gamma}\frac{\left|P_{k}\backslash K\right|^{1/2}}{k^{n/2}}\left\Vert \psi_{k}\right\Vert _{k}\le2C_{\gamma}C_{\psi}\left|P\backslash K\right|^{1/2}<\varepsilon.
\]
Hence there exists $N_{\varepsilon}>0$ such that the sum of three
terms is $<3\varepsilon$ for $k>N_{\varepsilon}$. 
\end{proof}

\subsubsection{Prove part (1) of Theorem \ref{thm:main converge}. }
\begin{proof}
Suppose that $\psi_{k}$ does not $L^{2}$-converge to $\psi$, then
there exists $\delta>0$ and a subsequence $\psi_{k_{j}}$ such that
\[
\left\Vert \psi_{k_{j}}-\psi\right\Vert _{L^{2}}\ge\delta,\ \forall j\ge1.
\]
We will show that $\psi_{k_{j}}$ admits a subsequence which $L^{2}$-converges
to $\psi$. It is a contradiction, we are done. For simplicity of
notation, we denote $\psi_{k_{j}}$ by $\psi_{k}$.

\textbf{Step 1}. Find the subsequence.

By Proposition \ref{prop: boundedness}, we have uniform boundedness,
\[
\left\Vert \phi_{k}\right\Vert _{k}+\sup_{P}\left|\alpha_{k}\right|+\int_{\partial P}\phi_{k}+\int_{P}\phi^{2}_{k}\le C.
\]
Using the weak compactness in $L^{2}(P)$ and Proposition \ref{prop: Donaldson compact},
we can take a subsequence (still denoted by $\psi_{k},\phi_{k}$)
such that 
\[
\phi_{k}\rightharpoonup\phi_{*}\text{ in }L^{2}(P),\ \phi_{k}\rightarrow\phi_{*}\text{ locally uniformly on }\mathrm{int}P,
\]
where $\phi_{*}\in\mathcal{C}_{1}(P)\cap L^{2}(P)$. Moreover, by
(\ref{eq: Donald liminf}) we have 
\begin{equation}
\liminf_{k}\int_{\partial P}\phi_{k}\ge\int_{\partial P}\phi_{*}.\label{eq: Donald liminf 2}
\end{equation}

Since $\alpha_{k}\in\mathcal{A}$ is also uniformly bounded, we can
assume $\alpha_{k}\ra\alpha\in\mathcal{A}$ after passing to a subsequence.
Recall that 
\[
\psi_{k}=\psi^{\bot}_{k}-\beta_{k}=\phi_{k}+\alpha_{k}-\beta_{k}\text{ and }\beta_{k}\ra\ell_{P},
\]
we have 
\[
\psi_{k}\rightharpoonup\psi_{*}\coloneqq\phi_{*}+\alpha-\ell_{P}\text{ in }L^{2}(P),\ \psi_{k}\rightarrow\psi_{*}\text{ locally uniformly on }\mathrm{int}P.
\]
Since $\left\Vert \psi_{k}\right\Vert _{k}$ is also bounded, we may
assume $\left\Vert \psi_{k}\right\Vert _{k}$ converges. We will not
take further subsequence in the following.

\textbf{Step 2}. show $\psi_{*}=\psi\coloneqq-\Theta_{P}$ and $\left\Vert \psi_{k}\right\Vert _{k}\ra\left\Vert \psi\right\Vert _{L^{2}}$.

It suffices to show 
\begin{align}
\liminf_{k}\left\Vert \psi_{k}\right\Vert ^{2}_{k} & \ge\int_{P}\psi^{2}_{*},\text{ and }\label{eq: liminf L2 norm}\\
\liminf_{k}L_{k}(\psi_{k}) & \ge\int_{\partial P}\psi_{*}\mathrm{d}\sigma.\label{eq: liminf Lk}
\end{align}
Once these hold, since $\psi_{k}$ and $\psi$ are the minimizer of
$\Psi_{k}$ and $\Psi$ respectively, we have 
\begin{align*}
\Psi(\psi)=\lim_{k}\Psi_{k}(\psi) & \ge\limsup_{k}\Psi_{k}(\psi_{k})\ge\liminf_{k}\Psi_{k}(\psi_{k})\\
 & \ge\liminf_{k}\frac{1}{2}\left\Vert \psi_{k}\right\Vert ^{2}_{k}+\liminf_{k}L_{k}(\psi_{k})\\
 & \ge\frac{1}{2}\int_{P}\psi^{2}_{*}+\int_{\partial P}\psi_{*}\mathrm{d}\sigma=\Psi(\psi_{*})\ge\Psi(\psi),
\end{align*}
where the first equality is due to the continuity of $\Theta_{P}$
and (\ref{eq:EM}). All inequalities become equalities, thus 
\begin{equation}
\lim_{k}\Psi_{k}(\psi_{k})=\Psi(\psi)\text{ and }\Psi(\psi_{*})=\Psi(\psi).\label{eq:inf converge}
\end{equation}
The uniqueness of minimizer implies $\psi_{*}=\psi$. (\ref{eq: liminf L2 norm})
also takes equality, so 
\begin{equation}
\lim_{k}\left\Vert \psi_{k}\right\Vert _{k}=\left\Vert \psi\right\Vert _{L^{2}},\label{eq: normk converge}
\end{equation}
since we know that $\left\Vert \psi_{k}\right\Vert _{k}$ converges
in Step1. 

To prove (\ref{eq: liminf L2 norm}), take any compact $K\subset\mathrm{int}P$,
by uniform convergence, 
\[
\liminf_{k}\left\Vert \psi_{k}\right\Vert ^{2}_{k}\ge\lim_{k}\frac{1}{k^{n}}\sum_{u\in P_{k}\cap K}\psi_{k}(u)^{2}=\int_{K}\psi^{2}_{*}.
\]
Take the supremum over $K$, (\ref{eq: liminf L2 norm}) follows.

For (\ref{eq: liminf Lk}), combine the estimate (\ref{eq: estimate Lk})
and (\ref{eq: b key estimate}), we have 
\begin{align}
L_{k}(\phi_{k}) & \ge\left(1-\frac{2C_{1}}{k}\right)\frac{1}{k^{n-1}}\sum_{u\in P_{k}\cap\partial P}\phi_{k}(u)-\frac{2C_{2}}{k^{n+1}}\sum_{u\in P_{k}}\phi_{k}(u)\label{eq:key usage}\\
 & \ge\left(1-\frac{2C_{1}}{k}\right)\left(1+\frac{C_{b}}{k}\right)^{-1}\int_{\partial P}\phi_{k}\mathrm{d}\sigma-\frac{2C_{2}}{k^{n+1}}\sum_{u\in P_{k}}\phi_{k}(u).\nonumber 
\end{align}
The last term tends to zero since $\left\Vert \phi_{k}\right\Vert _{k}\le C$.
Take $\liminf_{k}$ on both sides, use (\ref{eq: Donald liminf 2}),
we obtain 
\[
\liminf_{k}L_{k}(\phi_{k})\ge\int_{\partial P}\phi_{*}\mathrm{d}\sigma.
\]
This implies (\ref{eq: liminf Lk}) since $\alpha_{k}\ra\alpha$ and
$\beta_{k}\ra\ell_{P}$ uniformly.

\textbf{Step 3}. Show the subsequence $\psi_{k}\xrightarrow{L^{2}}\psi$.

Since $\phi^{2}_{k}$ is also convex, apply (\ref{eq:key estimate}),
\[
\int_{P}\phi^{2}_{k}\le\frac{1}{k^{n}}\sum_{u\in P_{k}}\phi_{k}(u)^{2}+\frac{C_{1}}{k^{n+1}}\sum_{u\in P_{k}\cap\partial P}\phi_{k}(u)^{2}+\frac{C_{2}}{k^{n+2}}\sum_{u\in P_{k}}\phi_{k}(u)^{2}
\]
The last two terms tends to zero since $\left\Vert \phi_{k}\right\Vert _{k}\le C$.
Take $\limsup$ on both sides, 
\[
\limsup_{k}\int_{P}\phi^{2}_{k}\le\limsup_{k}\left\Vert \phi_{k}\right\Vert ^{2}_{k}.
\]
We show the RHS is $\int_{P}\phi^{2}_{*}$. Write $\psi_{k}=\phi_{k}+\gamma_{k}$,
where $\gamma_{k}\coloneqq\alpha_{k}-\beta_{k}\ra\gamma\coloneqq\alpha-\ell_{P}$.
Use (\ref{eq: normk converge}) and Lemma \ref{lem: little}, 
\[
\left\Vert \phi_{k}\right\Vert ^{2}_{k}=\left\Vert \psi_{k}\right\Vert ^{2}_{k}-2\left\langle \psi_{k},\gamma_{k}\right\rangle +\left\Vert \gamma_{k}\right\Vert ^{2}_{k}\ra\left\Vert \psi\right\Vert ^{2}_{L^{2}}-2\int_{P}\psi\gamma+\int_{P}\gamma^{2}=\int_{P}\phi^{2}_{*}.
\]
Hence $\limsup\int_{P}\phi^{2}_{k}\le\int_{P}\phi^{2}_{*}$. Lower
semicontinuity of $L^{2}$-norm (for weak convergence) gives the converse
direction: $\liminf\int_{P}\phi^{2}_{k}\ge\int_{P}\phi^{2}_{*}$.
Hence we have 
\[
\lim_{k}\left\Vert \phi_{k}\right\Vert _{L^{2}}=\left\Vert \phi_{*}\right\Vert _{L^{2}}.
\]
This implies $\phi_{k}\xrightarrow{L^{2}}\phi_{*}$ since we already
have $\phi_{k}\rightharpoonup\phi_{*}$. So 
\[
\psi_{k}=\phi_{k}+\gamma_{k}\xrightarrow{L^{2}}\phi_{*}+\gamma=\psi.
\]
This is just the $L^{2}$-convergence of the subsequence selected
in Step 1. As stated at the beginning of the proof, by a contradiction
argument, we can obtain the $L^{2}$-convergence of the full sequence.
By (\ref{eq:inf converge}), we also have $\lim_{k}\Psi_{k}(\psi_{k})=\Psi(\psi)$. 

For the locally uniformly convergence, if there is a compact set $K\subset\mathrm{int}P$
such that $\sup_{K}\left|\psi_{k}-\psi\right|\not\rightarrow0$. By
the above argument, we can find a subsequence which locally uniformly
converges to $\psi$. It is a contradiction.

Finally, we show $\mathbb{M}^{\mathrm{rel}}_{P}(\psi_{k})\ra\mathbb{M}^{\mathrm{rel}}_{P}(\psi)$.
Since $\psi_{k}\xrightarrow{L^{2}}\psi$, it suffices to show $\int_{\partial P}\psi_{k}\ra\int_{\partial P}\psi$.
For any subsequence $\psi_{k_{j}}$, we show it admits a further subsequence
so that its boundary integral converges to $\int_{\partial P}\psi$.
Then we are done. 

Choose a subsequence as in Step 1 above, and still denote it by $\psi_{k}$.
We can use all the results obtained in Step 2 and 3. Combine $\lim\Psi_{k}(\psi_{k})=\Psi(\psi)$
with (\ref{eq: normk converge}), we have $\lim L_{k}(\psi_{k})=\int_{\partial P}\psi$.
Recall $\psi_{k}=\phi_{k}+\gamma_{k}$ and $\gamma_{k}\rightrightarrows\gamma$,
it follows that 
\[
\lim_{k}L_{k}(\phi_{k})=\int_{\partial P}\phi_{*},\  \phi_{*}=\psi-\gamma.
\]
Take $\limsup_{k}$ on both sides of (\ref{eq:key usage}), we obtain
\[
\int_{\partial P}\phi_{*}=\lim_{k}L_{k}(\phi_{k})\ge\limsup_{k}\int_{\partial P}\phi_{k}.
\]
(\ref{eq: Donald liminf 2}) gives the reverse side, hence $\int_{\partial P}\phi_{k}\ra\int_{\partial P}\phi_{*}$.
This implies $\int_{\partial P}\psi_{k}\ra\int_{\partial P}\psi$
since $\gamma_{k}\rightrightarrows\gamma$. 
\end{proof}

\begin{rem}
\label{rem: fail nonDelzant}The most delicate step in the proof is
(\ref{eq:key usage}). It requires an estimate of the form: 
\[
k^{n}\int_{P}\phi\le\sum_{u\in P_{k}}w(u)\phi(u)+\frac{C_{1}}{k}\sum_{u\in P_{k}\cap\partial P}\phi(u)+\frac{C_{2}}{k^{2}}\sum_{u\in P_{k}}\phi(u)
\]
holds for all convex $\phi\ge0$, and the weight $w(u)$ satisfies
\[
\begin{cases}
w(u)=1, & u\in\mathrm{int}P,\\
w(u)=1/2, & \mathrm{cd}(u)=1,\\
w(u)\le1/2, & \mathrm{cd}(u)>1.
\end{cases}
\]
Such estimate may not exist for non-Delzant polytope, for instance
\[
P=\mathrm{conv}\left((0,0),(1,0),(1,m)\right),\ m\ge4,
\]
which is singular at the vertex $(0,0)$. We take $\phi=\max\left\{ 1-kx_{1},0\right\} $
and compute 
\[
k^{2}\int_{P}\phi=\frac{m}{6}\text{ and }\phi(u)=0,\ \forall u\in P_{k}\backslash\{(0,0)\}.
\]
The above inequality implies 
\[
\frac{2}{3}\le\frac{m}{6}\le w(0,0)+\frac{C_{1}}{k}+\frac{C_{2}}{k^{2}}.
\]
Since $w(0,0)\le1/2$, for any $C_{1},C_{2}$, we can find a sufficiently
large $k$ making this inequality fail. 
\end{rem}

\subsubsection{Prove part (2) of Theorem \ref{thm:main converge}. }
\begin{proof}
This is also a typical $\Gamma$-convergence argument as the Step
2 above, but we use $C^{0}$-topology. In the following, convergence
always means the uniform convergence on $P$.

We assume that $\{\psi_{k}\}$ is $C^{0}$-precompact. Since $\Psi_{k}(\psi)\ge\Psi_{k}(\psi_{k})$,
we have 
\[
\Psi(\psi)=\lim_{k}\Psi_{k}(\psi)\ge\limsup_{k\rightarrow\infty}\Psi_{k}(\psi_{k}),
\]
where the first equality is due to the continuity of $\Theta_{P}$
and (\ref{eq:EM}).

If one can show that each subsequence of $\psi_{k}$ has a further
subsequence which uniformly converges to $\psi$, then by a contradiction
argument, we obtain $\psi_{k}\ra\psi$ uniformly. Without loss of
generality, we only consider the sequence $\psi_{k}$ itself.

By the precompactness assumption, $\psi_{k}$ has a subsequence $\psi_{k_{j}}$
which uniformly converges to $\psi_{*}\in\mathcal{C}(P)$. It remains
to show $\psi_{*}=\psi$. Apply (\ref{eq:EM enhanced}), we have 
\[
\liminf_{j\rightarrow\infty}\Psi_{k_{j}}(\psi_{k_{j}})=\Psi(\psi_{*})\geq\Psi(\psi).
\]
We already have $\Psi(\psi)\ge\limsup_{j}\Psi_{k_{j}}(\psi_{k_{j}})$.
It forces $\Psi(\psi_{*})=\Psi(\psi)$, so $\psi_{*}=\psi$ by the
uniqueness of minimizer.
\end{proof}

\appendix

\section{Two-term asymptotics for the Riemann sum of convex functions \label{Appendix: Riemann sum}}

We establish a folklore expansion of Riemann sum for convex functions. 
\begin{thm}
Let $P\subset\mathbb{R}^{n}$ be a lattice polytope. For any $k\in\mathbb{Z}_{>0}$,
set $P_{k}=P\cap k^{-1}\mathbb{Z}^{n}$. Let $f_{k}:P\rightarrow\mathbb{R}$
be a sequence of convex functions (continuous up to $\partial P$)
which uniformly converges to $f$. Then we have 
\begin{equation}
\lim_{k\rightarrow\infty}\frac{1}{k^{n-1}}\left(\sum_{u\in P_{k}}f_{k}(u)-k^{n}\int_{P}f_{k}\ \mathrm{d}x\right)=\frac{1}{2}\int_{\partial P}f\ \mathrm{d}\sigma.\label{eq:EM enhanced}
\end{equation}
In particular, when $f_{k}=f$ for all $k$, we have expansion 
\begin{equation}
\sum_{u\in P_{k}}f(u)=k^{n}\int_{P}f\ \mathrm{d}x+\frac{k^{n-1}}{2}\int_{\partial P}f\ \mathrm{d}\sigma+\mathrm{o}(k^{n-1}).\label{eq:EM}
\end{equation}
In the above, $\mathrm{d}x$ is the standard Lebesgue measure, and
the boundary measure $\mathrm{d}\sigma$ is defined as follows. For
each facet $F\subset\partial P$ spanning affine subspace $\mathrm{Aff}(F)$,
$\mathrm{d}\sigma|_{F}$ is the Lebesgue measure normalized by the
lattice $\mathrm{Aff}(F)\cap\mathbb{Z}^{n}$. 
\end{thm}

\begin{rem}
(1) when $f$ is smooth on an open set containing $P$, (\ref{eq:EM})
is just the leading two terms of the Euler--Maclaurin expansion.
There is vast literature discussing it, see \cite{tate_asymptotic_2011}
and its references. (2) if $f$ is merely continuous, (\ref{eq:EM})
may fail even in one dimension (Weierstrass type functions). (3) If
$f$ is convex and piece-wisely linear with rational coefficients,
(\ref{eq:EM}) can be derived from the asymptotics of lattice point
counting, see \cite[Prop. 4.1.3]{donaldson_scalar_2002}. In our setting,
(\ref{eq:EM}) cannot be derived from these results via a direct approximation.
The key in the following proof is decomposing the Riemann sum into
interior part and boundary part, and we handle the boundary part via
approximation. 
\end{rem}

\begin{proof}
\textbf{Step 1}. We extend $f_{k},f$ to $\mathbb{R}^{n}$ such that
uniform convergence still holds and they share a common continuity
modulus $\hat{\omega}$, i.e. for $g=f_{k}$ and $f$, we have 
\[
\left|g(x_{1})-g(x_{2})\right|\leq\hat{\omega}\left(\left\Vert x_{1}-x_{2}\right\Vert \right),\ \forall x_{i}\in\mathbb{R}^{n}.
\]
The above norm is $\left\Vert x\right\Vert \coloneqq\max_{1\leq i\leq n}\left|x_{i}\right|$,
so we use the distance $d(x,y)=\left\Vert x-y\right\Vert $. Let 
\[
\omega_{k}(\delta)\coloneqq\sup\left\{ \left|f_{k}(x)-f_{k}(y)\right|\mid x,y\in P,\ d(x,y)\leq\delta\right\} ,
\]
be the minimal continuity modulus of $f_{k}$. By the uniform convergence,
it is easy to see 
\[
\omega(\delta)\coloneqq\sup_{k}\omega_{k}(\delta)\rightarrow0,\text{ as }\delta\rightarrow0.
\]
Thus for $g=f_{k}$ or $f$, we have 
\[
\left|g(x)-g(y)\right|\leq\omega\left(\left\Vert x-y\right\Vert \right),\ \forall x,y\in P.
\]
To apply McShane's extension, the continuity modulus is required to
be subadditive. So we take $\hat{\omega}(\delta)$ be the concave
envelope of function $0\leq\delta\mapsto\omega(\delta)$. Since $\hat{\omega}(\delta)\rightarrow0$,
it is subadditive, i.e. $\hat{\omega}(\delta_{1}+\delta_{2})\leq\hat{\omega}(\delta_{1})+\hat{\omega}(\delta_{2})$.
Then we define McShane's extension: 
\[
\tilde{f}_{k}(x)\coloneqq\inf_{y\in P}f_{k}(y)+\hat{\omega}\left(\left\Vert x-y\right\Vert \right),\ \forall x\in\mathbb{R}^{n},
\]
and $\tilde{f}(x)$ is defined similarly. It coincides with original
function on $P$. 

They share a common continuity modulus $\hat{\omega}$. For any $x_{1},x_{2}\in\mathbb{R}^{n}$
and $y\in P$, by the subadditivity, we have 
\[
\tilde{f}_{k}(x_{1})\leq f_{k}(y)+\hat{\omega}\left(\left\Vert x_{1}-y\right\Vert \right)\leq f_{k}(y)+\hat{\omega}\left(\left\Vert x_{2}-y\right\Vert \right)+\hat{\omega}\left(\left\Vert x_{1}-x_{2}\right\Vert \right).
\]
Take $\inf_{y\in P}$ on both sides, 
\[
\tilde{f}_{k}(x_{1})\leq\tilde{f}_{k}(x_{2})+\hat{\omega}\left(\left\Vert x_{1}-x_{2}\right\Vert \right).
\]
Swap $x_{1},x_{2}$, we obtain 
\[
\left|\tilde{f}_{k}(x_{1})-\tilde{f}_{k}(x_{2})\right|\leq\hat{\omega}\left(\left\Vert x_{1}-x_{2}\right\Vert \right).
\]
This also holds for $\tilde{f}$. 

Finally, we show that $\tilde{f}_{k}$ uniformly converges to $\tilde{f}$
on $\mathbb{R}^{n}$. Take any $\varepsilon>0$, we have $\left|f_{k}-f\right|<\varepsilon$
on $P$ when $k>N$. Then for any $x\in\mathbb{R}^{n}$ and $y\in P$,
we have 
\[
f(y)+\hat{\omega}\left(\left\Vert x-y\right\Vert \right)-\varepsilon\leq f_{k}(y)+\hat{\omega}\left(\left\Vert x-y\right\Vert \right)\leq f(y)+\hat{\omega}\left(\left\Vert x-y\right\Vert \right)+\varepsilon.
\]
Take $\inf_{y\in P}$ on both sides, we obtain $\left|\tilde{f}_{k}(x)-\tilde{f}(x)\right|<\varepsilon$
for all $x\in\mathbb{R}^{n}$ when $k>N$. 

We will simply denote $\tilde{f}_{k},\tilde{f}$ by $f_{k}$ and $f$
respectively. 

\textbf{Step 2}: interior-boundary decomposition. 

For a (scaled) lattice point $u\in k^{-1}\mathbb{Z}^{n}$, we define
a cell (or box)
\[
C_{k}(u)\coloneqq u+\left(-\frac{1}{2k},\frac{1}{2k}\right)^{n}=\left\{ x\in\mathbb{R}^{n}\mid\left\Vert x-u\right\Vert <\frac{1}{2k}\right\} 
\]
with volume $\left|C_{k}(u)\right|=k^{-n}$. Let 
\[
A_{k}\coloneqq\left\{ u\in k^{-1}\mathbb{Z}^{n}\mid C_{k}(u)\subset P\right\} ,\ B_{k}\coloneqq\left\{ u\in k^{-1}\mathbb{Z}^{n}\mid C_{k}(u)\cap\partial P\neq\emptyset\right\} ,
\]
refer to Figure \ref{fig: EM proof} below. A key fact: there exists a constant $C_{P}$
such that 
\[
\left|A_{k}\right|\leq C_{P}k^{n},\ \left|B_{k}\right|\leq C_{P}k^{n-1}\text{ for }\forall k\geq1.
\]
We decompose the sum and integral ($\chi_{P}$ is the indicator function),
\[
\sum_{u\in P_{k}}f_{k}(u)=\sum_{u\in A_{k}}f_{k}(u)+\sum_{u\in B_{k}}f_{k}(u)\chi_{P}(u),
\]
\begin{equation}
\int_{P}f_{k}=\sum_{u\in A_{k}}\int_{C_{k}(u)}f_{k}+\sum_{u\in B_{k}}\int_{C_{k}(u)\cap P}f_{k}.\label{eq:int f_k}
\end{equation}
We try to find cancellations between them (modulo $\mathrm{o}(k^{n-1})$-errors). 

For the integral $\int_{C_{k}(u)\cap P}f_{k}$, by the uniform continuity,
we have 
\[
\left|\int_{C_{k}(u)\cap P}\left(f_{k}(x)-f_{k}(u)\right)\mathrm{d}x\right|\leq k^{-n}\hat{\omega}\left(\frac{1}{2k}\right).
\]
Thus 
\begin{align*}
\sum_{u\in B_{k}}\int_{C_{k}(u)\cap P}f_{k} & =\sum_{u\in B_{k}}f_{k}(u)\left|C_{k}(u)\cap P\right|+r_{k},\\
\text{where }\left|r_{k}\right| & \leq\left|B_{k}\right|k^{-n}\hat{\omega}\left(\frac{1}{2k}\right)\leq\frac{C_{P}}{k}\hat{\omega}\left(\frac{1}{2k}\right).
\end{align*}
Put this back into (\ref{eq:int f_k}), we obtain 
\begin{align}
\sum_{u\in P_{k}}f_{k}(u)-k^{n}\int_{P}f_{k} & =\sum_{u\in A_{k}}\left(f_{k}(u)-k^{n}\int_{C_{k}(u)}f_{k}\right)+\sum_{u\in B_{k}}f_{k}(u)w_{k}(u)+\mathrm{o}\left(k^{n-1}\right)\nonumber \\
\text{define} & \eqqcolon k^{n-1}\mathcal{A}_{k}(f_{k})+k^{n-1}\mathcal{B}_{k}(f_{k})+\mathrm{o}\left(k^{n-1}\right),
\end{align}
where the weight $w_{k}(u)$ is 
\[
w_{k}(u)\coloneqq\chi_{P}(u)-k^{n}\left|C_{k}(u)\cap P\right|=\begin{cases}
\frac{\left|C_{k}(u)\backslash P\right|}{\left|C_{k}(u)\right|}, & u\in P,\\
-\frac{\left|C_{k}(u)\cap P\right|}{\left|C_{k}(u)\right|}, & u\notin P.
\end{cases}
\]
Note that $\left|w_{k}(u)\right|\leq1$. Now, the question is reduced
to prove 
\[
\mathcal{A}_{k}(f_{k})\rightarrow0\text{ and }\mathcal{B}_{k}(f_{k})\rightarrow\frac{1}{2}\int_{\partial P}f\mathrm{d}\sigma.
\]

\textbf{Step 3}: show 
\begin{equation}
\mathcal{B}_{k}(f_{k})\coloneqq\frac{1}{k^{n-1}}\sum_{u\in B_{k}}f_{k}(u)w_{k}(u)\rightarrow\frac{1}{2}\int_{\partial P}f\mathrm{d}\sigma.\label{eq:EM step2}
\end{equation}
In this step, the convexity of $f_{k}$ is not used. We derive (\ref{eq:EM step2})
from (\ref{eq:EM} for smooth functions) via approximation, so it
is not a direct proof. Before that, we give an intuitive but not fully
rigorous demonstration. Readers may skip this part. 

\textit{A heuristic proof of} (\ref{eq:EM step2}):

\begin{figure}[htbp]
\centering
\includegraphics[clip,scale=0.25]{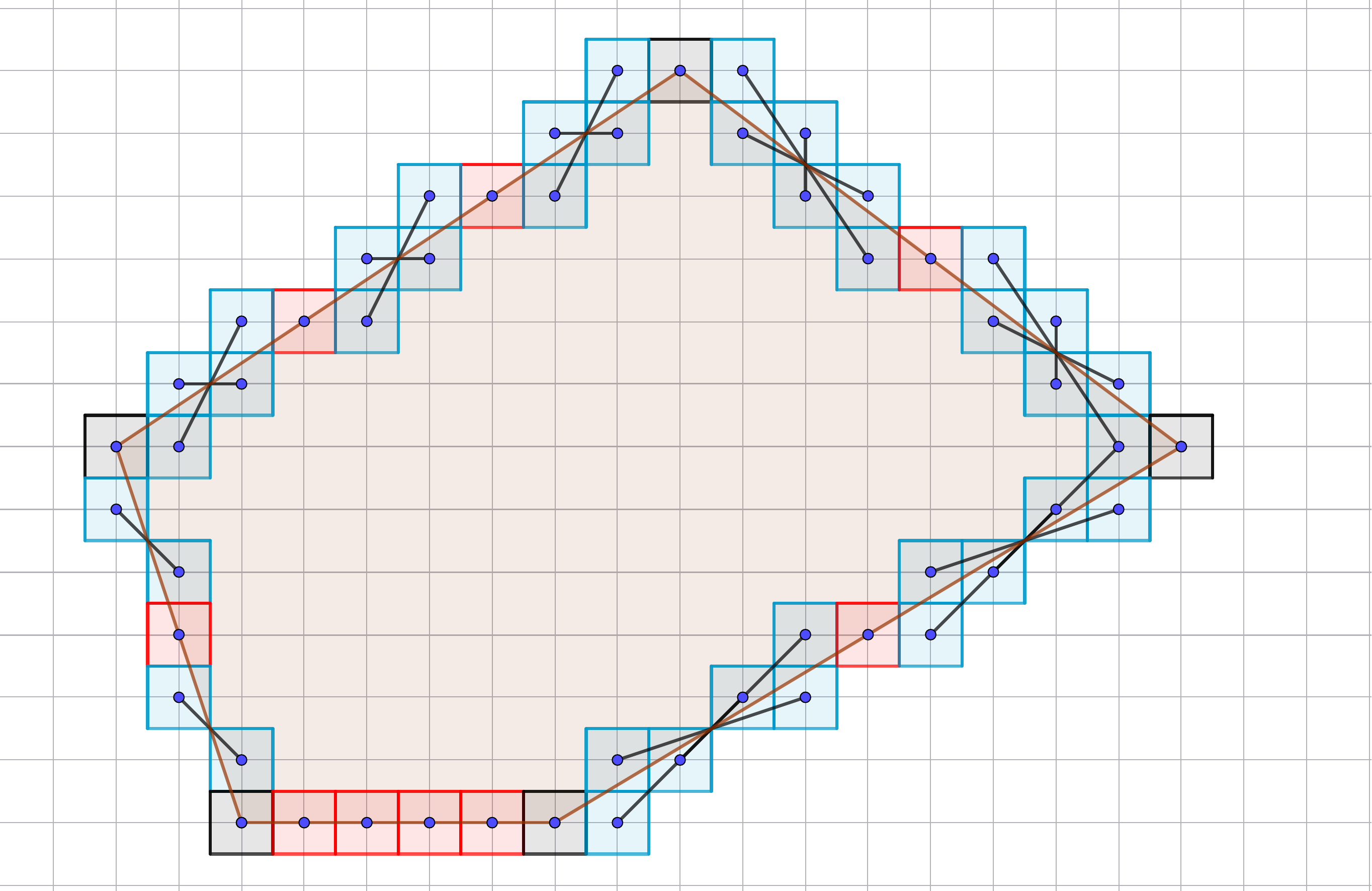}
\caption{\label{fig: EM proof} colored boxes are $C_{k}(u)$ with $u\in B_{k}$.}
\end{figure}

Let $\partial^{>1}P$
be the union of faces of $P$ with codimension $>1$, and we set 
\[
B^{>1}_{k}\coloneqq\left\{ u\in B_{k}\mid d_{\infty}\left(u,\partial^{>1}P\right)<\frac{1}{k}\right\} 
\]
(black cells in Figure \ref{fig: EM proof}). Note that $\left|B^{>1}_{k}\right|=\mathrm{O}\left(k^{n-2}\right)$,
so we have $\sum_{u\in B^{>1}_{k}}f_{k}(u)w_{k}(u)=\mathrm{O}\left(k^{n-2}\right)$,
so $B^{>1}_{k}$ summand makes no contribution to $\lim_{k}\mathcal{B}_{k}(f_{k})$. 

Consider the other summands with $u\in B_{k}\backslash B^{>1}_{k}$,
$u$ may lie either inside, outside or on the boundary of $P$. We
claim that only the summands that $u\in\partial P$ make contribution
to the limit $\lim_{k}\mathcal{B}_{k}$. As the figure illustrates,
except $\mathrm{O}\left(k^{n-2}\right)$ points, each $u\in\mathrm{int}P\backslash B^{>1}_{k}$
can be assigned a partner $u'$ outside of $P$ such that 
\[
w_{k}(u)=-w_{k}(u')\text{ and }d_{\infty}(u,u')\leq\frac{C}{k}.
\]
In Figure \ref{fig: EM proof}, partner points are connected by a line segment. For
the associated two terms in the sum $\mathcal{B}_{k}(f)$, we have
\[
\left|f_{k}(u)w_{k}(u)+f_{k}(u')w_{k}(u')\right|=\left|f_{k}(u)-f_{k}(u')\right|\left|w_{k}(u)\right|\leq\hat{\omega}\left(\frac{C}{k}\right).
\]
On the other hand, for $u\in B^{\partial}_{k}\coloneqq\left(B_{k}\cap\partial P\right)\backslash B^{>1}_{k}$,
we have $w_{k}(u)=\frac{1}{2}$. Combining all these estimates, we
conclude 
\[
\lim_{k\rightarrow\infty}\mathcal{B}_{k}(f_{k})=\lim_{k\rightarrow\infty}\frac{1}{k^{n-1}}\sum_{u\in B^{\partial}_{k}}\frac{1}{2}f_{k}(u)=\frac{1}{2}\int_{\partial P}f\mathrm{d}\sigma.
\]

\textit{Rigorous proof of} (\ref{eq:EM step2}). For any $\varepsilon>0$,
take a smooth function $g$ on $\mathbb{R}^{n}$ such that 
\[
\left|f-g\right|\leq\varepsilon\text{ on }U\left(P,1\right)\coloneqq\left\{ x\in\mathbb{R}^{n}\mid d(x,P)<1\right\} .
\]
We will show $\mathcal{A}_{k}(g)\rightarrow0$ below. Combine it with
(\ref{eq:EM}) for smooth functions \cite{tate_asymptotic_2011},
we have $\mathcal{B}_{k}(g)\rightarrow\frac{1}{2}\int_{\partial P}g\mathrm{d}\sigma$.
Hence there exists $K>0$ (depending on $g$) such that when $k>K$,
we have 
\[
\left|\mathcal{B}_{k}(g)-\frac{1}{2}\int_{\partial P}g\mathrm{d}\sigma\right|<\varepsilon,\ B_{k}\subset U\left(P,1\right),
\]
\[
\text{and }\left|f_{k}-f\right|<\varepsilon\text{ on }\mathbb{R}^{n}.
\]
Thus $\left|f_{k}-g\right|\leq2\varepsilon$ on $U\left(P,1\right)$.
These imply 
\begin{align*}
\left|\mathcal{B}_{k}(f_{k})-\frac{1}{2}\int_{\partial P}f\mathrm{d}\sigma\right| & \leq\left|\mathcal{B}_{k}(f_{k})-\mathcal{B}_{k}(g)\right|+\left|\mathcal{B}_{k}(g)-\frac{1}{2}\int_{\partial P}g\mathrm{d}\sigma\right|+\frac{1}{2}\int_{\partial P}\left|g-f\right|\mathrm{d}\sigma\\
 & \leq\frac{1}{k^{n-1}}\sum_{u\in B_{k}}\left|f_{k}(u)-g(u)\right|+\varepsilon+\frac{\varepsilon}{2}\left|\partial P\right|_{\mathrm{d}\sigma}\\
 & \leq\frac{2\varepsilon}{k^{n-1}}\left|B_{k}\right|+\varepsilon+\frac{\varepsilon}{2}\left|\partial P\right|_{\mathrm{d}\sigma}\leq C^{\prime}_{P}\varepsilon,\text{ for }k>K.
\end{align*}
Thus (\ref{eq:EM step2}) holds. 

Now we show $\mathcal{A}_{k}(g)\rightarrow0$. Since $g$ is smooth,
near each $u\in A_{k}$, we have Taylor expansion 
\begin{equation}
g(x)=g(u)+\nabla g(u)\cdot(x-u)+R_{u}(x),\ \left|R_{u}(x)\right|\leq C_{g}\left\Vert x-u\right\Vert ^{2},\label{eq:taylor}
\end{equation}
where the constant $C_{g}$ depends on the $C^{2}$-norm of $g$.
Integrating both sides over $C_{k}(u)$, we obtain 
\[
\int_{C_{k}(u)}g(x)\mathrm{d}x=k^{-n}g(u)+\int_{C_{k}(u)}R_{u}(x)\mathrm{d}x,
\]
where the integral of linear term vanishes by symmetry. Hence 
\[
\left|g(u)-k^{n}\int_{C_{k}(u)}g\right|\leq k^{n}\int_{C_{k}(u)}\left|R_{u}(x)\right|\mathrm{d}x\leq C_{g}k^{n}\int_{C_{k}(u)}\left\Vert x-u\right\Vert ^{2}\mathrm{d}x=\mathrm{O}\left(k^{-2}\right).
\]
It implies $\left|\mathcal{A}_{k}(g)\right|\leq\frac{1}{k^{n-1}}\left|A_{k}\right|\mathrm{O}\left(k^{-2}\right)=\mathrm{O}\left(k^{-1}\right)\rightarrow0$.

\textbf{Step 4}: show 
\[
-\mathcal{A}_{k}(f_{k})\coloneqq\frac{1}{k^{n-1}}\sum_{u\in A_{k}}\left(k^{n}\int_{C_{k}(u)}f_{k}-f_{k}(u)\right)\rightarrow0.
\]
By the convexity of $f_{k}$, each summand is nonnegative, thus $-\mathcal{A}_{k}(f_{k})\ge0$.
To bound $-\mathcal{A}_{k}(f_{k})$ from the above, we need to consider
more interior regions. Let 
\[
A^{\mathrm{in}}_{k}\coloneqq\left\{ u\in P_{k}\mid d_{\infty}(u,\partial P)\ge k^{-1}\right\} \subset A_{k}.
\]
Note that $\left|A_{k}\backslash A^{\mathrm{in}}_{k}\right|\le C_{P}k^{n-1}$,
this implies the sum over $A_{k}\backslash A^{\mathrm{in}}_{k}$ in
$-\mathcal{A}_{k}(f_{k})$ converges to zero, since we have 
\[
\sum_{u\in A_{k}\backslash A^{\mathrm{in}}_{k}}\left(k^{n}\int_{C_{k}(u)}f_{k}-f_{k}(u)\right)\le k^{n}\sum_{u\in A_{k}\backslash A^{\mathrm{in}}_{k}}\int_{C_{k}(u)}\left|f_{k}-f_{k}(u)\right|\le\left|A_{k}\backslash A^{\mathrm{in}}_{k}\right|\hat{\omega}\left(\frac{1}{2k}\right).
\]
It remains to handle the part $\sum_{u\in A^{\mathrm{in}}_{k}}$. 

For $u\in A^{\mathrm{in}}_{k}$, we define a function
\[
h_{u}(x)\coloneqq f_{k}(u+x)+f_{k}(u-x)-2f_{k}(u),\ \left\Vert x\right\Vert \le k^{-1},
\]
which will play the role of $R_{u}(x)$ in (\ref{eq:taylor}). Since
$d_{\infty}(u,\partial P)\ge k^{-1}$ and $f_{k}$ is convex on $P$,
$h_{u}(x)$ is convex on $\left\{ \left\Vert x\right\Vert \le k^{-1}\right\} $
and $h_{u}(x)\geq0=h_{u}(0)$. 

For $u\in A^{\mathrm{in}}_{k}$, by the convexity, we have 
\begin{align*}
k^{n}\int_{C_{k}(u)}f_{k}\mathrm{d}x-f_{k}(u) & =\frac{k^{n}}{2}\int_{C_{k}(0)}h_{u}(x)\mathrm{d}x\leq_{a}\frac{1}{2}\max_{v\in\left\{ \pm\frac{1}{2k}\right\} ^{n}}h_{u}(v)\leq\frac{1}{2}\sum_{v\in\left\{ \pm\frac{1}{2k}\right\} ^{n}}h_{u}(v)\\
 & \leq_{b}\frac{1}{4}\sum_{v\in\left\{ \pm\frac{1}{2k}\right\} ^{n}}h_{u}(2v)=\frac{1}{4}\sum_{v\in\left\{ \pm\frac{1}{k}\right\} ^{n}}h_{u}(v)
\end{align*}
where $\leq_{a}$ is because that a convex function on $C_{k}(0)$
attains its maximum on some vertex, and $\leq_{b}$ also uses convexity:
$2h_{u}(v)\leq h_{u}(0)+h_{u}(2v)$. Take the sum over $u\in A^{\mathrm{in}}_{k}$,
\begin{equation}
\sum_{u\in A^{\mathrm{in}}_{k}}\left(k^{n}\int_{C_{k}(u)}f_{k}-f_{k}(u)\right)\leq\frac{1}{4}\sum_{u\in A^{\mathrm{in}}_{k}}\sum_{v\in\left\{ \pm\frac{1}{k}\right\} ^{n}}h_{u}(v)=\frac{1}{4}\sum_{v\in\left\{ \pm\frac{1}{k}\right\} ^{n}}\sum_{u\in A^{\mathrm{in}}_{k}}h_{u}(v).\label{eq:estimate A_k}
\end{equation}
There is a lot of cancelation in the sum $\sum_{u\in A^{\mathrm{in}}_{k}}h_{u}(v)$.
Define the difference operator 
\[
\delta_{v}f_{k}(u)\coloneqq f_{k}(u)-f_{k}(u-v),
\]
then we have 
\begin{equation}
\left|\delta_{v}f_{k}(u)\right|\leq\hat{\omega}\left(\frac{1}{k}\right),\text{ for }v\in\left\{ \pm\frac{1}{k}\right\} ^{n}.\label{eq: controll diff operator}
\end{equation}
Rewrite 
\[
h_{u}(v)=\left(f_{k}(u+v)-f_{k}(u)\right)-\left(f_{k}(u)-f_{k}(u-v)\right)=\delta_{v}f_{k}(u+v)-\delta_{v}f_{k}(u).
\]
Fix a direction $v\in\left\{ \pm\frac{1}{k}\right\} ^{n}$, consider
the points $u\in A^{\mathrm{in}}_{k}$ lying on a chain $\left\{ b+lv\mid l\in\mathbb{Z}\right\} $,
$b\in B_{k}$. After cancelation, the sum 
\[
\sum_{l:\,u=b+lv\in A^{\mathrm{in}}_{k}}h_{u}(v)=\sum_{l:\,u=b+lv\in A^{\mathrm{in}}_{k}}\delta_{v}f_{k}\left(b+(l+1)v\right)-\delta_{v}f_{k}\left(b+lv\right)
\]
is reduced to the sum of two end terms, which is controlled by $2\hat{\omega}\left(k^{-1}\right)$
using (\ref{eq: controll diff operator}). Hence we have 
\[
\sum_{u\in A^{\mathrm{in}}_{k}}h_{u}(v)\leq\sum_{b\in B_{k}}\sum_{l:\,b+lv\in A^{\mathrm{in}}_{k}}h_{b+lv}\left(v\right)\leq\sum_{b\in B_{k}}2\hat{\omega}\left(\frac{1}{k}\right)\leq2\left|B_{k}\right|\hat{\omega}\left(\frac{1}{k}\right).
\]
Bring this estimate back to (\ref{eq:estimate A_k}), we conclude
\[
\frac{1}{k^{n-1}}\sum_{u\in A^{\mathrm{in}}_{k}}\left(k^{n}\int_{C_{k}(u)}f_{k}-f_{k}(u)\right)\leq\frac{2^{n}}{4k^{n-1}}2\left|B_{k}\right|\hat{\omega}\left(\frac{1}{k}\right)\rightarrow0.
\]
\end{proof}

\section{Algorithm for the shortest GKZ vectors \label{Appendix:Algorithm}}

If the point set $A$ is large, the secondary polytope $\Sigma(A)$
would be a huge monster. The number of vertices (i.e. the number of
regular triangulations) is at least exponential with respect to $\left|A\right|$.
Actually, \cite{fernandez_associahedra_2025} shows that any $n$
points in general position on the plane has at least 
\[
C_{n-2}:=\frac{1}{n+1}\binom{2n-4}{n-2}\sim4^{n}n^{-3/2}
\]
regular triangulations, where $C_{n-2}$ is Catalan number, the number
of triangulations for the regular $n$-gon. Hence, to find $\sigma_{A}$,
we need an algorithm without visiting the entire $\Sigma(A)$.

Given $\rho\in\Sigma(A)$, Lemma \ref{lem: min-max} (1) says that
$\rho=\sigma_{A}$ iff $\min_{g\in\Sigma(A)}\left\langle \rho,g\right\rangle =\left\Vert \rho\right\Vert ^{2}$.
Hence we use 
\[
\delta(\rho)\coloneqq\left\langle \rho,\rho\right\rangle -\min_{g\in\Sigma(A)}\left\langle \rho,g\right\rangle \ge0
\]
to measure the deviation of $\rho$ from $\sigma_{A}$. Actually,
we have estimate 
\[
\left\Vert \rho-\sigma_{A}\right\Vert \le\sqrt{2\delta(\rho)}.
\]
This is because 
\begin{align*}
\left\Vert \rho-\sigma_{A}\right\Vert ^{2} & =\left\Vert \rho\right\Vert ^{2}+\left\Vert \sigma_{A}\right\Vert ^{2}-2\left\langle \rho,\sigma_{A}\right\rangle \\
 & \le2\left\Vert \rho\right\Vert ^{2}-2\min_{g\in\Sigma(A)}\left\langle \rho,g\right\rangle =2\delta(\rho).
\end{align*}
The good news is that we can compute $\delta(\rho)$ without knowing
$\Sigma(A)$. Take any triangulation $T\prec S(\rho)$. By Theorem
\ref{thm: structure of Sigma A} (3), we have $\min_{g\in\Sigma(A)}\left\langle \rho,g\right\rangle =\left\langle \rho,g_{T}\right\rangle $. 

The following is the flowchart. $\mathtt{sGKZ}$ \cite{yao_program_2025}
is a computer program based on this algorithm, see $\mathtt{QUICKSTART}$
to run an example. 

$\ $

\noindent\fbox{\begin{minipage}[t]{1\columnwidth - 2\fboxsep - 2\fboxrule}%
sGKZ algorithm:

Input: marked polytope $(Q,A)$, error tolerance $\epsilon$.
\begin{description}
\item [{Step1}] an initial triangulation $T_{0}$ of $(Q,A)$. Set $\mathbb{V}=\{g_{T_{0}}\}$,
which is the set of collected GKZ vectors.
\item [{Step2}] Solve a quadratic programming problem (numerically, not
exactly): 
\begin{equation}
\rho=\mathrm{argmin}\left\{ \left|g\right|^{2}\mid g\in\mathrm{conv}(\mathbb{V})\subset\Sigma(A)\right\} .\label{eq: QP problem}
\end{equation}
\item [{Step3}] Check the deviation $\delta(\rho)$ and enlarge $\mathbb{V}$.
Take a regular triangulation $T\prec S(\rho)$. By the above discussion,
we have 
\[
\delta(\rho)\coloneqq\left\langle \rho,\rho\right\rangle -\left\langle \rho,g_{T}\right\rangle .
\]
If $\delta(\rho)<\epsilon$, jump to \textbf{Step 4}. Otherwise, add
$g_{T}$ into $\mathbb{V}$ and jump to \textbf{Step 2}. More efficiently,
we may add several $g_{T}$ with $T\prec S(\rho)$ into $\mathbb{V}$.
\item [{Step4}] Solve (\ref{eq: QP problem}) exactly over rational numbers,
its solution is $\rho_{e}$. Compute $T_{e}\prec S(\rho_{e})$ and
$\delta(\rho_{e})$ as above. If $\delta(\rho_{e})=0$, output $\sigma_{A}=\rho_{e}$.
Otherwise, add $g_{T_{e}}$ into $\mathbb{V}$ and jump to \textbf{Step
2}.
\end{description}
\end{minipage}}

\printbibliography

\end{document}